\documentclass{article}

\usepackage{vmargin}
\usepackage{epsfig}
\usepackage[francais]{babel}
\usepackage[latin1]{inputenc}
\usepackage[tbtags]{amsmath}
\usepackage{amsthm}
\usepackage{amssymb,bbm,mathrsfs}
\usepackage{array,multirow}
\usepackage{alltt}
\usepackage{pifont}
\usepackage[all]{xy}
\usepackage{url}
\usepackage{times}

\newtheorem{introtheo}{Théorème}
\newtheorem{theo}{Théorème}[section]
\newtheorem{lemme}[theo]{Lemme}
\newtheorem{prop}[theo]{Proposition}
\newtheorem{cor}[theo]{Corollaire}

\newtheorem{conj}[theo]{Conjecture}
\theoremstyle{definition}
\newtheorem{deftn}[theo]{Définition}
\newtheorem{introdef}[introtheo]{Définition}
\theoremstyle{remark}
\newtheorem{rem}[theo]{Remarque}

\def\leq{\leqslant}
\def\geq{\geqslant}

\def\1{\mathbbm{1}}
\def\A{\mathbb{A}}
\def\L{\mathbb{L}}
\def\N{\mathbb{N}}
\def\Z{\mathbb{Z}}
\def\Q{\mathbb{Q}}

\def\R{\mathbb{R}}

\def\F{\mathbb{F}}
\def\G{\mathbb{G}}
\def\O{\mathcal{O}}

\def\cro#1{\left[ #1 \right]}

\def\epsilon{\varepsilon}

\def\calA{\mathcal{A}}
\def\calC{\mathcal{C}}
\def\calL{\mathcal{L}}

\def\calX{\mathcal{X}}

\def\id{\text{\rm id}}

\def\Frac{\text{\rm Frac}\:}
\def\spec{\text{\rm Spec}\:}
\def\val{\text{\rm val}}
\def\card{\text{\rm Card}}
\def\min{\text{\rm min}}
\def\max{\text{\rm max}}
\def\lex{\text{\rm lex}}
\def\Var{\text{\rm Var}}

\def\red{\text{\rm red}}
\def\df{\text{\rm def}}
\def\mod{\,\text{\rm mod}\,}

\def\puian{{k[[u^{1/\infty}]]}}
\def\puico{{k((u^{1/\infty}))}}

\def\Leb{\text{\rm Leb}}
\def\Reg{\text{\rm Reg}}
\def\ord{\text{\rm ord}}
\def\weyl{w}
\def\weylp{\varpi}
\def\weyld{\weyl^\star}
\def\weyldd{\weyl^\vee}
\def\Rec{\text{\rm Rec}}
\def\rec{\text{\rm rec}}
\def\inv{\text{\rm inv}}

\def\vect#1{\smash{\vec{#1}}}

\def\perm#1#2#3#4{(#1\,#2\,#3\,#4)}

\title{Estimation des dimensions de certaines variétés de Kisin}
\author{Xavier Caruso}
\date{Mai 2010}

\begin{document}

\maketitle

\begin{abstract}
Dans cet article, nous nous intéressons aux dimensions de certaines 
variétés qui ont été introduites récemment par Kisin pour démontrer la 
modularité de certaines représentations galoisiennes. Nous étudions plus 
spécialement un cas particulier pour lequel nous donnons une estimation 
de la dimension en question, puis, en nous basant sur ce résultat, nous 
énonçons une conjecture dans le cas général.
\end{abstract}

\renewcommand{\abstractname}{Abstract}

\begin{abstract}
In this paper, we study dimensions of some varieties, that were 
introduced recently by Kisin in order to prove modularity of some
Galois representations.
In fact, we mainly consider a special case for which we obtain
an estimation of the dimension we are interested in. Then, based on
this result, we state a conjecture for the general case.
\end{abstract}

\setcounter{tocdepth}{2}
\tableofcontents

\bigskip

\noindent
\hrulefill

\bigskip

Motivé par l'étude de certains problèmes de modularité et poursuivant
des travaux de Breuil, Kisin a introduit et étudié dans \cite{kisin} une
certain nombre de variétés, notées $\mathscr{GR}_{V_\F}$,
$\mathscr{GR}_{V_\F,0}$ et $\mathscr{GR}_{V_\F,0}^{\mathbf v,
\text{loc}}$ dans \emph{loc. cit.}, paramétrant certains types de
schémas en groupes définis sur l'anneau des entiers d'un corps local $K$
d'inégale caractéristique $(0,p)$. Dans ce qui précède, l'indice $V_\F$
désigne une représentation du groupe de Galois absolu de $K$ à
coefficients dans un corps fini $\F$ de caractéristique $p$. S'inspirant
de cette construction, Pappas et Rapoport ont ensuite défini dans
\cite{rapoport} un champ sur $\Z_p$ dont certaines fibres s'interprètent
comme les variétés que Kisin avait définies, et en ont profité pour
nommer ces dernières \emph{variétés de Kisin}. Comprendre la géométrie
des variétés de Kisin, et notamment calculer leurs dimensions, est d'une
grande importance pour les applications. Toutefois, en dehors du cas où
$V_\F$ est de dimension $2$ considéré dans certains travaux de Kisin
(voir \cite{kisin}), Hellmann (\cite{hellmann}, \cite{hellmann2}) et 
Imai (\cite{imai1}, \cite{imai2}, \cite{imai3}), pratiquement rien n'est 
connu.
Dans cet article, nous entâmons l'étude de la dimension des variétés
de Kisin lorsque la représentation galoisienne $V_\F$ est de dimension
supérieure à $2$. Plus précisément, dans un premier temps, nous
donnons des estimations de ces dimensions dans l'exemple (déjà 
compliqué) où le corps des coefficients $\F$ est le corps premier $\F_p$ 
et où le groupe de Galois agit trivialement sur $V_\F$ et, forts de 
cela, dans un second temps, nous formulons un certain nombre de
conjectures générales qui laissent croire que le cas particulier étudié
est plutôt représentatif.

\medskip

Afin de décrire plus en détails le contenu de cet article, il est
nécessaire de commencer par rappeler la définition des variétés
de Kisin. Soit $p$ un nombre premier et $k$ un corps de caractéristique
$p$. On pose $K = k((u))$ et on appelle $\phi$ l'unique morphisme de
$k$-algèbres $\phi : K \to K$ qui est continu pour la topologie
$u$-adique et qui envoie $u$ sur $u^p$. Pour tout l'article, on fixe 
un nombre entier $d$ supérieur ou égal à $1$ et on pose $M = K^d$. Le
Frobenius $\phi$ s'étend naturellement en un opérateur sur $M$, encore
noté $\phi$, en agissant coordonnée par coordonnée\footnote{C'est le
choix de cette action particulière de $\phi$ sur $M$ qui correspond au
fait que l'on se restreint à l'action triviale de Galois sur l'espace
$V_\F$.}. Un \emph{réseau} $L$ de $M$ est, par définition, un
sous-$k[[u]]$-module $L \subset M$ engendré par une $k((u))$-base de
$M$. On note $\calL_{\leq e}$ l'ensemble des réseaux $L$ de $M$
vérifiant la condition
\begin{equation}
\label{eq:condBreuil}
u^e L \subset \phi(k[[u^{1/p}]] \otimes_{k[[u]]} L) \subset L
\end{equation}
où $\phi$ est étendu à $k[[u^{1/p}]] \otimes_{k[[u]]} L$ de façon
évidente. Kisin démontre que $\calL_{\leq e}$ apparaît naturellement
comme les $k$-points d'une variété algébrique notée $\calX_{\leq e}$.
Dans cet article, nous démontrons le théorème suivant.

\begin{introtheo}
\label{theo:dimkisin}
Avec les notations précédentes, on a :
$$\cro{\frac{d^2} 4} \cdot \cro{\frac{e-p+2}{p+1}} \leq \dim_k
\calX_{\leq e} \leq \frac{d(d-1)} 2 + \cro{\frac{d^2} 4} \cdot \frac e
{p+1}$$
où $[x]$ désigne la partie entière du réel $x$.
\end{introtheo}

\noindent
On insiste sur le fait que le théorème n'affirme en aucune façon que les
variétés $\dim_k \calX_{\leq e}$ sont équidimen\-sionnelles, ni même que
l'inégalité annoncée vaut pour toutes les composantes irréductibles. Le
nombre $\dim_k \calX_{\leq e}$ désigne bien uniquement la plus grande
dimension d'une composante irréductible.

En réalité, à la place du théorème \ref{theo:dimkisin}, nous allons 
démontrer un résultat légèrement plus général que nous énonçons 
maintenant. On se donne deux entiers $h$ et $b$ avec $b \geq 2$, et on 
remplace $\phi$ par l'application $\sigma : k((u)) \to k((u))$ donnée
par la formule suivante :
\begin{equation}
\label{eq:frobenius}
\sigma : K \to K, \quad \sum_{i \gg -\infty} a_i u^i \mapsto
\sum_{i \gg -\infty} a_i^{p^h} u^{bi}.
\end{equation}
Lorsque $h = 0$ et $b = p$, on retrouve l'opérateur $\phi$. Un autre cas 
qui semble intéressant est celui où $h = 1$ et $b = p$. En effet, un
théorème de Breuil (voir \cite{breuil}) dit alors que les éléments de
$\calX_{\leq e}(k)$ sont en bijection avec l'ensemble des classes
d'isomorphisme de modèles entiers du schéma en groupes $(\Z/p\Z)^d_K$ où
$K$ est une extension totalement ramifiée fixée de degré $e$ de $\Frac
W(k)$ (où $W(k)$ désigne l'anneau des vecteurs de Witt à coefficients
dans $k$). Le cas $b = 1$ est, à vrai dire, lui aussi très intéressant.
Nous avons préféré l'écarter dans cet article simplement car il conduit
à certaines variétés de Deligne-Lusztig affines qui ont déjà été
largement étudiées et, en particulier, dont les dimensions ont déjà été
déterminées (dans une plus grande généralité) dans les articles
\cite{gortz} et \cite{viehmann}.

\begin{introtheo}
\label{theo:dimcaruso}
Si $h \neq 0$, on a :
$$\cro{\frac{d^2} 4} \cdot \cro{\frac{e-b+2}{b+1}} \leq \dim_k \calX_{\leq e}
\leq \cro{\frac{d^2} 4} \cdot \frac e {b+1}.$$
Si $h = 0$, on a :
$$\cro{\frac{d^2} 4} \cdot \cro{\frac{e-b+2}{b+1}} \leq \dim_k \calX_{\leq e}
\leq \frac{d(d-1)} 2 + \cro{\frac{d^2} 4} \cdot \frac e {b+1}.$$
\end{introtheo}

Dans leurs articles respectifs, Kisin d'une part et Pappas et Rapoport
d'autre part définissent également des variantes des variétés
$\calX_{\leq e}$ qui ne sont plus paramétrées par un unique entier $e$
mais par un $d$-uplet d'entiers relatifs $(\mu_1, \ldots, \mu_d)$ tels
que $\mu_1 \geq \cdots \geq \mu_d$. Dans la généralité considérée ici
--- c'est-à-dire lorsque $b$ et $h$ peuvent être quelconques --- ces
variantes ont encore un sens. Précisément, si $\mu = (\mu_1, \ldots,
\mu_d)$ est un $d$-uplet comme précédemment, on peut construire des
variétés $\calX_\mu$ et $\calX_{\leq \mu}$ dont les points
$k$-rationnels sont respectivement :
$$\calL_\mu = \left\{\, 
\text{réseaux } L \text { de } M \, \left| \,
\begin{array}{c}
\text{il existe une base } m_1, \ldots, m_d \text{ de } L 
\text{ telle que} \\
u^{\mu_1} m_1, \ldots, u^{\mu_d} m_d \text{ soit une base de } 
\sigma(k[[u^{1/b}]] \otimes_{k[[u]]} L)
\end{array} \right. \, \right\}$$
et
$$\calL_{\leq \mu} = \bigcup_{\mu' \leq \mu} \calL_{\mu'}$$
où l'on convient que $\mu' = (\mu'_1, \cdots, \mu'_d)$ est plus
petit ou égal à $\mu$ si $\mu'_1 + \cdots + \mu'_t \leq \mu_1 + \cdots +
\mu_t$ pour tout $t \in \{1, \ldots, d\}$ avec égalité si $t = d$. Dans
cet article, nous nous intéressons également à la dimension de ces 
variétés. Pour énoncer les résultats obtenus, il est commode de munir 
$\R^d$ du produit scalaire usuel $\left< \cdot | \cdot \right>_d$ et
d'introduire le vecteur
$$\vec \rho = \Big( \frac {d-1} 2, \frac{d-3} 2, \ldots, \frac{1-d} 2 
\Big) \in \R^d$$
(la $i$-ième coordonnée est donnée par la formule $\frac{d+1} 2-i$). 

\begin{introdef}
\label{def:breg}
On dit qu'un $d$-uplet $\mu = (\mu_1, \ldots, \mu_d) \in \R^d$ est :
\begin{itemize}
\item \emph{$b$-régulier} si $\mu_i - \mu_{i+1} \leq b (\mu_{d-i}
- \mu_{d-i+1})$ pour tout $i \in \{1, \ldots, d-1\}$ ;
\item \emph{intégralement $b$-régulier} s'il est $b$-régulier, si
tous les $\mu_i$ sont entiers et $b-1$ divise $\mu_1 + \cdots
+ \mu_d$,
\item \emph{fortement intégralement $b$-régulier} s'il est intégralement
$b$-régulier et vérifie en plus :
$$\mu_{d-1} - \mu_d \leq b(\mu_1 - \mu_2) - d(b^2-1).$$
\end{itemize}
\end{introdef}

Les définitions d'éléments $b$-réguliers et intégralement $b$-réguliers
semblent s'imposer dans ce contexte. Par contre, l'inégalité renforcée
qui apparaît dans la définition de fortement intégralement $b$-régulier
n'est probablement pas optimale et devra sans doute être corrigée
ultérieurement. On remarque néanmoins que si $\mu = (\mu_1, \ldots,
\mu_d)$ est $b$-régulier, alors $\mu_1 \geq \cdots \geq \mu_d$.
Réciproquement si les $\mu_i$ sont rangés par ordre décroissant et
\emph{deux à deux distincts}, le $d$-uplet $\mu$ est $b$-régulier pour
$b$ suffisamment grand.

\begin{introtheo}
\label{theo:dimdivelem2}
Soit $\mu = (\mu_1, \ldots, \mu_d) \in \Z^d$ tel que $\mu_1 \geq \mu_2
\geq \cdots \geq \mu_d$. Si $b-1$ ne divise pas $\mu_1 + \cdots +
\mu_d$, alors la variété $\calX_\mu$ est vide. On suppose donc tout au
long du théorème que $b-1$ divise $\mu_1 + \cdots + \mu_d$.

On pose $\varepsilon = 1$ si $h = 0$ et $\varepsilon = 0$ dans le
cas contraire.
Alors, il existe un entier $\delta \in \{0, 1, \ldots, \varepsilon
{\cdot}\frac{d(d-1)} 2\}$ tel que l'on ait la congruence :
$$\dim_k \calX_\mu \equiv \delta - \sum_{i=1}^d i \cdot \mu_i \pmod
{b-1}.$$
En particulier, si $h \geq 0$, on a :
$$\dim_k \calX_\mu \equiv - \sum_{i=1}^d i \cdot \mu_i \pmod
{b-1}.$$

On suppose maintenant en plus $b \geq 1 + [\frac{(d-1)^2} 4]$.
Alors on a :
$$\dim_k \calX_\mu \leq \varepsilon \cdot \frac{d(d-1)} 2
+ (b-1) \cdot \min_{\weyl \in \mathfrak S_d} \sum_{i=1}^d
\sum_{n=1}^\infty \mu_i \cdot \frac{d+1-i-\weyl^n(i)}{b^n}$$
où, bien entendu, $\mathfrak S_d$ désigne le groupe des permutations de
$\{1, \ldots, d\}$ et $\weyl^n = \weyl \circ \cdots \circ \weyl$ ($n$
fois). En outre, si $\mu$ est $b$-régulier, alors le minimum précédent
est atteint pour $\weyl = \weyl_0 : i \mapsto d+1-i$ et vaut $\frac 1
{b^2-1} \cdot \left< 2 \vect \rho | \mu \right>_d$ (le produit de ce
minimum par $(b-1)$ est donc égal à $\frac 1 {b+1} \cdot \left< 2 
\vect \rho | \mu \right>_d$).

On suppose toujours $b \geq 1 + [\frac{(d-1)^2} 4]$.
Il existe des constantes positives $c_1$ et $c_2$ (qui ne dépendent que
de $d$ et $b$) telles que si les $\mu_i$ vérifient en plus $\mu_i \geq
\mu_{i+1} + c_1$ pour tout $i$, alors :
$$\dim_k \calX_\mu \geq - c_2 + (b-1) \cdot \min_{\weyl
\in \mathfrak S_d} \sum_{i=1}^d \sum_{n=1}^\infty \mu_i \cdot
\frac{d+1-i-\weyl^n(i)}{b^n}.$$
\end{introtheo}

\noindent
Encore une fois, on ne dit rien quant à l'équidimensionnalité des variétés
$\calX_\mu$. Cependant, lorsque $h \neq 0$, on peut se demander
s'il est vrai que toutes les composantes irréductibles de $\calX_\mu$
ont des dimensions congrues à $- \sum_{i=1}^d i \cdot \mu_i$ modulo
$(b-1)$. À part cela, il est clair que les sommes infinies qui apparaissent
dans la formule du théorème précédent convergent. Étant donné que toute
permutation $\weyl$ est d'ordre fini, on peut même facilement calculer
leur limite qui s'exprime toujours comme le produit de $\mu_i$
par un nombre rationnel, ce dernier étant même la valeur en $b$ d'une
fraction rationnelle à coefficients entiers. 

On en vient maintenant aux variétés $\calX_{\leq \mu}$.

\begin{introtheo}
\label{theo:dimdivelem}
Soit $\mu = (\mu_1, \ldots, \mu_d) \in \R^d$ tel que $\mu_1 \geq
\mu_2 \geq \cdots \geq \mu_d$. On pose $\varepsilon = 1$ si $h = 0$
et $\varepsilon = 0$ dans le cas contraire. Alors :
$$- (d-1)^2 - \frac {(d-2)^2} 4 +
\sup_{\substack{\mu' \leq \mu \\ \mu' \text{\rm\ f.i. } b\text{\rm -rég.}}}
\frac{\left< 2 \vect \rho | \mu' \right>_d} {b+1} 
\leq \dim_k \calX_{\leq \mu} \leq \varepsilon \cdot \frac{d(d-1)} 2 + 
\frac{\left< 2 \vect \rho | \mu \right>_d} {b+1}.$$
Si, en outre, $b \geq 1 + \max(d,[\frac {(d-1)^2} 4])$, alors la 
majoration peut être renforcée comme suit :
$$\dim_k \calX_{\leq \mu} \leq \varepsilon \cdot \frac{d(d-1)} 2
+ \sup_{\substack{\mu' \leq \mu \\ \mu'\,b\text{\rm -rég.}}}
\frac{\left< 2 \vect \rho | \mu' \right>_d} {b+1}.$$
\end{introtheo}

\medskip

\noindent
Il est utile de commenter un peu le théorème. Pour la première
assertion, on remarque que si $\mu$ est lui-même fortement intégralement
$b$-régulier, alors la borne supérieure qui apparaît est atteinte pour
$\mu' = \mu$. Ainsi le théorème dit, dans ce cas, que la quantité
$\frac{\left< 2 \vect \rho | \mu \right>_d} {b+1}$ est une bonne
approximation de la dimension de $\calX_{\leq \mu}$. La deuxième
assertion mérite, quant à elle, une discussion plus approfondie. Tout
d'abord, il est facile de prouver que la borne supérieure qui apparaît
est plus petite ou égale --- et en général strictement plus petite, du
moins si $\mu$ n'est pas lui-même $b$-régulier --- que $\frac{\left< 2
\vect \rho | \mu \right>_d} {b+1}$ ; ainsi, comme cela est déjà précisé
dans l'énoncé du théorème, la majoration écrite est meilleure que la
précédente. Par ailleurs, on a clairement $\calX_{\leq \mu} =
\bigcup_{\mu' \leq \mu} \calX_{\leq \mu'}$, d'où on déduit que :
$$\dim_k \calX_{\leq \mu} = \sup_{\mu' \leq \mu} \, \dim_k \calX_{\leq
\mu'}.$$
L'inégalité du théorème dit donc \emph{en substance} que, si $b \geq
1 + \max(d,[\frac{(d-1)^2} 4])$, les variétés $\calX_{\leq \mu'}$, pour 
$\mu' \leq \mu$ non $b$-régulier, n'apportent pratiquement pas de 
nouvelles dimensions à $\calX_{\leq \mu}$. Notamment, contrairement à 
ce qui se passe dans le cas des variétés de Deligne-Lusztig affines, il 
n'est pas clair --- et ce n'est d'ailleurs en général pas vrai --- que 
l'essentiel de la dimension de  $\calX_{\leq \mu}$ est concentré dans la 
variété $\calX_\mu$. Du fait que $\calX_\mu$ est un ouvert dans 
$\calX_{\leq \mu}$, il suit que $\calX_{\leq \mu}$ n'est généralement 
pas équidimensionnelle lorsque $\mu$ n'est pas $b$-régulier.

Finalement, la borne supérieure qui apparaît dans la dernière inégalité
du théorème est aussi égale au minimum d'un nombre fini de formes
linéaires sur $\R^d$, ce qui permet de la calculer efficacement. Malgré
tout, bien que ces formes linéaires soient définies de façon plutôt
explicite, leur nombre et leur complexité croît très rapidement lorsque
$d$ augmente. À titre d'exemple, le tableau \ref{tab:coordextr} (page
\pageref{tab:coordextr}) les donne pour $d \leq 4$.

\medskip

Il est vrai que les théorèmes précédents peuvent sembler ni vraiment
intéressants, ni faciles à appliquer car ils énoncent finalement, sous
des hypothèses plutôt fortes, des résultats à la fois techniques et
imprécis. Ce point de vue convient toutefois d'être nuancé (voire
reconsidéré) pour plusieurs raisons. 

Tout d'abord, en ce qui concerne les hypothèses, il ne faut pas voir les 
résultats de cet article comme une fin en soi, mais bel et bien comme un 
premier pas vers la résolution d'un problème plus général. Mieux encore, 
le théorème \ref{theo:dimdivelem2} semble directement donner les clés de 
cette vaste généralisation. Généralisation tout d'abord au cas d'un 
opérateur $\sigma : M \to M$ n'agissant pas nécessairement coordonnée 
par coordonnée (c'est-à-dire d'une représentation $V_\F$ quelconque avec 
encore $\F = \F_p$) pour lequel l'auteur pense que le théorème 
\ref{theo:dimdivelem2} s'étend simplement en modifiant certaines 
constantes (voir conjecture \ref{conj:anyfrob}, page 
\pageref{conj:anyfrob}, pour plus de précisions). Mais généralisation 
également au cas des variétés de Kisin associés à un groupe réductif 
connexe déployé quelconque (le cas présenté ici est celui de 
$\text{GL}_d$), ce qui englobe notamment le cas des variétés de Kisin 
associées à des représentations $V_\F$ à coefficients dans une extension 
finie arbitraire de $\F_p$. Pour un énoncé précis dans cette direction, 
on se contente de renvoyer le lecteur au \S \ref{subsec:reductif}, et 
plus particulièrement à la conjecture \ref{conj:reductif}. De surcroît, 
l'auteur pense que les méthodes développées dans cet article sont de 
nature à s'étendre à la situation générale des groupes réductifs, et y 
reviendra sans doute dans un travail ultérieur.

Au sujet, ensuite, de la technicité des résultats, il est à noter que, 
si l'on se restreint à des $\mu$ intégralement $b$-réguliers, tous les 
théorèmes de cet article deviennent particulièrement simples 
puisqu'alors, à une constante près, tous les minorants et majorants sont 
égaux à $\frac 1 {b+1} \cdot \left< 2 \vect \rho | \mu \right>_d$. Il 
est vrai, enfin, que le problème de l'imprécision reste, quant à lui, 
non résolu même conjecturalement. Il est toutefois intéressant de 
comparer la forme générale des formules du théorème 
\ref{theo:dimdivelem2} avec les formules connues pour les dimensions des 
variétés de Deligne-Lusztig, démontrées dans \cite{gortz} et 
\cite{viehmann}. Un rapide coup d'\oe il à ces références montre que 
cette dimension s'écrit comme le somme d'une contribution linéaire (qui 
s'exprime en terme de produit scalaire avec le vecteur $\rho$ comme dans 
cet article) et d'une contribution bornée. Il semble donc que l'on ait 
découvert, ici, l'analogue de la partie linéaire et, en ce sens, le 
théorème \ref{theo:dimdivelem2} apparaît à nouveau comme un premier pas 
incontournable pour le calcul de la dimension des variétés de Kisin en 
toute généralité.

\paragraph{Présentation sommaire de la méthode et du plan de l'article}

De façon générale, la méthode suivie est largement inspirée de l'article 
\cite{viehmann} de Viehmann : l'idée est de définir une stratification 
plus fine des variétés de Kisin pour laquelle on sait calculer 
précisément la dimension des strates. Le problème du calcul de la 
dimension des variétés de Kisin se métamorphose alors complètement en un 
nouveau problème combinatoire que l'on parvient à résoudre ensuite, au 
moins de façon approchée. Si la première partie du cheminement suit 
d'assez près les arguments de Viehmann, les chemins se séparent 
nettement pour la résolution du problème combinatoire qui s'avère être 
bien plus délicat dans le cas des variétés de Kisin.

De façon plus détaillée, on commence dans le \S \ref{subsec:defvarphi}
par associer à chaque réseau $L \subset M$ une donnée combinatoire
$\varphi(L)$ constituée de $d$ fonctions. Il résulte des travaux de
Viehmann que ces données combinatoires sont soumises à des nombreuses
contraintes qui imposent une forte rigidité. 
Dans le \S \ref{subsec:parametrisation}, on étudie plus en détails ces
contraintes, et on en déduit une paramétrisation des \og
données combinatoires admissibles \fg\ par les points d'un réseau à
l'intérieur d'un convexe vivant dans un espace vectoriel réel de
dimension $\frac{d(d+1)} 2$. D'un point de vue géométrique, la
construction précédente définit une stratification des variétés de Kisin
par des sous-espaces localement fermés $\calX_\varphi$. On démontre
alors le théorème \ref{theo:dim} qui donne une estimation (et même une
formule exacte dans certains cas) pour la dimension des variétés
$\calX_\varphi$, qui s'exprime de façon complètement explicite en
fonction du point du réseau paramétrant $\varphi$.

À ce stable, le problème du calcul des variétés de Kisin se reformule 
complètement en termes de programmation linéaire. Dans le \S 
\ref{sec:methode}, on introduit les outils nécessaires (qui sont plus ou 
moins classiques) à sa résolution et, à titre d'exemple, on fait 
fonctionner la méthode dans un cas simple, aboutissant ainsi à une 
démonstration de la majoration du théorème \ref{theo:dimcaruso} sous 
l'hypothèse supplémentaire $b > d$. Dans le \S \ref{sec:dimkisin}, la 
machine se met enfin véritablement en route avec pour objectif de 
démontrer complètement les théorèmes \ref{theo:dimcaruso}, 
\ref{theo:dimdivelem2} et \ref{theo:dimdivelem}. Cela devient alors 
assez vite très technique et il n'est pas vraiment envisageable d'en 
dire beaucoup plus dans cette introduction, sauf peut-être que 
l'essentiel de la démonstration consiste à donner des descriptions 
précises de certaines parties convexes de $\R^d$, et ne fait intervenir 
que des arguments élémentaires (mais parfois subtils).

Dans le \S \ref{sec:conjectures} finalement sont examinées plusieurs 
perspectives offertes par les résultats des théorèmes 
\ref{theo:dimcaruso}, \ref{theo:dimdivelem2} et \ref{theo:dimdivelem}, 
et notamment les généralisations éventuelles à des $\sigma$ agissant sur 
$M$ pas nécessairement coordonnée par coordonnée (voir \S 
\ref{subsec:frobarbit}) et d'autres groupes réductifs (voir \S 
\ref{subsec:reductif}) qui ont été évoquées précédemment. On discute 
également, dans le \S \ref{subsec:dimexacte}, de la possibilité 
d'obtenir une formule exacte pour la dimension. À part un calcul 
explcite en dimension $3$ dont la conclusion reste mystérieuse, il faut 
bien dire que rien de vraiment précis ne se dégage pour l'instant.

\paragraph{Remerciements}

C'est un plaisir de remercier Michael Rapoport pour m'avoir soumis le
problème dont il est question dans cet article, et également pour ses
encouragements constants. Je le remercie également de m'avoir plusieurs
fois invité à l'Université de Bonn, qui est un lieu extraordinaire pour
faire des mathéma\-tiques. Mes remerciements vont aussi à Eugen Hellmann
pour d'intéressantes discussions et à Eva Viehmann pour m'avoir fait
connaître son article \cite{viehmann}, duquel tout ce travail est
inspiré. Je remercie également David Monniaux pour m'avoir fait
connaître son logiciel {\tt mjollnir} \cite{monniaux} qui m'a été fort
utile lors de l'élaboration de cet article.

Finalement, je remercie l'Agence Nationale de la Recherche (ANR) pour 
son soutien financier par l'intermédiaire du projet CETHop (Calculs 
Effectifs en Théorie de Hodge $p$-adique) référencé ANR-09-JCJC-0048-01.

\numberwithin{equation}{section}

\section{Une stratification utile}
\label{sec:stratification}

On commence par rappeler très brièvement que les variétés $\calX_{\leq 
e}$, $\calX_\mu$ et $\calX_{\leq \mu}$ sont définies par l'intermédiaire 
de leur \og foncteur des points \fg. Pour $\calX_{\leq e}$ par exemple, 
étant donné $R$ une $k$-algèbre, on définit l'ensemble $\calX_{\leq
e}(R)$ comme l'ensemble des $R[[u]]$-sous-modules $L$ de $M 
\otimes_{k((u))} R((u)) = R((u))^d$ qui sont tels que :
\begin{itemize}
\item[i)] $L$ est un réseau de $M \otimes_{k((u))} R((u)) = R((u))^d$,
c'est-à-dire que $L$ est un $R[[u]]$-module localement libre de
type fini (pour la topologie de $\spec R$) et le morphisme naturel $L
\otimes_{R[[u]]} R((u)) \to M$ est un isomorphisme ;
\item[ii)] la condition \eqref{eq:condBreuil} est satisfaite où $\phi$
est remplacé par $\sigma$ (et où $\sigma$ opère encore par $x \mapsto
x^{p^h}$ sur $R$, envoie $u$ sur $u^b$ et s'étend à $M \otimes_{k((u))}
R((u)) = R((u))^d$ en agissant coordonnée par coordonnée).
\end{itemize}
On montre ensuite que le foncteur $R \mapsto \calX_{\leq e} (R)$ ainsi 
défini est représentable par un schéma de type fini sur $k$ qui apparaît
naturellement comme un sous-schéma fermé de la grassmanienne affine sur
$k$. Des considérations analogues conduisent à une définition rigoureuse 
de $\calX_\mu$ et $\calX_{\leq \mu}$.

Le but de ce premier chapitre est de montrer que les variétés 
précédentes sont stratifiées par des sous-variétés $\calX_\varphi$ (où 
$\varphi$ est une donnée combinatoire qui dépend de $\frac{d(d+1)}2$ 
entiers) dont on sait calculer la dimension.

\subsection{Donnée combinatoire associée à un réseau}

Dans cette partie, on associe à chaque réseau $L$ de $M$ la donnée 
combinatoire évoquée précédemment qui s'avère être un ensemble de $d$ 
fonctions soumises à un certain nombre de contraintes. On montre ensuite 
que ces contraintes imposent une rigidité telle qu'elles réduisent la 
donnée des $d$ fonctions à celle de seulement $\frac{d(d+1)}2$ nombres.

\subsubsection{Définitions}
\label{subsec:defvarphi}

Soit $\val$ la valuation naturelle sur $k((u))$ : la valuation d'une 
somme $\sum_{i \gg -\infty} a_i u^i$ est le plus petit entier $v$ tel 
que $a_v \neq 0$ et on convient que $\val(0) = +\infty$. La valuation 
s'étend de manière unique à l'extension totalement ramifiée 
$k((u^{1/b}))$, et on note encore $\val$ ce prolongement ; on a donc 
$\val(u^{1/b}) = \frac 1 b$.

On pose $M_{k((u^{1/b}))} = k((u^{1/b})) \otimes_{k((u))} M =
k((u^{1/b}))^d$ et on note $(e_1, \ldots, e_d)$ la base canonique de $M
= k((u))^d$. Les vecteurs $1 \otimes e_i$ forment une base de
$M_{k((u^{1/b}))}$ sur le corps $k((u^{1/b}))$. La valuation $\val$
définit une application $\val_M : M_{k((u^{1/b}))} \backslash \{0\} \to
\frac 1 b \Z \times \{1, \ldots, d\}$ par la formule :
\begin{eqnarray*}
\val_M(x_1, \ldots, x_d) = (v,i) & \text{où} &
v = \min \, \{ \val(x_1), \ldots, \val(x_d) \} \\
& \text{et} & i = \min \, \{ j \, | \, v = \val(x_j) \}
\end{eqnarray*}
On prolonge $\val_M$ à $M_{k((u^{1/b}))}$ tout entier en convenant que 
$\val_M(0) = \infty$ où le symbole $\infty$ désigne un nouvel élément 
que l'on ajoute au produit $\frac 1 b \Z \times \{1, \ldots, d\}$. On 
vérifie immédiatement que si $\lambda \in k((u^{1/b}))$ et $x \in 
M_{k((u^{1/b}))}$, on a $\val_M(\lambda x) = \val(\lambda) + \val_M(x)$ 
avec la convention évidente que $\infty + t = \infty$ lorsque $t$ est un 
nombre rationnel ou un couple $(v,i) \in \frac 1 b \Z \times \{1, 
\ldots, d\}$. De plus, si l'on munit l'ensemble $\frac 1 b \Z \times 
\{1, \ldots, d\}$ de l'ordre lexicographique\footnote{Cela signifie que 
$(v,i) < (v',i')$ si, et seulement si soit $v < v'$, soit $v = v'$ et $i 
< i'$.} et que l'on convient que $\infty$ est strictement plus grand que 
tous les couples $(v,i)$, alors, pour tous $x$ et $y$ dans $M$, on a 
$\val_M(x+y) \geq \min \{ \val_M(x), \val_M(y) \}$ et l'égalité a lieu 
dès que $\val_M(x) \neq \val_M(y)$.

\begin{deftn}
\label{def:varphitilde}
Soit $L$ un réseau de $M$. Pour tout $v \in \frac 1 b \Z$ et tout
$i \in \{1, \ldots, d\}$, on pose
\begin{equation}
\label{eq:defvarphi}
\tilde \varphi_i(L)(v) = \sup_{\substack{x \in k[[u^{1/b}]]
\otimes_{k[[u]]} L \\ \val_M(x) = (v,i)}} \Big( \sup \big\{\, n \in \Z
\,\,|\,\, \sigma(x) \in u^n L \, \big\} \Big)
\end{equation}
où, par convention, la borne supérieure d'un ensemble non majoré est
$+\infty$ et celle de l'ensemble vide est $-\infty$.
\end{deftn}

La définition ci-dessus n'est en fait rien d'autre qu'une adaptation de
la définition des fonctions $\varphi$ qui apparaissent dans
\cite{viehmann}. On a choisi de conserver la formulation de \emph{loc. 
cit.} mais il est sans doute bon de garder à l'esprit que l'on peut
interpréter les nombres $\tilde \varphi_i(L)(v)$ de façon plus parlante
en termes de distance. Pour ce faire, on munit l'espace $k((u))$ de la
norme $||x|| = a^{\val(x)}$ pour un certain réel $a$ fixé dans $]0,1[$. 
Le choix du réseau $L$ définit une norme $|| \cdot ||_L$ sur $M$ comme
suit : si $(f_1, \ldots, f_d)$ est une $k[[u]]$-base de $L$, on pose
$||\sum_{i=1}^d \lambda_i f_i||_L = \max ||\lambda_i||$. La norme
obtenue ne dépend alors pas du choix des $f_i$. Par ailleurs, un examen
direct des définitions montre que, pour la distance associée à
$||\cdot||_L$, le nombre réel $a^{\tilde \varphi_i(L)(v)}$ est égal à la
distance de l'origine au sous espace $\sigma(B)$ où $B$ est l'ensemble
des éléments de $k[[u^{1/b}]] \otimes_{k[[u]]} L$ de valuation $(v,i)$
(ensemble que l'on peut voir aussi comme la boule unité pour une autre
norme). Comme $0 \not\in \sigma(B)$ (puisque $\sigma$ est injectif et
que $0$ n'est pas dans $B$), il suit de la description précédente que la
fonction $\tilde \varphi_i(L)$ ne prend jamais la valeur $+ \infty$.

\begin{prop}
\label{prop:varphi}
Soit $L$ un réseau de $M$. Les fonctions $\tilde \varphi_i(L) : \frac 1
b \Z \to \Z \cup \{-\infty\}, v \mapsto \tilde \varphi_i(L)(v)$ 
vérifient les propriétés suivantes.
\begin{enumerate}
\item Pour tout $i$, la fonction $\tilde \varphi_i(L)$ est strictement
croissante où, par un léger abus d'écriture, l'on entend par
là\footnote{Et ce sera aussi le cas dans tout l'article.} que $\tilde
\varphi_i(L)$ est croissante et qu'elle est strictement croissante sur
l'ensemble où elle prend des valeurs finies.
\item Pour tout $i$, il existe un entier $\tilde q_i(L)$ tel que 
\begin{itemize}
\item la fonction $\tilde \varphi_i(L)$ prend des valeurs finies 
exactement sur l'intervalle $[\tilde q_i(L), +\infty[$, et
\item pour $v$ suffisamment grand, on a $\tilde \varphi_i(L)(v) = bv - 
\tilde q_i(L)$.
\end{itemize}
\item Pour $j \in \{1, \ldots, d\}$, il existe des fonctions croissantes
$\psi_j(L) : \Z \to \frac 1 b \Z \cup\{-\infty\}$ telles que $\psi_1(L)
\leq \psi_2(L) \leq \cdots \leq \psi_d(L)$ et pour tout couple $(v,\mu)
\in \frac 1 b \Z \times \Z$, il y a autant d'indices $i \in \{1, \ldots,
d\}$ tels que $\mu = \tilde \varphi_i(L)(v)$ que d'indices $j \in \{ 1,
\ldots, d\}$ tels que $v = \psi_j(L)(\mu)$.

Ces fonctions sont en outre uniquement déterminées. 

\item Si $u^{\mu_1(L)}, \ldots, u^{\mu_d(L)}$, avec $\mu_1(L) \geq
\cdots \geq \mu_d(L)$, sont les diviseurs élémentaires du
$k[[u]]$-module engendré par $\sigma(L)$ par rapport à $L$, alors, pour
tout $i$, la fonction $\psi_i$ prend des valeurs finies exactement sur
l'intervalle $[\mu_i(L), +\infty[$.
\end{enumerate}
\end{prop}

\begin{proof}
C'est simplement une transposition du lemme 4.1 de \cite{viehmann}.
\end{proof}

Si $L$ est un réseau de $M$, on définit aussi des fonctions 
$\varphi_1(L), \ldots, \varphi_d(L) : \frac 1 b \Z \to \Z \cup \{ 
-\infty\}$ en convenant que pour tout $v \in \frac 1 b \Z$, les nombres 
$\varphi_1(L)(v), \ldots, \varphi_d(L)(v)$ sont les mêmes que $\tilde 
\varphi_1(L)(v), \ldots, \tilde \varphi_d(L)(v)$ mais triés par ordre 
décroissant. Les fonctions précédentes vérifient donc tautologiquement 
l'inégalité $\varphi_1(L) \geq \varphi_2(L) \geq \cdots \geq 
\varphi_d(L)$ et on montre sans peine qu'elles satisfont encore aux 
quatre alinéas de la proposition \ref{prop:varphi} : les entiers 
$\mu_j(L)$ restent inchangés tandis que les $\tilde q_i(L)$ sont \emph{a 
priori} permutés. Dans la suite, l'entier correspondant à la fonction 
$\varphi_i(L)$ sera noté $q_i(L)$.

\subsubsection{Un exemple en dimension $2$}

En guise d'illustration de la proposition \ref{prop:varphi} et pour 
familiariser le lecteur avec la définition \ref{def:varphitilde}, on 
examine un exemple en dimension $2$ (\emph{i.e.} avec $d = 2$). Si $L$ 
est un réseau de $M$, il existe des entiers relatifs $\alpha, \delta$ et 
un élément $c \in k((u))$ de valuation $\gamma$ tels que $L$ soit 
engendré sur $k[[u]]$ par les vecteurs $f_1 = (u^\alpha, 0)$ et $f_2 = 
(c, u^\delta)$. De plus, quitte à retirer à $c$ un multiple de 
$u^\alpha$, on peut supposer que soit $c = 0$, soit $\gamma < \alpha$. 

\begin{lemme}
\label{lem:ex2}
Soit $x = (x_1, x_2) \in k((u^{1/b})) \otimes_{k((u))} M$.
Le plus grand $n$ tel que $x \in u^n k[[u^{1/b}]]
\otimes_{k[[u]]} L$ est le plus petit des deux nombres $\val(x_2) -
\delta$ et $\val(x_1 - x_2 c u^{-\delta}) - \alpha$.
\end{lemme}

\begin{proof}
Le vecteur $(x_1, x_2)$ se décompose sur la base $(f_1, f_2)$ sous la 
forme :
$$(x_1, x_2) = (x_1 u^{-\alpha} - x_2 c u^{-(\alpha + \delta)}) f_1
+ x_2 u^{-\delta} f_2.$$
Ainsi il appartient à $M$ si, et seulement si les deux coefficients que
l'on voit apparaître dans l'écriture précédente sont de valuation
positive ou nulle, c'est-à-dire si, et seulement si $\val(x_1 - x_2 c
u^{-\delta}) \geq \alpha$ et $\val(x_2) \geq \delta$. Le lemme en
découle.
\end{proof}

À partir de maintenant, on suppose que $c = u^\gamma$ avec $\gamma <
\delta < \alpha$ (les autres cas se traitent au moyen de calculs
analogues).
On détermine tout d'abord la fonction $\tilde \varphi_1(L)$. Soient $v 
\in \frac 1 b \Z$ et $x \in k[[u^{1/b}]] \otimes_{k[[u]]} L$ tel que 
$\val(x) = (v,1)$. Quitte à multiplier $x$ par un élément inversible de 
$k[[u^{1/b}]]$, ce qui ne change pas la valeur de la deuxième borne 
inférieure dans la formule \eqref{eq:defvarphi}, on peut supposer que 
$x$ s'écrit $u^v \otimes e_1 + u^v y \otimes e_2$ où $y \in 
k[[u^{1/b}]]$. Par le lemme \ref{lem:ex2}, le fait que $x$ appartienne 
à $L$ se traduit par les inégalités $\val(u^v y) \geq \delta$ et 
$\val(u^v - u^{v + \gamma - \delta} y) \geq \alpha$, ce qui se réécrit 
encore :
\begin{equation}
\label{eq:condy}
\val(y) \geq \delta - v \quad \text{et} \quad
y \equiv u^{\delta-\gamma} \pmod {u^{\delta-\gamma + (\alpha - v)}}
\end{equation}
On suppose pour commencer que $v < \alpha$. Dans ce cas, la dernière
congruence implique que $y$ est de valuation $\delta-\gamma$, d'où on
déduit $\delta-\gamma \geq \delta-v$, c'est-à-dire $v \geq \gamma$.
Autrement dit, si $v < \gamma$, aucun $x$ ne satisfait aux conditions
requises et on a alors $\tilde \varphi_1(L)(v) = -\infty$. Si, au
contraire, $\gamma \leq v < \alpha$, il existe des $x$ convenables qui
sont précisément les vecteurs de la forme
$$x = u^v \otimes e_1 + u^{v + \delta-\gamma} ( 1 + u^{\alpha - v} z)
\otimes e_2$$
pour un certain élément $z \in k[[u^{1/b}]]$. Le lemme \ref{lem:ex2}
appliqué au vecteur $\sigma(x)$ donne ainsi :
\begin{eqnarray*}
\tilde \varphi_1(L) (v) & = & 
\sup_{z \in k[[u^{1/b}]]} \big( \min \{ \, b(v +
\delta-\gamma) - \delta, \, bv - \alpha + \val( 1 - u^{(b-1) (\delta -
\gamma)} (1 + u^{b(\alpha - v)} \sigma(z))) \, \} \big) \\
& = & \sup_{z \in k[[u^{1/b}]]} \big( 
\min \{ \, b(v + \delta-\gamma) - \delta, \, bv - \alpha \, \} \big) 
\qquad \text{car } \delta - \gamma > 0 \\
& = & \min \{ \, b(v + \delta-\gamma) - \delta, \, bv - \alpha \, \}
\end{eqnarray*}
Or, $b(\delta - \gamma) > 0 > \delta - \alpha$, d'où on obtient 
finalement $\tilde \varphi_1(L)(v) = bv - \alpha$ pour $v \in [\gamma,
\alpha[$.

On suppose désormais que $v \geq \alpha$. Dans ce cas, la première 
condition
de la ligne \eqref{eq:condy} est automatique vérifiée (on rappelle que
$y$ est dans $k[[u^{1/b}]]$) tandis que la congruence qui suit se réduit
à $\val(y) \geq \delta-\gamma + \alpha - v$. On est ainsi amené à
calculer
$$\sup_y \big( \min \{ \, b(v + \val(y)) - \delta, \, bv - \alpha + \val(
1 - u^{\gamma - \delta} \sigma(y)) \, \} \big)$$
où la borne supérieure est prise sur tous les $y$ dont la valuation est
à la fois supérieure ou égale à $0$ et à $\delta-\gamma + \alpha - v$.
Pour les $y$ de valuation strictement plus petite (resp. strictement
plus grande) que $\frac{\delta - \gamma} b$, la valuation de la
différence $1 - u^{\gamma-\delta} \sigma(y)$ vaut $\gamma-\delta + b
\val(y)$ (resp. $0$), et un calcul simple montre que le minimum qui
apparaît dans la formule précédente est inférieur ou égal (resp. est
égal) à $bv - \alpha$. Par contre, pour les $y$ qui ont la valuation
critique $\frac{\delta - \gamma} b$, la borne supérieure est atteinte 
lorsque $1 - u^{\gamma - \delta} \sigma(y)$ s'annule, et elle est
égale à $b(v + \val(y)) - \delta = bv - \gamma$. Du fait que $\gamma <
\alpha$, on déduit que la valeur de $\tilde \varphi_1(L)(v)$ est $bv -
\gamma$ dès que l'on peut choisir $y$ de valuation $\frac{\delta -
\gamma} b$, c'est-à-dire dès que $\frac{\delta - \gamma} b \geq \delta -
\gamma + \alpha - v$ ou encore après simplification $v \geq \alpha +
\frac{b-1}b (\delta - \gamma)$, et qu'elle est $bv - \alpha$ dans le cas
contraire.

En résumé, la fonction $\tilde \varphi_1(L)$ prend la forme simple 
suivante :
$$\begin{array}{rrcll}
\tilde \varphi_1(L)\,:& v & \mapsto & -\infty & \text{si } v < \gamma \\
& v & \mapsto & bv - \alpha & \text{si } \gamma \leq v < \ \alpha + 
\frac{b-1}b (\delta - \gamma) \\
& v & \mapsto & bv - \gamma & \text{sinon}.
\end{array}$$

On en vient à la fonction $\tilde \varphi_2(L)$. Soit $x \in L$
tel que $\val_M(x) = (v,2)$. Comme dans le cas
précédent, quitte à multiplier $x$ par un élément inversible de
$k[[u^{1/b}]]$, on peut supposer qu'il est de la forme $x = u^v y \otimes e_1
+ u^v \otimes e_2$ où $y \in k((u^{1/b}))$ est un élément de valuation
strictement positive. Comme l'on sait que $x$ est élément de $L$, le 
lemme \ref{lem:ex2} implique que $\val(y - u^{\gamma - \delta}) \geq 
\alpha -  v$, soit encore $\gamma - \delta \geq \alpha - v$ car $\gamma 
- \delta < 0 < \val(y)$. Ainsi la valuation de la différence $y - 
u^{\gamma - \delta}$ est $\gamma - \delta$. En particulier, si $v < 
\alpha + \delta - \gamma$, aucun  élément $y$ ne
satisfait aux conditions requises, et donc $\tilde \varphi_2(L)(v) =
-\infty$. Par contre, si $v \geq \alpha + \delta - \gamma$, tous les $y$ 
de valuation strictement positive conviennent et un calcul analogue à celui
mené pour $\tilde \varphi_1(L)$ conduit dans ce cas à $\tilde \varphi_2
(L) (v) = bv - \alpha + \gamma - \delta$. En résumé, on a donc :
$$\begin{array}{rrcll}
\tilde \varphi_2(L)\,:& v & \mapsto & -\infty & \text{si } v < \alpha +
\delta - \gamma \\
& v & \mapsto & bv - \alpha + \gamma - \delta & \text{sinon}.
\end{array}$$
On observe sans difficulté que $\tilde \varphi_1(L)(v) \geq \tilde
\varphi_2(L)(v)$ pour tout $v \in \frac 1 b \Z$ ; ainsi 
$\varphi_1(L) = \tilde \varphi_1(L)$ et $\varphi_2(L) = \tilde \varphi_2
(L)$. Il est maintenant facile de vérifier la proposition \ref{prop:varphi}
sur cet exemple et en particulier de décrire les fonctions $\psi_1(L)$
et $\psi_2(L)$. On trouve :
$$\begin{array}{rrcll}
\psi_1(L)\,:& \mu & \mapsto & -\infty & \text{si } \mu < b(\alpha +
\delta - \gamma) - \delta \\
& \mu & \mapsto & b^{-1} (\mu + \gamma) & \text{sinon} \\ \\
\psi_2(L)\,:& \mu & \mapsto & -\infty & \text{si } \mu < b\gamma -
\alpha \\
& \mu & \mapsto & b^{-1} (\mu + \alpha) & \text{si } b\gamma -
\alpha \leq \mu < (b-1)(\alpha + \delta - \gamma) \\
& \mu & \mapsto & b^{-1} (\mu + \alpha + \delta - \gamma) & \text{sinon.}
\end{array}$$
Les entiers $\mu_1(L)$ et $\mu_2(L)$ valent donc respectivement
$b(\alpha + \delta - \gamma) - \delta$ et $b\gamma - \alpha$, et l'on
peut vérifier par un calcul indépendant que ce sont bien les exposants
des diviseurs élémentaires du $k[[u]]$-module engendré par
$\sigma(L)$ par rapport à $L$.

\subsubsection{Prolongement des fonctions $\varphi_i$}

Pour la suite, il sera commode, afin de mieux visualiser les fonctions 
$\varphi_i(L)$ de les prolonger à tout $\R$ en posant 
$$\varphi_i(q) = \varphi_i(v) + b(q-v) \quad \textstyle \text{pour tout 
} q \in [v, v + \frac 1 b[$$
avec la convention que $-\infty + x = -\infty$ pour tout nombre réel 
$x$. On prolonge de la même façon les fonctions $\tilde \varphi_i(L)$. 
Il est alors clair que pour tout réel $q$ les nombres $\varphi_1(L)(q), 
\ldots, \varphi_d(L)(q)$ sont les mêmes que $\tilde \varphi_1(L)(q), 
\ldots, \tilde \varphi_d(L)(q)$ mais triés par ordre décroissant. Il est 
également évident que les fonctions $\tilde \varphi_i(L)$ et 
$\varphi_i(L)$ sont continues à droite, affines par morceaux, et que 
leurs dérivées valent $b$ partout où elles sont définies. Les fonctions 
$\psi_j(L)$ données par la proposition \ref{prop:varphi} se prolongent 
elles aussi à $\R$ tout entier en convenant qu'elles valent $-\infty$ 
sur l'intervalle $]-\infty, \mu_j(L)[$ et qu'elles sont affines de pente 
$\frac 1 b$ sur tout intervalle de la forme $[\mu, \mu+1[$ où $\mu$ est 
un nombre entier supérieur ou égal à $\mu_j(L)$. Ces fonctions ainsi 
prolongées vérifient encore la condition de l'alinéa 3 de la proposition 
\ref{prop:varphi} lorsque $v$ et $\mu$ sont des éléments de $\R$.

\medskip

La méthode que l'on a employée pour prolonger les fonctions $\varphi_i$
et $\psi_j$ à tout $\R$ peut sembler artificielle, mais en fait il n'en
est rien comme le montre la proposition suivante.

\begin{prop}
On note $\puian = \bigcup_{n \geq 1} k[[u^{1/n}]]$ (resp. $\puico =
\bigcup_{n \geq 1} k((u^{1/n}))$) l'anneau (resp. le corps) des séries
de Puiseux à coefficients dans $k$. Alors pour tout $i \in \{1, \ldots,
d\}$ et tout nombre rationnel $q$, on a :
$$\tilde \varphi_i(L)(q) = \sup_{\substack{x \in \puian
\otimes_{k[[u]]} L \\ \val_{M,\Q}(x) = (q,i)}} \Big( \sup \big\{\, n \in \Q
\,\,|\,\, \sigma(x) \in u^n (\puian \otimes_{k[[u]]} L) \, \big\} \Big)$$
où $\val_{M,\Q}$ désigne le prolongement naturel de $\val_M$ à $\puico 
\otimes_{k((u))} M$.
\end{prop}

\begin{rem}
En d'autres termes, la proposition dit que la fonction
$$q \mapsto \sup_{\substack{x \in \puian \otimes_{k[[u]]} L \\ \val_{M,\Q}
(x) = (q,i)}} \Big( \sup \big\{\, n \in \Q \,\,|\,\, \sigma(x) \in u^n (\puian
\otimes_{k[[u]]} L) \, \big\} \Big)$$
est entièrement déterminée par ses valeurs sur l'ensemble $\frac 1 b \Z$
et que, de surcroît, pour calculer les valeurs en ces points, on peut se
contenter d'étendre les scalaires à $k((u^{1/b}))$ (sans aller, donc,
jusqu'aux séries de Puiseux).
\end{rem}

\begin{proof}
Étant donné que la proposition ne sera pas utilisée dans la suite,
on se contente de donner quelques indications sur sa preuve.
L'idée directrice est de comprendre comment l'on peut calculer de façon
algorithmique la borne supérieure qui apparaît dans l'énoncé de la
proposition, notée $f_i(L) (q)$ dans le restant de la
preuve. Dans la suite, on notera également $(e_1, \ldots, e_d)$ la base
canonique de $M$.
L'entier $i$ restant fixé, on explique tout d'abord comment calculer
la borne inférieure des $q$ tels qu'il existe dans $\puian
\otimes_{k[[u]]} L$ des éléments de valuation $(q,i)$. Soit $q_i$ cette 
borne inférieure. Dire que $\puian
\otimes_{k[[u]]} L$ contient un élément de valuation $(q,i)$ signifie
exactement qu'il existe un nombre rationnel $\varepsilon > 0$ tel que 
$$u^q e_i \in L + \sum_{j=1}^i u^{q+\epsilon} e_j k[[u^{1/\infty}]] + 
\sum_{j=i+1}^d u^q e_j k[[u^{1/\infty}]].$$
Si l'on fixe $m_1, \ldots, m_d$ une base de $L$ et que l'on note 
$M_i$ (resp. $E_i$) le vecteur colonne des coordonnées de $m_i$
(resp. $e_i$) dans la base canonique, la condition précédente 
signifie que le vecteur colonne $u^q E_i$ est dans l'image de la
matrice par blocs
$$\big(\,M_1 \,\big|\, M_2 \,\big|\, \cdots \big|\, M_d \,\big|\, 
u^{q+\epsilon} E_1 \,\big|\, \cdots \,\big|\, u^{q+\epsilon} E_i 
\,\big|\, u^q E_{i+1} \,\big|\, \cdots \,\big|\, u^q E_d \,\big),$$
c'est-à-dire dans le module engendré par les vecteurs colonne de cette 
matrice.
On peut à présent effectuer des opérations sur les colonnes de la matrice précédente
(ce qui ne modifie pas son image) pour se ramener à permutation des
lignes près à une matrice de la forme 
$$\left( \raisebox{0.5\depth}{\xymatrix @R=5pt @C=1pt {
u^{n_1} \ar@{.}[rrdd] & 0 \ar@{.}[rrrr] \ar@{.}[rrdd] 
& & & & 0 \ar@{.}[dd] \\
\star \ar@{.}[rd] \ar@{.}[d] \\
\star \ar@{.}[r] & \star & u^{n_d} & 0 \ar@{.}[rr] & & 0 
}}\right)$$
pour certains nombres rationnels $n_j$ ($1 \leq j \leq d$) rangés par
ordre croissant. En outre, un examen de l'algorithme classique de calcul
de la forme précédente permet d'établir la dépendance en $q$ et
$\varepsilon$ (pourvu que ce dernier reste suffisamment petit) des $n_j$ :
on trouve qu'il existe des \emph{entiers} $v_1 < \cdots <
v_n$ et $c_{0,j} < \cdots < c_{n,j}$ tels que, en posant $v_0 = -\infty$ 
et $v_{n+1} = +\infty$, on ait pour tout $j$, et sur chaque intervalle 
$[v_s, v_{s+1}[$, soit $n_j = c_{s,j}$, soit $n_j = q + c_{s,j}$, soit 
$n_j = q + \epsilon + c_{s,j}$. Il est enfin possible d'exprimer $q_i$ 
en fonction des $c_{s,j}$, à partir de quoi l'on déduit que $q_i$ est 
entier. 
Il résulte également de ces considérations qu'il existe dans $L$ (sans
tensoriser par $\puian$) un élément de valuation $(q_i,i)$ puisque 
dans le cas où $q$ est un nombre entier, toutes les opérations 
effectuées peuvent se faire dans $k[[u]]$.

Pour $q < q_i$, la proposition est clairement vraie puisque les deux
nombres qui apparaissent dans l'égalité à établir valent tous deux
$-\infty$.  D'autre part, dans tous les cas, le nombre $f_i(L)(q)$ 
s'interprète aussi comme la borne inférieure des nombres rationnels
$n$ tels que l'implication suivante soit vraie :
\begin{equation}
\label{eq:implication}
(x \in \puian \otimes_{k[[u]]} L \,\, \text{et}\,\, \val(x) = (q,i)) 
\quad \Longrightarrow \quad \sigma(x) \not\in u^n (\puian \otimes_{k[[u]]} 
L).
\end{equation}
Or, si $z_i \in L$ est un élément fixé de valuation $(q_i,i)$ et si 
l'on suppose $q \geq q_i$, un élément $x$ vérifie la prémisse de
l'implication si, et seulement si
$$x - u^{q - q_i} z_i \in (\puian \otimes_{k[[u]]} L) \cap \Bigg(
\sum_{j=1}^i u^{q+\varepsilon} e_j + \sum_{j=i+1}^d u^q e_j\Bigg)$$
pour un certain $\varepsilon > 0$. On conclut alors de manière semblable
à ce qui a été déjà fait : on commence par calculer l'intersection qui
apparaît dans la formule précédente en effectuant des opérations sur les
lignes d'une matrice, et réinjectant cela dans l'implication
\eqref{eq:implication}, on trouve qu'il existe des entiers $v_1 < 
\cdots < v_n$ et $c_0 < \cdots < c_n$ tel que sur chaque intervalle 
$[\frac {v_s} b, \frac{v_{s+1}} b[$, on ait $f_i(L)(q) = b q - c_s$.
Ce faisant, on obtient également que $f_i(L)(q) = \tilde \varphi_i(L)
(q)$ si $q \in \frac 1 b \Z$, d'où il résulte la proposition.
\end{proof}

En s'autorisant à travailler dans des corps encore plus gros que
$\puico$, on peut aussi interpréter les nombres $\tilde \varphi_i
(L)(q)$ pour $q \in \R$ comme des bornes supérieures du type précédent. 
Par exemple, on peut considérer l'anneau $k[[u^{\R^+}]]$ formé des
séries formelles $\sum_{i \in I} a_i u^i$ où $I \subset \R^+$ est un
monoïde de type fini. Celui-ci est encore muni d'une valuation naturelle
qui permet de définir $\val_{M,\R}$ sur $k[[u^{\R^+}]] \otimes_{k[[u]]}
M$. On a alors pour tout nombre réel $q$ :
$$\tilde \varphi_i(L)(q) = \sup_{\substack{x \in k[[u^{\R^+}]]
\otimes_{k[[u]]} L \\ \val_{M,\R}(x) = (q,i)}} \Big( \sup \big\{\, n \in
\R \,\,|\,\, \sigma(x) \in u^n (k[[u^{\R^+}]] \otimes_{k[[u]]} L) \,
\big\} \Big)$$
la démonstration étant en tout point analogue à celle de la
proposition précédente.

\subsection{Paramétrisation de l'espace des fonctions $\varphi$}
\label{subsec:parametrisation}

\begin{figure}
\begin{center}
\includegraphics{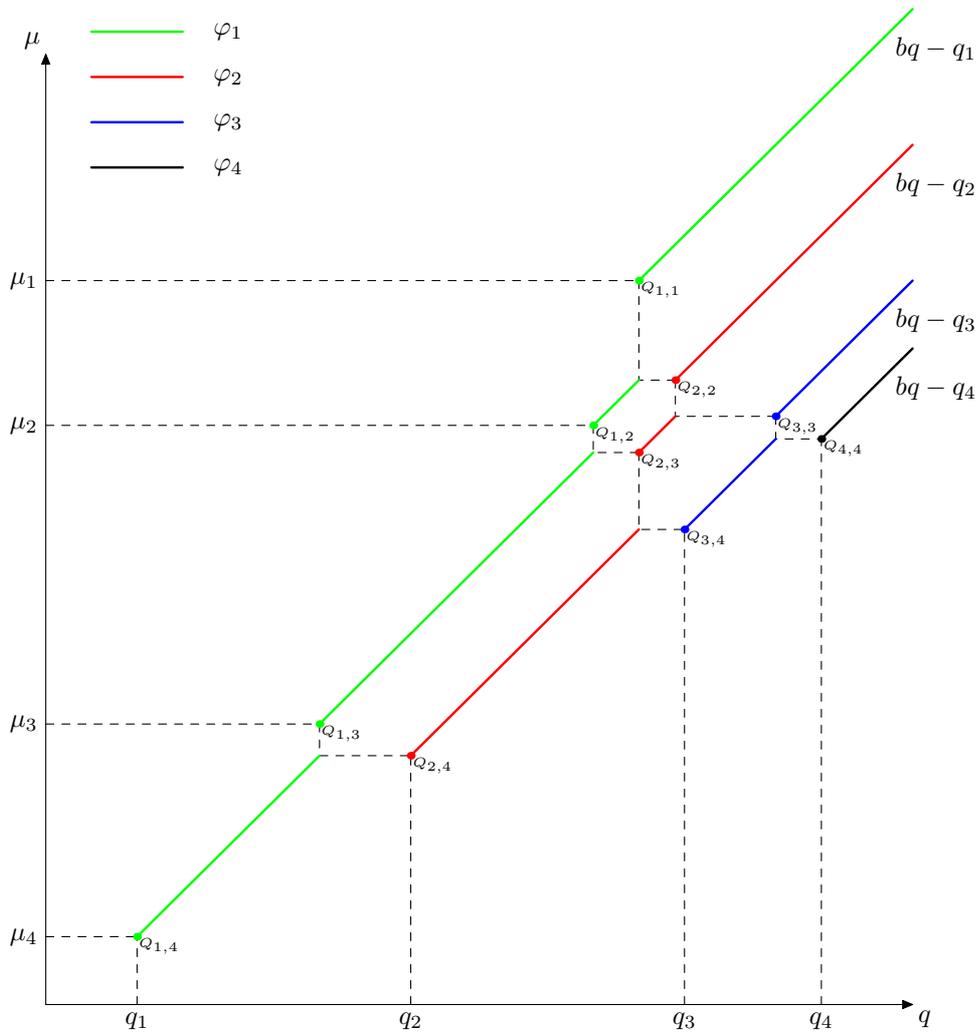}

\caption{Un exemple de $d$-uplet $(\varphi_1, \ldots, \varphi_d)$
appartenant à $\Phi$ avec $d=4$}
\label{fig:exemplevarphi}

\medskip

\begin{minipage}{13cm}
Les points $Q_{i,j}$ qui apparaissent sur le graphique sont ceux
de coordonnées $(q_{i,j}, \mu_{i,j})$. 

Les valeurs $q_i = q_{i,d}$ indiquées sur l'axe des abscisses 
correspondent aux endroits à partir desquels les fonctions $\varphi_i$ 
prennent des valeurs finies.

Les valeurs $\mu_j = \mu_{1,j}$ indiquées sur l'axe des ordonnées
correspondent, quant à elles, aux endroits à partir desquels les 
fonctions $\psi_j$ prennent des valeurs finies et aussi, dans le cas où 
le $d$-uplet $(\varphi_1, \ldots, \varphi_d)$ provient d'un réseau $L$, 
aux exposants des diviseurs élémentaires de $\sigma(k[[u^{1/b}]] 
\otimes_{k[[u]]} L)$ par rapport à $L$.
\end{minipage}
\end{center}
\end{figure}

L'objectif de ce numéro est de décrire complètement les $d$-uplets 
$\varphi = (\varphi_1, \ldots, \varphi_d)$ où les $\varphi_i$
sont des fonctions de $\R$ dans $\R \cup \{-\infty\}$ qui satisfont aux
conditions suivantes :
\begin{enumerate}
\label{enu:condphi}
\item on a $\varphi_1 \geq \varphi_2 \geq \cdots \geq \varphi_d$ ;
\item les fonctions $\varphi_i$ sont strictement
croissantes\footnote{On rappelle que l'on entend par là que les $\varphi_i$
sont croissantes et strictement croissantes sur l'intervalle où elles 
prennent des valeurs finies} et continues à droite ;
\item pour tout $i$, il existe un nombre réel $q_i$ (nécessairement
unique) tel que
\begin{itemize}
\item la fonction $\varphi_i$ prend des valeurs finies exactement
sur l'intervalle $[q_i, +\infty[$,
\item la fonction $\varphi_i$ est affine par morceaux sur 
$[q_i, +\infty[$ et pour presque tout $q$ dans cet intervalle, on
a $\varphi'_i(q) = b$, et
\item pour $q$ suffisamment grand, on a $\varphi_i(q) = bq -
q_i$ ;
\end{itemize}
\item pour $j \in \{1, \ldots, d\}$, il existe des fonctions strictement
croissantes et continues à droites $\psi_j : \R \to \R\cup\{-\infty\}$
telles que $\psi_1 \leq \psi_2 \leq \cdots \leq \psi_d$ et pour
tout couple $(q,\mu) \in \R^2$, il y a autant d'indices $i \in \{1,
\ldots, d\}$ tels que $\mu = \varphi_i(q)$ que d'indices $j \in \{
1, \ldots, d\}$ tels que $q = \psi_j(\mu)$.
\end{enumerate}
À partir de maintenant, on note $\Phi$ l'ensemble des $d$-uplets de
fonctions $\varphi = (\varphi_1, \cdots, \varphi_d)$
vérifiant les conditions précédentes. La proposition \ref{prop:varphi}, 
dit que les $\varphi(L) = (\varphi_1(L), \ldots, \varphi_d(L))$ 
provenant d'un réseau $L \subset M$ définissent des éléments de $\Phi$.
Par contre, il n'est pas vrai que, réciproquement, tout élément de 
$\Phi$ s'obtient de cette manière. En effet, si $\varphi$ provient d'un 
réseau, il vérifie en outre au moins les deux propriétés supplémentaires 
suivantes :
\begin{enumerate}
\label{enu:condphiZ}
\setcounter{enumi}{4}
\item pour tout $i$, le réel $q_i$ est un nombre entier ;
\item pour tout $i$, les réels en lesquels $\varphi_i$ est discontinue
appartiennent à $\frac 1 b \Z$.
\end{enumerate}
Ces deux dernièrs conditions seront appelées \emph{conditions 
d'intégrité} dans la suite de cet article, tandis que l'ensemble des 
éléments de $\Phi$ qui les satisfont sera noté $\Phi_\Z$.

\subsubsection{Les réels $q_{i,j}$ et $\mu_{i,j}$}

Soient $\varphi = (\varphi_1, \ldots, \varphi_d) \in \Phi$ et $\psi =
(\psi_1, \ldots, \psi_d)$ le $d$-uplet de fonctions correspondant.

On suppose pour commencer --- et il s'agit d'une hypothèse qui évitera
bien des problèmes techniques --- que $\varphi_1 > \cdots > \varphi_d$
où, étant donné deux fonctions $f,g : \R \to \R \cup \{-\infty\}$, on
convient que $f > g$ si $f(q) \geq g(q)$ pour tout réel $q$ et que
l'inégalité est stricte dès que $f(q) \neq -\infty$. On prolonge les
fonctions $\varphi_i$ et $\psi_j$ à $\R \cup \{-\infty\}$ en posant
$\varphi_i (-\infty) = \psi_j(-\infty) = -\infty$. Pour tout couple
$(i,j)$ d'entiers vérifiant $1 \leq i \leq j \leq d$, on définit :
\begin{itemize}
\item le nombre $q_{i,j}$ comme la borne inférieure des nombres
réels $q$ tels que $\psi_j \circ \varphi_i (q) \geq q$ ; 
\item le nombre $\mu_{i,j}$ comme la borne inférieure des nombres
réels $\mu$ tels que $\varphi_i \circ \psi_j (\mu) \geq \mu$.
\end{itemize}
L'ordre dans lequel sont classés les $\varphi_i$ et les $\psi_j$ impose 
que les fonctions $\varphi_i$ et $\psi_i$ sont inverses l'une de l'autre 
sur des voisinages de $+\infty$ ; il s'ensuit qu'il existe toujours des 
$q$ et des $\mu$ satisfaisant les inégalités précédentes. Par ailleurs 
si $q$ (resp. $\mu$) est suffisamment petit, on a $\varphi_i(q) = 
-\infty$ pour tout $i$ (resp. $\psi_j(\mu) = -\infty$ pour tout $j$). On 
en déduit que les $q_{i,j}$ et $\mu_{i,j}$ sont bien des nombres réels. 
Enfin, il est clair que $q_{i,j} \leq q_{i+1,j}$, $q_{i,j} \geq 
q_{i,j+1}$, $\mu_{i,j} \leq \mu_{i+1,j}$ et $\mu_{i,j} \geq \mu_{i,j+1}$ 
pour tout couple $(i,j)$ pour lequel cela a un sens. Dans la suite, pour 
des raisons pratiques, on posera également $q_{i,i-1} = \mu_{j+1,j} 
= +\infty$ pour tous indices $i$ et $j$.

\begin{lemme}
\label{lem:degenere}
Pour tout entiers $i$ et $j$ tels que $1 \leq i \leq j \leq d$, on a
$$\mu_{i,j} \leq \varphi_i (q_{i,j})
\qquad \text{et} \qquad
q_{i,j} \leq \psi_j (\mu_{i,j}).$$
De plus ces inégalités sont des égalités si, et 
seulement s'il existe $q$ tel que $\psi_j \circ \varphi_i(q) = q$ si, et
seulement s'il existe $\mu$ tel que $\varphi_i \circ \psi_j(\mu) = \mu$.
\end{lemme}

\begin{proof}
La croissance et la continuité à droite des fonctions $\varphi_i$ et
$\psi_j$ impliquent que $\psi_j \circ \varphi_i$ est aussi continue à
droite. On en déduit que $\psi_j \circ \varphi_i(q_{i,j}) \geq q_{i,j}$.
À partir de là, en appliquant $\varphi_i$, on obtient $\varphi_i \circ
\psi_j(\mu) \geq \mu$ avec $\mu = \varphi_i(q_{i,j})$. Par définition
de la borne inférieure, il vient $\mu_{i,j} \leq \mu$ comme annoncé. On
démontre de même l'autre inégalité.

Il est clair que si les inégalités sont des égalités, il existe
$q$ et $\mu$ satisfaisant à la condition du lemme : il suffit de prendre
$q = q_{i,j}$ et $\mu = \mu_{i,j}$. On suppose maintenant qu'il existe
$q$ tel que $\psi_j \circ \varphi_i (q) = q$. Du fait que la fonction
$\psi_j \circ \varphi_i - \id$ est en escalier (puisqu'elle est affine
par morceaux et que sa dérivée s'annule partout où elle est définie), on
déduit facilement que $\psi_j \circ \varphi_i(q_{i,j}) = q_{i,j}$.
En appliquant maintenant $\psi_j$ à la première égalité du lemme, on
trouve que $\psi_j (\mu_{i,j}) \leq q_{i,j}$ et donc finalement que
$\psi_j(\mu_{i,j}) = q_{i,j}$. L'autre égalité se démontre de même
analogue en appliquant $\varphi_i$ à la deuxième inégalité du lemme. Le
cas où il existe $\mu$ tel que $\varphi_i \circ \psi_j(\mu) = \mu$ se
traite pareillement.
\end{proof}

Dans le cas où l'une des inégalités du lemme est stricte, on dira que le 
couple $(i,j)$ est \emph{dégénéré}. Il est à noter que les couples de la 
forme $(i,i)$ ne sont jamais dégénérés puisque l'on a dit que sur des 
voisinages de l'infini les fonctions $\varphi_i$ et $\psi_i$ étaient 
inverses l'une de l'autre. En particulier, on a toujours $\varphi_i 
(q_{i,i}) = \mu_{i,i}$ et $\psi_i(\mu_{i,i}) = q_{i,i}$.

\begin{prop}
\label{prop:bijphipsi}
Soit $\varphi = (\varphi_1, \ldots, \varphi_d) \in \Phi$ tel que
$\varphi_1 > \cdots > \varphi_d$. Alors, pour tout couple $(i,j)$
avec $1 \leq i \leq j \leq d$, les
fonctions $\varphi_i$ et $\psi_j$ définissent par restriction des
bijections
$$\varphi_{i|[q_{i,j}, q_{i,j-1}[} : [q_{i,j}, q_{i,j-1}[ \; \to 
[\mu_{i,j}, \mu_{i+1,j}[ \quad \text{et} \quad
\psi_{j|[\mu_{i,j}, \mu_{i+1,j}[} : [\mu_{i,j}, \mu_{i+1,j}[ \; \to
[q_{i,j}, q_{i,j-1}[$$
inverses l'une de l'autre. (On notera qu'il est possible que les
intervalles précédents soient vides.)

De plus, pour tout $i$ (resp. tout $j$), la fonction $\varphi_i$ (resp.
$\psi_j$) vaut $-\infty$ sur l'intervalle $]-\infty, q_{i,d}[$ (resp.
l'intervalle $]-\infty, \mu_{1,j}[$).
\end{prop}

\begin{rem}
Telle quelle, la proposition est fausse sans l'hypothèse $\varphi_1 >
\cdots > \varphi_d$. On expliquera rapidement au \S \ref{subsec:multi}
comment modifier la définition des $q_{i,j}$ et $\mu_{i,j}$ pour que la
proposition s'étende sans cette hypothèse.
\end{rem}

\begin{proof}
On raisonne par récurrence sur $j$ en commençant par traiter le cas $j=1$.
%Considérons la fonction $f = \psi_1 \circ \varphi_1 - \id$. Comme nous
%l'avons déjà utilisé, elle est en escalier et vaut $0$ sur un voisinage
%de l'infini. Par ailleurs, elle est croissante : il suffit de le vérifier
%aux points de discontinuité, et si $q$ est un tel point, on a
%$\psi_1 \circ \varphi_1(q^-) \leq \psi_1 \circ \varphi_1(q^+)$ (car 
%$\psi_1 \circ \varphi_1$ est croissante) et donc $f(q^-) \leq f(q^+)$.
%Ainsi $f$ est négative, \emph{i.e.} $\psi_1 \circ \varphi_1 (q) \leq q$
%pour tout réel $q$. Par définition de $q_{1,1}$, il est alors évident
%que $\psi_1 \circ \varphi_1(q) = q$ pour tout $q \geq q_{1,1}$. De même,
%on démontre que $\varphi_1 \circ \psi_1 (\mu) = \mu$ pour tout $\mu 
%\geq \mu_{1,1}$. Par ailleurs, comme le couple $(1,1)$ est non dégénéré
Puisque $\psi_1$ est strictement croissante, il est clair qu'elle prend
des valeurs finies exactement sur un intervalle de la forme $[\mu_1,
+\infty[$ pour $\mu_1 \in \R$. Du fait que $\psi_1$ est la plus petite
des fonctions $\psi_j$, on déduit que la réunion des graphes des
$\varphi_i$ n'intersecte pas la région du plan suivante :
$$D_1 = \{ \, (q, \mu) \in \R^2 \,\, | \,\, q < \psi_1(\mu) \}.$$
Soit $\mu$ un réel plus grand ou égal à $\mu_1$. On pose $q = 
\psi_1(\mu)$ et et on considère un réel $\mu'$ strictement supérieur à 
$\mu$. Étant donné que $\psi_1$ est croissante, le couple $(q,\mu')$ est 
dans $D_1$, ce
qui signifie que $\varphi_i(q) \neq \mu'$ pour tout $i$. Comme ceci 
est vrai pour tout $\mu' > \mu$, on obtient $\varphi_i(q) \leq \mu$, 
\emph{i.e} $\varphi_i \circ \psi_1(\mu) \leq \mu$. Par ailleurs, on
sait que cette inégalité doit être une égalité pour au moins un $i$.
Si $i_0$ est cet indice privilégié, on a $\mu = \varphi_{i_0} \circ
\psi_1(\mu) \leq \varphi_1 \circ \psi_1(\mu) \leq \mu$ car les
$\varphi_i$ sont triés par ordre décroissant. Ainsi $\varphi_1 \circ
\psi_1(\mu) = \mu$, et ce pour tout $\mu \geq \mu_1$. On en déduit que
$\mu_{1,1} \leq \mu_1$. Mais comme $\psi_1$ vaut $-\infty$ sur les $\mu
< \mu_1$, il vient $\mu_{1,1} = \mu_1$ et donc $q_{1,1} = \psi_1
(\mu_1)$ par le lemme \ref{lem:degenere}. En outre, sur l'intervalle 
$[\mu_{1,1}, +\infty[$ la fonction $\psi_1$ est inversible à gauche, et 
son inverse à gauche est la restriction de $\varphi_1$ à $[q_{1,1}, 
+\infty[$. Du fait que ces fonctions sont en outre affines par morceaux 
et strictement croissantes, on en déduit facilement ce qui est énoncé 
dans la proposition.

Plutôt que de traiter l'hérédité de la récurrence dans le cas général
--- ce qui multiplierait encore les notations ---, on se contente 
d'expliquer comment le cas \og $j=2$ \fg\ se déduit de ce que
l'on vient de faire, les arguments pour les $j$ supérieurs étant
similaires. Comme précédemment, on commence par remarquer que la
fonction $\psi_2$ prend des valeurs finies exactement sur un intervalle
de la forme $[\mu_2, +\infty[$ pour un certain nombre réel $\mu_2$. De
$\psi_2(\mu_{1,1}) \geq \psi_1(\mu_{1,1}) > -\infty$, on tire $\mu_2
\leq \mu_{1,1}$. On introduit le domaine :
$$D_2 = \{ \, (q, \mu) \in \R^2 \,\, | \,\, q < \psi_2(\mu) \}.$$
Des résultats de l'étude menée pour $j = 1$, on déduit que $D_2$ intersecte
la réunion des graphes de $\varphi_i$ au plus selon le graphe de la
restriction de $\varphi_1$ à l'intervalle $[q_{1,1}, +\infty[$. Soient
$\mu \geq \mu_2$ et $q = \psi_2(\mu)$. On considère un nombre réel $\mu' 
> \mu$. Le couple $(q, \mu')$ appartenant à $D_2$, il en résulte que :
\begin{itemize}
\item si $q < q_{1,1}$, alors $\varphi_i(q) \neq \mu'$ pour tout $i \in 
\{1, \ldots, d\}$ ;
\item si $q \geq q_{1,1}$, alors $\varphi_i(q) \neq \mu'$ pour tout $i 
\in \{2, \ldots, d\}$.
\end{itemize}
Les conclusions précédentes étant valables pour tout $\mu' > \mu$, on
peut remplacer dans leurs énoncés \og $\varphi_i(q) \neq \mu'$ \fg\ par
\og $\varphi_i(q) \leq \mu$ \fg. En utilisant le fait que les fonctions
$\varphi_i$ sont triées par ordre décroissant, et que la fonction
$\varphi_1$ est connue sur $[q_{1,1}, +\infty[$, on obtient pour tout
$\mu \geq \mu_2$ :
\begin{itemize}
\item si $\psi_2(\mu) < q_{1,1}$, alors $\varphi_1 \circ \psi_2(\mu) 
= \mu$ ;
\item si $\psi_2(\mu) \geq q_{1,1}$, alors $\varphi_2 \circ \psi_2(\mu) 
= \mu$, et donc $\varphi_1 \circ \psi_2(\mu) \geq \mu$.
\end{itemize}
On déduit déjà de cela que $\mu_{1,2} = \mu_2$ puisque $\varphi_1 \circ
\psi_2 (\mu)$ vaut $-\infty$ si $\mu < \mu_2$ et est supérieur ou égal à
$\mu$ sinon. On a également $\psi_2(\mu_{2,2}) = q_{2,2}$ et $\varphi_2
(q_{2,2}) = \mu_{2,2}$ étant donné que le couple $(2,2)$ n'est pas
dégénéré. Soit $\mu'$ l'infimum des nombres réels $\mu$ tels que 
$\psi_2(\mu) 
> q_{1,1}$. La restriction de $\psi_2$ à $[\mu_{1,2}, \mu'[$ admet 
alors pour inverse $\varphi_1$ tandis que sa restriction à $[\mu', 
+\infty[$ admet pour inverse $\varphi_2$. Pour conclure, il suffit donc
de montrer que $\mu' = \mu_{2,2}$ et que si $\mu_{1,2} < \mu_{2,2}$ 
alors $\psi_2(\mu_{1,2}) = q_{1,2}$.
Du fait que $\varphi_2 \circ \psi_2(\mu') = \mu'$, on déduit que 
$\mu_{2,2} \leq \mu'$. Mais si l'inégalité était stricte, on aurait
$\varphi_1 = \varphi_2$ sur l'intervalle de la forme $[q_{1,1} -
\varepsilon, q_{1,1}[$ (pour $\varepsilon > 0$), ce que l'on a exclu
au départ. L'autre point résulte maintenant des descriptions que l'on 
vient d'obtenir.
\end{proof}

Étant donné qu'une fonction définie sur un intervalle à valeurs dans un
autre intervalle, qui est à la fois affine par morceaux et bijective est
affine, on déduit facilement de la proposition la description suivante
des fonctions $\varphi_i$ et $\psi_j$ :
\begin{equation}
\label{eq:formephi}
\begin{array}{rrcll}
\varphi_i\,:& q & \mapsto & -\infty & \text{si } q < q_{i,d} \\
& q & \mapsto & b(q - q_{i,j}) + \mu_{i,j}
& \text{si } q_{i,j} \leq q < q_{i,j-1},
\quad \text{et ce pour tout } j \in \{i, \ldots, d\} \\ \\
\psi_j\,:& \mu & \mapsto & -\infty & \text{si } \mu < \mu_{1,j} \\
& \mu & \mapsto & b^{-1} (\mu - \mu_{i,j}) + q_{i,j}
& \text{si } \mu_{i,j} \leq \mu < \mu_{i+1,j},
\quad \text{et ce pour tout } i \in \{1, \ldots, j\}
\end{array}
\end{equation}
Il résulte en particulier de cette écriture que les fonctions $\varphi_i$
sont entièrement déterminées par la donnée des $q_{i,j}$ et $\mu_{i,j}$.

\subsubsection{Relations entre les $q_{i,j}$ et $\mu_{i,j}$}

On sait déjà que les nombres $q_{i,j}$ et $\mu_{i,j}$ ne peuvent être 
quelconques puisque ceux-ci vérifient les inégalités $q_{i,j} \leq 
q_{i+1,j}$, $q_{i,j} \geq q_{i,j+1}$, $\mu_{i,j} \leq \mu_{i+1,j}$ et 
$\mu_{i,j} \geq \mu_{i,j+1}$. Dans ce paragraphe, on détermine un 
certain nombres d'autres contraintes auxquelles ils doivent satisfaire.

\paragraph{Relations égalitaires}

On considère $\varphi = (\varphi_1, \ldots, \varphi_d) \in \Phi$
et on suppose encore $\varphi_1 > \cdots > \varphi_d$. Par définition,
au voisinage de $+\infty$, la fonction $\varphi_i$ est donnée par $q
\mapsto bq - q_{i,d}$ (puisque la fonction $\varphi_i$ prend des valeurs
finis à partir de $q_{i,d}$).  En comparant avec la forme de $\varphi_i$
obtenu ci-dessus, on trouve :
\begin{equation}
\label{eq:rele1}
\mu_{i,i} = b q_{i,i} - q_{i,d}
\end{equation}
pour tout $i \in \{1, \ldots, d\}$.
D'autre part, par la proposition \ref{prop:bijphipsi}, la fonction
$\varphi_i$ réalise une bijection de l'intervalle $[q_{i,j}, q_{i,j-1}[$
sur l'intervalle $[\mu_{i,j}, \mu_{i+1,j}[$. Comme on sait que cette 
fonction dilate la mesure de Lebesgue d'un facteur $b$, il vient :
\begin{equation}
\label{eq:rele2}
\mu_{i+1,j} - \mu_{i,j} = b (q_{i,j-1} - q_{i,j})
\end{equation}
pour tout $i,j \in \{1, \ldots, d\}$ tels que $1 \leq i < j \leq d$.
À partir des ces deux relations, on voit facilement que les $\mu_{i,j}$ 
s'expriment en fonction des $q_{i,j}$ :
\begin{equation}
\label{eq:rele3}
\mu_{i,j} = b q_{j,j} - q_{j,d} + b \cdot \sum_{s=i}^{j-1} (q_{s, j} 
- q_{s,j-1})
\end{equation}
pour tout couple $(i,j)$ tel que $1 \leq i \leq j \leq d$. En fait,
la formule \eqref{eq:rele3} implique réciproquement les
formules \eqref{eq:rele1} et \eqref{eq:rele2}. Ainsi, on peut décider
d'oublier ces deux dernières relations et de ne travailler qu'avec les
$q_{i,j}$. On peut également inverser les formules \eqref{eq:rele3} et
exprimer les $q_{i,j}$ en fonction des $\mu_{i,j}$ ; on obtient pour
$1 \leq i \leq j \leq d$
\begin{equation}
\label{eq:rele4}
q_{i,j} = \frac 1{b-1} \cdot \Bigg( \mu_{i,i} + \sum_{s=i+1}^j (\mu_{i,s} 
- \mu_{i+1,s}) + \sum_{s=j+1}^d \frac{\mu_{i,s} - \mu_{i+1,s}} b\Bigg).
\end{equation}
Il est donc également possible d'oublier les $q_{i,j}$ et de
travailler uniquement avec les $\mu_{i,j}$. L'avantage, néanmoins, de
continuer à considérer simultanément les $q_{i,j}$ et les $\mu_{i,j}$
réside dans le fait que les formules qui apparaîtront dans la suite 
pourront souvent s'écrire manière plus agréable.

\paragraph{Relations inégalitaires}

On rappelle que l'on a déjà vu les inégalités
\begin{eqnarray}
q_{i,j} \geq q_{i,j+1} & ; & \mu_{i,j} \leq \mu_{i+1,j} \label{eq:reli1} \\
q_{i,j} \leq q_{i+1,j} & ; & \mu_{i,j} \geq \mu_{i,j+1} \label{eq:reli2}
\end{eqnarray}
pour tout couple $(i,j)$ pour lequel cela a un sens.

\begin{lemme}
Avec les notations précédentes, on a les relations supplémentaires :
\begin{equation}
\label{eq:reli3}
q_{i,j} \leq q_{i+1,j+1} \qquad \text{et} \qquad
\mu_{i,j} \geq \mu_{i+1,j+1}
\end{equation}
pour tout couple $(i,j)$ tel que $1 \leq i < j \leq d$.
\end{lemme}

\begin{proof}
On démontre seulement la première inégalité, la seconde étant
complètement analogue. Si $q_{i,j} = q_{i,j+1}$ ou $q_{i+1,j+1} =
q_{i+1,j}$, l'inégalité résulte de \eqref{eq:reli2}. On peut donc
supposer que $q_{i,j} > q_{i,j+1}$ et $q_{i+1,j+1} < q_{i+1,j}$. Dans ce cas,
on applique la proposition \ref{prop:bijphipsi} qui nous assure que la
fonction $\psi_{j+1}$ réalise une bijection croissante de $[\mu_{i,j+1},
\mu_{i+1, j+1}[$ dans $[q_{i,j+1}, q_{i,j}[$, et également de
$[\mu_{i+1,j+1}, \mu_{i+2, j+1}[$ dans $[q_{i+1,j+1}, q_{i+1,j}[$. Par
hypothèse, tous ces intervalles sont non vides. On en déduit que 
$\lim_{\mu \to \mu_{i+1,j+1}^-} \psi_{j+1}(\mu) = q_{i,j}$ et que 
$\psi_{j+1}(\mu_{i+1,j+1}) = q_{i+1,j+1}$. En utilisant la croissance de 
$\psi_{j+1}$, on obtient finalement $q_{i,j} \leq q_{i+1,j+1}$ comme 
voulu.
\end{proof}

En fait, les inégalités \eqref{eq:reli2} résultent de \eqref{eq:reli1} 
et \eqref{eq:reli3}. En outre, en vertu de l'égalité \eqref{eq:rele2}, 
les deux inégalités de la ligne \eqref{eq:reli1} se déduisent 
mutuellement l'une de l'autre. Ainsi, les six jeux d'inégalités obtenues 
précédemment se résument finalement aux trois suivants :
\begin{equation}
\label{eq:reli4}
q_{i,j} \geq q_{i,j+1} \quad ; \quad
q_{i,j} \leq q_{i+1,j+1} \quad ; \quad
\mu_{i,j} \geq \mu_{i+1,j+1}.
\end{equation}
où $(i,j)$ parcourt l'ensemble des couples d'entiers tels que $1 \leq
i \leq j < d$. Bien entendu, par ailleurs, les inégalités portant sur 
les $\mu_{i,j}$ peuvent se réécrire en termes de $q_{i,j}$ en utilisant 
la formule \eqref{eq:rele3}, et réciproquement en utilisant la formule
\eqref{eq:rele4}.

\subsubsection{Un mot sur la gestion des multiplicités}
\label{subsec:multi}

Dans ce qui précède, on a toujours supposé pour simplifier que 
$\varphi_1 >
\cdots > \varphi_d$, c'est-à-dire si l'on préfère qu'il n'y a aucun
point $(q,\mu)$ qui appartient simultanément aux graphes du plusieurs
fonctions $\varphi_i$. Dans le cas où cette condition n'est pas
réalisée, des complications techniques apparaissent. En
particulier, la définition des $q_{i,j}$ et $\mu_{i,j}$ n'est plus
correcte et doit être remplacée par :
\begin{itemize}
\item le nombre $q_{i,j}$ est la borne inférieure des nombres
réels $q$ tels que soit $\psi_j \circ \varphi_i (q) > q$, soit $\psi_j
\circ \varphi_i (q) = q$ et
\begin{equation}
\label{eq:cardmult}
\card \, \big\{ \, i' \leq i \,\,|\,\, \varphi_{i'} (q) = \mu \, \big\}
\, \leq \,
\card \, \big\{ \, j' \leq j \,\,|\,\, \psi_{j'} (\mu) = q \, \big\}
\end{equation}
où on a posé $\mu = \varphi_i(q)$ ;
%, en posant $\mu = \varphi_i(q)$, il y a au
%moins autant d'indices $i' \leq i$ tels que $\mu = \varphi_{i'} (q)$
%que d'indices $j' \leq j$ tels que $q = \psi_{j'}(\mu)$ ;
\item le nombre $\mu_{i,j}$ est la borne inférieure des nombres
réels $\mu$ tels que soit $\varphi_i \circ \psi_j (\mu) > \mu$, soit
$\varphi_i \circ \psi_j (\mu) = \mu$ et l'inégalité \eqref{eq:cardmult}
est vérifiée avec $q = \psi_j(\mu)$.
%, et en posant $q = \psi_j(\mu)$,
%il y a au moins autant d'indices $i' \leq i$ tels que $\mu =
%\varphi_{i'}(q)$ que d'indices $j' \leq j$ tels que $q = \psi_{i'}(\mu)$.
\end{itemize}
Dans le cas où $\varphi_1 > \cdots > \varphi_d$, on retrouve bien la
définition donnée auparavant. En effet, il y a alors un unique indice
$i'$ (resp. $j'$) tel que $\mu = \varphi_{i'} (q)$ (resp. $q =\psi_{j'}
(\mu)$) à savoir $i' = i$ (resp. $j' = j$), et donc la condition
supplémentaire est toujours vérifiée. L'idée dans cette définition 
est que si, étant donné un couple $(q, \mu)$, on note $i_1 < i_2 <
\cdots < i_k$ et $j_1 < j_2 < \cdots < j_k$ les entiers tels que :
$$\varphi_{i_1}(q) = \varphi_{i_2}(q) = \cdots = \varphi_{i_k}(q) =
\mu \quad \text{et} \quad
\psi_{j_1}(\mu) = \psi_{j_2}(\mu) = \cdots = \psi_{j_k}(\mu) = q,$$
alors on a bien $\psi_{j_s} \circ \varphi_{i_s} (q) = q$ mais, si $t<s$
le nombre $\psi_{j_t} \circ \varphi_{i_s} (q)$ doit être considéré comme
infinitésimalement plus petit que $q$ et, en tout cas, ne doit pas être
pris en compte dans la définition de $q_{i_s,j_t}$. En formalisant cette
vision des choses et en reprenant les arguments développés précédemment,
on peut montrer avec un peu de présévérance (exercice laissé au lecteur) 
que la proposition \ref{prop:bijphipsi} est encore vraie pour tous les 
$\varphi \in \Phi$ à condition de prendre la définition modifiée 
précédente. De même, la formule \eqref{eq:formephi} demeure, ainsi que 
les relations \eqref{eq:rele1}, \eqref{eq:rele2} et \eqref{eq:reli4}.

\subsubsection{Une bijection}

À partir de maintenant, on va faire varier les éléments $\varphi$ dans 
$\Phi$. C'est pourquoi, afin de lever tout risque d'ambiguïté, on notera 
dans la suite $q_{i,j}(\varphi)$ et $\mu_{i,j}(\varphi)$ respectivement 
à la place de $q_{i,j}$ et $\mu_{i,j}$. Soit $I$ l'ensemble des couples 
$(i,j)$ tels que $1 \leq i \leq j \leq d$. On considère l'espace 
vectoriel $(\R^2)^I$ des suites $(q_{i,j}, \mu_{i,j})$ indicées par les 
éléments de $I$. Soit $K$ le sous-ensemble convexe de $(\R^2)^I$ définis 
par les relations \eqref{eq:rele1}, \eqref{eq:rele2} et 
\eqref{eq:reli4}.

\begin{theo}
\label{theo:parametrisation}
L'application
$$\Phi \to K, \quad \varphi \mapsto (q_{i,j}(\varphi), \mu_{i,j}
(\varphi))_{(i,j)\in I}$$
est une bijection et son inverse est donné par la formule
\eqref{eq:formephi}.
\end{theo}

\begin{proof}
Il s'agit de démontrer qu'étant donné $(q_{i,j}, \mu_{i,j}) \in K$, la
formule \eqref{eq:formephi} définit un $d$-uplet $\varphi = (\varphi_1,
\ldots, \varphi_d)$ qui appartient à $\Phi$ et qui est tel que 
$q_{i,j}(\varphi) = q_{i,j}$ et $\mu_{i,j}(\varphi) = \mu_{i,j}$
pour tout $(i,j) \in I$. Les conditions 2 et 3 qui définissent 
l'ensemble $\Phi$ (voir page \pageref{enu:condphi})
ne posent aucun problème à part peut-être en ce qui concerne la
croissance des $\varphi_i$, mais celle-ci résulte directement des
égalités et des inégalités supposées sur les $q_{i,j}$ et $\mu_{i,j}$.
On a en outre les inégalités suivantes qui seront utiles dans la suite :
\begin{eqnarray}
\forall q < q_{i,j}, & &
\varphi_i(q) \leq b(q-q_{i,j}) + \mu_{i+1,j+1} \label{eq:majophi} \\
\forall q \geq q_{i,j}, & &
\varphi_i(q) \geq b(q-q_{i,j}) + \mu_{i,j} \label{eq:minophi}
\end{eqnarray}
On montre à présent la condition 1, c'est-à-dire que pour tout indice $i
\in \{1, \ldots, d-1\}$, on a $\varphi_i \geq \varphi_{i+1}$. Soit $q
\in \R$. Si $q < q_{i+1,d}$, on a bien $\varphi_i(q) \geq \varphi_{i+1}
(q) = -\infty$. Sinon, il existe $j$ tel que $q_{i+1,j} \leq q <
q_{i+1,j-1}$. On a alors
$\varphi_{i+1}(q) = b (q - q_{i+1,j}) + \mu_{i+1,j} 
\leq b(q - q_{i,j-1}) + \mu_{i,j-1} \leq \varphi_i(q)$
la première inégalité résultant des hypothèses $q_{i+1,j} \geq 
q_{i,j-1}$ et $\mu_{i+1,j} \leq \mu_{i,j-1}$, et la seconde résultant
de \eqref{eq:minophi} après avoir remarqué que $q \geq q_{i+1,j} \geq
q_{i,j-1}$. 

On en vient à la condition 4. Bien sûr, on prend les fonctions
$\psi_j$ définies par la formule \eqref{eq:formephi}. Les inégalités
supposées impliquent de même que précédemment qu'elles sont strictement
croissantes et rangées par ordre croissant. Soit $(q,\mu)$ un couple de
nombres réels. Par définition, un indice $i \in \{1, \ldots, d\}$
vérifie $\mu = \varphi_i(q)$ si, et seulement s'il existe $j \in \{1,
\ldots, d\}$ tel que
\begin{equation}
\label{eq:condij}
q_{i,j} \leq q < q_{i,j-1} \quad \text{et} \quad
\mu - \mu_{i,j} = b(q - q_{i,j})
\end{equation}
De plus, si un tel indice $j$ existe, il est clair qu'il est unique.
Ainsi, il existe autant d'indices $i$ satisfaisant $\mu = \varphi_i
(q)$ que de couples $(i,j)$ satisfaisant \eqref{eq:condij}. De même, 
on démontre qu'il existe autant d'indices $j$ satisfaisant $q =
\psi_j(\mu)$ que de couples $(i,j)$ satisfaisant :
\begin{equation}
\label{eq:condij2}
\mu_{i,j} \leq \mu < \mu_{i+1,j} \quad \text{et} \quad
\mu - \mu_{i,j} = b(q - q_{i,j}).
\end{equation}
Il suffit donc de montrer que les conditions \eqref{eq:condij} et
\eqref{eq:condij2} sont équivalentes, ce qui résulte sans peine de
l'égalité \eqref{eq:rele2}.

Il reste enfin à démontrer que $q_{i,j}(\varphi) = q_{i,j}$ et
$\mu_{i,j}(\varphi) = \mu_{i,j}$. Grâce à la relation \eqref{eq:rele3}, 
il suffit de démontrer l'égalité pour les $q_{i,j}$. Comme précédemment, 
on n'écrit la preuve que dans le cas où $\varphi_1 > \cdots > 
\varphi_d$. Il faut alors montrer que pour tout $(i,j) \in I$, on a 
$\psi_j \circ \varphi_i(q_{i,j}) \geq q_{i,j}$ et $\psi_j \circ 
\varphi_i(q) < q$ pour tout $q < q_{i,j}$. La formule \eqref{eq:minophi} 
montre que $\varphi_i(q_{i,j}) \geq \mu_{i,j}$, tandis que, de manière 
analogue, on démontre que $\psi_j(\mu_{i,j}) \geq q_{i,j}$. Il en 
résulte, en utilisant la croissance, que $\psi_j \circ
\varphi_i(q_{i,j}) \geq \psi_j(\mu_{i,j}) \geq q_{i,j}$. Soit $q <
q_{i,j}$. Si $\varphi_i(q) = -\infty$, il n'y a rien à démontrer ; on
suppose donc que ce n'est pas le cas. La formule \eqref{eq:majophi} 
implique que $\varphi_i(q) < \mu_{i+1,j+1} \leq \mu_{i,j}$. Comme, par
ailleurs, sur l'intervalle $]-\infty, \mu_{i,j}[$, on a $\psi_j(\mu) 
\leq b^{-1}(\mu - \mu_{i,j}) + q_{i-1,j-1}$, on obtient :
$$\psi_j \circ \varphi_i(q) \leq q - (q_{i,j} - q_{i-1,j-1}) -
\frac {\mu_{i,j} - \mu_{i+1,j+1}} b.$$
Si l'inégalité est stricte, la démonstration est terminée. Sinon, cela 
signifie que toutes les inégalités utilisées sont des égalités, et donc 
en particulier que $\varphi_i(q) = b(q - q_{i,j}) + \mu_{i+1,j+1}$ et 
que $q_{i,j} - q_{i-1,j-1} = \mu_{i,j} - \mu_{i+1,j+1} = 0$ puisque ces 
deux différences sont toujours positives ou nulles. On applique alors
l'inégalité \eqref{eq:majophi} au couple $(i-1, j-1)$ : cela donne
$\varphi_{i-1} (q) \leq b(q - q_{i-1,j-1}) + \mu_{i,j} = b(q - q_{i,j})
+ \mu_{i,j} = \varphi_i(q)$, ce qui contredit l'hypothèse.
\end{proof}

On rappelle que l'on avait défini un sous-ensemble $\Phi_\Z$ de $\Phi$
caractérisé par certaines conditions dites d'intégrité (les conditions
5 et 6, page \pageref{enu:condphiZ}). La proposition suivante montre que
ce sous-ensemble est facilement caractérisable à l'aide de la bijection
du théorème précédent.

\begin{prop}
\label{prop:integrite}
Pour tout $\varphi \in \Phi$, les trois conditions suivantes sont
équivalentes :
\begin{itemize}
\item[i)] $\varphi \in \Phi_\Z$
\item[ii)] pour tout $(i,j) \in I$, $q_{i,j}(\varphi) \in \frac 1 b \Z$,
et pour tout $i \in \{1, \ldots, d\}$, $q_{i,d}(\varphi) \in \Z$ ;
\item[iii)] pour tout $(i,j) \in I$, $\mu_{i,j}(\varphi) \in \Z$ et pour
tout $i \in \{1, \ldots, d\}$, la somme $\mu_{i,i}(\varphi) +
\mu_{i,i+1}(\varphi) + \cdots + \mu_{i,d}(\varphi)$ est divisible par
$b-1$.
\end{itemize}
\end{prop}

\begin{proof}
L'équivalence entre les conditions \emph{ii)} et \emph{iii)} résulte
directement des formules \eqref{eq:rele3} et \eqref{eq:rele4} qui
permettent d'exprimer les $\mu_{i,j}$ en fonction des $q_{i,j}$ et
réciproquement. Reste donc à montrer l'équivalence entre \emph{i)} 
et \emph{ii)}. À partir de la proposition \ref{prop:bijphipsi}, on
déduit que \emph{ii)} implique \emph{i)}. Si l'on suppose
maintenant que $\varphi \in \Phi_\Z$, alors toutes les 
composées $\psi_j \circ \varphi_i$ sont constantes sur les
intervalles de la forme $[v, v + \frac 1 b[$ pour $v \in \frac 1 b
\Z$ et de là, en revenant à la définition, on déduit que les $q_{i,j}
(\varphi)$ appartiennent tous à $\frac 1 b \Z$. Finalement, il est
clair que $q_{i,d}(\varphi)$ est entier pour tout $i$, puisque
celui-ci est égal à $q_i$ qui est justement supposé entier.
\end{proof}

\subsection{Les variétés $\calX_\varphi$ et leurs dimensions}

Étant donné un $d$-uplet $\tilde \varphi = (\tilde \varphi_1, \ldots,
\tilde \varphi_d)$, on définit $\tilde \calX_{\tilde \varphi}(k)$ comme
l'ensemble des réseaux $L \subset M$ tels que $\tilde \varphi_i(L) =
\tilde \varphi_i$ pour tout $i$. Si l'on note $\varphi$ le $d$-uplet de
fonctions $\varphi_i : \R \to \R \cup \{-\infty\}$ obtenu à partir de
$\tilde \varphi$ après réordonnement et prolongement, les propositions
\ref{prop:varphi} et \ref{prop:bijphipsi} montrent ensemble que $\tilde
\calX_{\tilde \varphi}(k)$ est inclus dans $\calX_\mu(k)$ pour $\mu =
(\mu_{1,1}(\varphi), \mu_{1,2}(\varphi), \ldots, \mu_{1,d} (\varphi))$.
De façon plus précise, on démontre comme dans le lemme 4.2 de
\cite{viehmann} que l'on obtient ce faisant une sous-variété $\tilde
\calX_{\tilde \varphi}$ de $\calX_\mu$ qui est localement fermée.

De façon similaire, étant donné $\varphi \in \Phi_\Z$, on note
$\calX_\varphi(k)$ l'ensemble des réseaux $L \subset M$ tels que
$\varphi(L) = \varphi$ ; cet ensemble s'écrit manifestement comme une
union de $\tilde \calX_{\tilde \varphi}(k)$. Le lemme suivant montre
qu'il s'agit même d'une union finie et, par voie de conséquence, que
$\calX_\varphi(k)$ est aussi l'ensemble des $k$-points d'une
sous-variété algébrique localement fermée de $\calX_\mu$ pour le même
$\mu$ que précédemment.

\begin{lemme}
On fixe un élément $\varphi \in \Phi_\Z$. Alors, 
il n'existe qu'un nombre fini de $d$-uplets de fonctions $(\tilde
\varphi_1, \ldots, \tilde \varphi_d)$ qui satisfont aux conditions de
la proposition \ref{prop:varphi} et qui redonnent le $d$-uplet $\varphi$ 
après réordon\-nement.
\end{lemme}

\begin{proof}
Si $(\tilde \varphi_1, \ldots, \tilde \varphi_d)$ est un tel $d$-uplet,
alors il existe une permutation $w \in \mathfrak S_d$ telle que, pour
tout $i$, on ait $\tilde \varphi_i(v) = - \infty$ pour $v < q_{w(i),
d}(\varphi)$ et $\tilde \varphi_i(v) = b v - q_{w(i), d}(\varphi)$ pour
$v$ suffisamment grand. Du fait que les fonctions $\tilde \varphi_i$
doivent en outre être croissantes, on déduit que l'égalité $\tilde
\varphi_i(v) = b v - q_{w(i), d}(\varphi)$ vaut pour
tout $v \geq q_{d,d}(\varphi)$. Après cela, il ne reste plus qu'un 
nombre fini de $v$ et donc qu'un nombre fini de possibilités pour
attribuer les valeurs manquantes aux $\tilde \varphi_i$ puisque pour
chaque $v$, on ne peut que permuter les nombres $\varphi_i(v)$ et on
a donc au maximum $d!$ possibilités.
\end{proof}

\subsubsection{La fonction dimension sur $\Phi$}

\begin{deftn}
\label{def:dim}
Soit $\Leb$ la mesure de Lebesgue sur $\R$.
Si $\varphi = (\varphi_1, \ldots, \varphi_d) \in \Phi$, on pose :
$$\dim(\varphi) = \sum_{1 \leq i < i' \leq d} \Leb (\varphi_{i'}(\R)
\backslash \varphi_i(\R))$$
où, si $E$ et $E'$ sont des ensembles, on note $E \backslash E'$ 
l'ensemble des éléments qui appartiennent à $E$ mais pas à $E'$.
\end{deftn}

\begin{rem}
Pour $\varphi = (\varphi_1, \ldots, \varphi_d) \in \Phi$, il est
clair qu'il existe des constantes $A$ et $B$ telles que $[A, +\infty[ \;
\subset \varphi_i(\R) \subset [B, +\infty[$ pour tout $i$. On en déduit
que les différences $\varphi_{i'}(\R) \backslash \varphi_i(\R)$ sont
toutes incluses dans l'intervalle $[A,B]$ et donc, en particulier,
qu'elles ont une mesure finie. Ainsi $\dim(\varphi)$ est toujours fini.
\end{rem}

\emph{Via} la bijection de la proposition \ref{prop:varphi}, un élément
$\varphi \in \Phi$ est entièrement déterminé par la donnée des
$q_{i,j}(\varphi)$ et $\mu_{i,j}(\varphi)$. Ainsi, le nombre
$\dim(\varphi)$ que l'on vient de définir doit s'exprimer en fonction
des $q_{i,j}(\varphi)$ et $\mu_{i,j}(\varphi)$. On a plusieurs
possibilités pour cela, comme le montre le lemme suivant.

\begin{lemme}
\label{lem:dim}
Pour tout $\varphi \in \Phi$, on a :
\begin{eqnarray*}
\dim(\varphi) & = &
\sum_{j=1}^d (d+1-j) \cdot \mu_{1,j}(\varphi) - \sum_{(i,j) \in I}
\mu_{i,j}(\varphi) \\
%& = & \sum_{i=1}^d j \cdot (\mu_{1,d+1-j}(\varphi) - q_{j,d}(\varphi)) \\
& = & b \cdot \sum_{(i,j) \in I} q_{i,j}(\varphi) + 
\sum_{i=1}^d (2i-1-d-bi) \cdot q_{i,d}(\varphi).
\end{eqnarray*}
\end{lemme}

\begin{rem}
On constate en particulier --- et ce sera crucial 
dans la suite --- que $\dim(\varphi)$ dépend de façon linéaire des
$\mu_{i,j}(\varphi)$ et $q_{i,j}(\varphi)$.
\end{rem}

\begin{proof}
Comme d'habitude, on ne donne la démonstration que dans le cas où
$\varphi_1 > \cdots > \varphi_d$.
On fixe un indice $i \in \{1, \ldots, d\}$. D'après la proposition
\ref{prop:bijphipsi}, l'image de la fonction $\varphi_i$ s'écrit :
$$\varphi_i(\R) = \bigsqcup_{i \leq j \leq d} [\mu_{i,j}, \mu_{i+1,j}[.$$
Comme on a en outre $\mu_{i+1,j+1} \leq \mu_{i,j}$, les intervalles qui
apparaissent dans l'union précédente sont \og rangés par ordre
décroissant \fg. On en déduit que le réel $\mu$ n'est pas dans l'image
de $\varphi_i$ si, et seulement s'il existe $j \in \{i, \ldots, d\}$ tel
que $\mu_{i+1,j+1} \leq \mu < \mu_{i,j}$ où on a posé par convention
$\mu_{i+1,d+1} = -\infty$. Pour un tel $\mu$, on se propose de compter
le nombre d'indices $i' > i$ tels que $\mu$ appartienne à l'image de
$\varphi_{i'}$. Comme les fonctions $\varphi_{i'}$ sont injectives sur
l'intervalle où elles prennent des valeurs finies et que l'on a supposé
$\varphi_1 > \cdots > \varphi_d$, cela revient encore à compter le
nombre de $q$ pour lesquels il existe $i' > i$ tel que $\mu =
\varphi_{i'}(q)$. Cette condition se réécrit encore :
$$\exists i' \in \{1, \ldots, d\}, \quad \mu = \varphi_{i'}(q)
\text{ et } \varphi_{i+1}(q) \geq \mu$$
puis, d'après la définition des $\psi_j$ :
$$\exists j' \in \{1, \ldots, d\}, \quad q = \psi_{j'}(\mu) 
\text{ et } \varphi_{i+1}(q) \geq \mu.$$
En remplaçant $q$ par $\psi_{j'}(\mu)$, la dernière inégalité devient
$\varphi_{i+1} \circ \psi_{j'} (\mu) \geq \mu$, ce qui 
équivaut encore à $\mu \geq \mu_{i+1,j'}$. Au final, on
cherche donc à dénombrer les réels $q$ s'écrivant sous la forme
$\psi_{j'}(\mu)$ pour un indice $j'$ tel que $\mu \geq \mu_{i+1,j'}$.
De l'hypothèse supplémentaire $\varphi_1 > \cdots > \varphi_d$, on
déduit facilement que $\psi_1 < \cdots < \psi_d$, d'où il suit que le
nombre cherché est aussi le nombre d'indices $j'$ tels que $\mu \geq
\mu_{i+1,j'}$. Or, comme $\mu$ a été pris dans l'intervalle
$[\mu_{i+1, j+1}, \mu_{i,j}[$ et donc \emph{a fortiori} dans
$[\mu_{i+1, j+1}, \mu_{i+1,j}[$, on déduit de la décroissance de la
suite $j' \mapsto \mu_{i+1,j'}$ que les $j'$ convenables sont ceux de
l'ensemble $\{j+1, \ldots, d\}$. En particulier, il y en a $d-j$.

En résumé, on vient de montrer que le complémentaire de l'image des
$\varphi_i$ est la réunion disjointe des intervalles $[\mu_{i+1, j+1},
\mu_{i,j}[$ pour $j$ variant dans $\{i, \ldots, d\}$, et que si $\mu \in
[\mu_{i+1, j+1}, \mu_{i,j}[$ pour un certain $j$, il y a exactement
$d-j$ indices $i' > i$ tels que $\mu \in \varphi_{i'} (\R)$. Il en
résulte que
$$\sum_{i' = i+1}^d \Leb (\varphi_{i'}(\R) \backslash \varphi_i(\R)) =
\sum_{j=i}^{d-1} (d-j) \cdot (\mu_{i,j} - \mu_{i+1,j+1}).$$
En sommant ces égalités pour tout $i$, il vient :
$$\dim(\varphi) = 
\sum_{1 \leq i \leq j < d} (d-j) \cdot (\mu_{i,j} - \mu_{i+1,j+1}).$$
Les formules annoncées dans la proposition s'en déduisent (avec un peu
de calcul) à partir de la relation \eqref{eq:rele3}.
\end{proof}

\begin{cor}
\label{cor:dim}
Pour tout $\varphi \in \Phi_\Z$, on a la congruence :
$$\dim(\varphi) \equiv - \sum_{j=1}^d j \cdot \mu_{1,j} (\varphi)
\pmod{b-1}.$$
\end{cor}

\begin{proof}
C'est une conséquence immédiate de la première égalité de la proposition
précédente et de la proposition \ref{prop:integrite}.
\end{proof}

Le corollaire est intéressant notamment car, dans le cas où $\varphi$
provient d'un réseau $L$, les nombres $\mu_{1,j} (\varphi)$
s'interprètent comme les exposants des diviseurs élémentaires du module
engendré par $\sigma(L)$ par rapport à $L$. Par ailleurs, comme on peut
s'y attendre, la fonction $\dim$ que l'on vient de définir est liée de
près à la dimension des variétés $\calX_\varphi$. Plus précisément, on
a le théorème suivant.

\begin{theo}
\label{theo:dim}
Soit $\varphi = (\varphi_1, \ldots, \varphi_d) \in \Phi_\Z$. Alors
\begin{itemize}
\item si $h \neq 0$, on a $\dim_k \calX_{\varphi} = 
\dim(\varphi)$ ;
\item si $h = 0$, on a $\dim(\varphi) \leq \dim_k 
\calX_{\varphi} \leq \dim(\varphi) + \frac{d(d-1)} 2$.
\end{itemize}
\end{theo}

\subsubsection{Démonstration du théorème \ref{theo:dim}}

On fixe un élément $\varphi = (\varphi_1, \ldots, \varphi_d) \in
\Phi_\Z$. Puisque $\calX_\varphi$ s'écrit comme l'union finie des
$\tilde \calX_{\tilde \varphi}$ sur les $d$-uplets $\tilde \varphi =
(\tilde \varphi_1, \ldots, \tilde \varphi_d)$ vérifiant les conditions
de la proposition \ref{prop:varphi} et redonnant $\varphi$ après
réordonnement, il suffit de démontrer

\begin{itemize}
\item d'une part, que la dimension de toutes les variétés 
$\tilde \calX_{\tilde \varphi}$ est majorée par $\dim(\varphi)$ 
dans le cas où $h \neq 0$ et par $\dim(\varphi) + 
\frac{d(d-1)} 2$ dans le cas contraire, et
\item d'autre part, qu'il existe un $d$-uplet $\tilde \varphi$
particulier pour lequel $\dim_k \tilde \calX_{\tilde \varphi} \geq
\dim(\varphi)$.
\end{itemize}

\medskip

\noindent 
Pour cela, on suit la méthode de \cite{viehmann}. 

\paragraph{Notion de famille correcte}

On fixe des $d$-uplet $\varphi$ et $\tilde \varphi$ comme précédemment,
et on définit les ensembles suivants :
$$\textstyle Q = \frac 1 b \Z \times \{1, \ldots, d\}
\quad ; \quad  
\tilde V_\mu = \big\{ \, (q,i) \in Q
\,\, \big| \,\, \tilde \varphi_i(q) = \mu
\big\} \, \text{ (pour } \mu \in \Z \text{)}
\quad ; \quad
\tilde V = \bigcup_{\mu \in \Z} \tilde V_\mu$$
$$\textstyle A = \big\{ \, (q,i,q',i') \in Q^2 \,\, \big| \,\, 
(q',i') > (q,i) \text{ et } \varphi_i(q-\frac 1 b) < \varphi_{i'}(q') 
< \varphi_i(q) \, \big\}$$
$$\textstyle \tilde A = \big\{ \, (q,i,q',i') \in Q^2 \,\, \big|
\,\, (q',i') > (q,i) \text{ et } \tilde \varphi_i(q-\frac 1 b) < \tilde
\varphi_{i'} (q') < \tilde \varphi_i(q) \, \big\}$$
et, enfin, pour tout $(q,i) \in \tilde V$ :
$$\textstyle \tilde A(q,i) = \big\{ \, (q',i') \in Q \,\, \big| \,\,
(q',i') > (q,i) \text{ et } \tilde \varphi_i(q-\frac 1 b) < \tilde
\varphi_{i'} (q') < \tilde \varphi_i(q) \, \big\}.$$

\begin{lemme}
\label{lem:cardA}
On a $\card\,\tilde A \leq \card\,A = \dim(\varphi)$.
\end{lemme}

\begin{proof}
On remarque dans un premier temps que si $(q,i,q',i')$
est un élément de $A$, alors on a nécessairement $i < i'$. En effet, on
déduit de $(q,i,q',i') \in A$ que $\varphi_i(q) > \varphi_{i'}(q') \geq
\varphi_{i'}(q)$, ce qui ne peut se produire si $i' \leq i$ étant donné
que $\varphi_1 \geq \cdots \geq \varphi_d$ par hypothèse. On déduit en
particulier de cette propriété que, dans la définition de $A$, on peut
remplacer l'inégalité $(q',i') > (q,i)$ par la condition plus simple $q'
\geq q$.

Pour tout $q \in \frac 1 b \Z$, il existe une permutation $\tau_q \in
\mathfrak S_d$ telle que $\tilde \varphi_i(q) = \varphi_{\tau_q(i)} (q)$
pour tout indice $i$. Pour démontrer que $\card\,\tilde A \leq
\card\,A$, il suffit donc de montrer que, pour tout triplet $(q,q',i')
\in \frac 1 b \Z \times Q$ avec $q' \geq q$, il n'y a pas plus
d'indices $i$ tels que $(q,i,q',i') \in \tilde A$ que d'indices $i$ tels
que $(q,i,q',\tau_{q'}(i')) \in A$. D'après le résultat du premier
alinéa de la démonstration, il suffit pour cela de démontrer que, si on a
posé $\mu = \tilde \varphi_{i'}(q') = \varphi_{\tau_{q'}(i')} (q')$, les
deux ensembles suivants :
$$\textstyle \tilde B = \big\{ \, i \, | \, \tilde
\varphi_i(q-\frac 1 b) < \mu < \tilde \varphi_i(q) \, \big\}
\quad \text{et} \quad
B = \big\{ \, i \, | \, \varphi_i(q-\frac 1 b)
< \mu < \varphi_i(q) \, \big\}$$
ont même cardinal. Or, on peut écrire
$$\textstyle \tilde B = \tilde B_1 \backslash \tilde B_2 
\quad \text{avec} \quad
\tilde B_1 = \big\{ \, i \, | \, \mu < \tilde \varphi_i(q) \, \big\} 
\quad \text{et} \quad
\tilde B_2 = \big\{ \, i \, | \, \mu \leq \tilde \varphi_i(q-\frac 1 b)
\, \big\}$$
et de même $B = B_1 \backslash B_2$ où $B_1$ et $B_2$ sont définis de
manière analogue en remplaçant $\tilde \varphi_i$ par $\varphi_i$. On
a alors les inclusions $\tilde B_2 \subset \tilde B_1$ et $B_2 \subset
B_1$ alors que, par ailleurs, la permutation $\tau_q$ (resp.
$\tau_{q-\frac 1 b}$) induit une bijection de $\tilde B_1$ dans $B_1$
(resp. de $\tilde B_2$ dans $B_2$). La conclusion s'ensuit.

Il reste à démontrer que $\card\,A = \dim(\varphi)$. Pour cela, on fixe
deux entiers $i$ et $i'$ avec $i < i'$. L'ensemble différence
$\varphi_{i'} (\R) \backslash \varphi_i(\R)$ s'écrit comme une union
disjointe d'intervalles de la forme $[\mu, \mu+1[$ pour certains entiers
$\mu$. De plus, si $\mu$ est un tel entier (\emph{i.e.} si $\mu \in
\varphi_{i'} (\R) \backslash \varphi_i(\R)$), il existe $q, q' \in
\frac 1 b \Z$ tels que $\varphi_{i'} (q') = \mu$
et $(q,i,q',i') \in A$ (on rappelle que, dans la définition de
l'ensemble $A$, on peut remplacer $(q',i') > (q,i)$ par $q' \geq q$). En
outre, les rationnels $q$ et $q'$ sont uniquement déterminés. À partir
de la définition de $\dim(\varphi)$ (voir définition \ref{def:dim}), on
en déduit que $\dim(\varphi)$ compte le nombre de quadruplets
$(q,i,q',i') \in A$ tels que $i < i'$, c'est-à-dire le nombre d'éléments
de $A$ puisque l'on a démontré que tout $(q,i,q',i') \in A$ vérifie la
condition supplémentaire $i < i'$.
\end{proof}

On définit $\tilde q_i$ comme le nombre associé aux fonctions $\tilde
\varphi_i$ et on note $(e_1, \ldots, e_d)$ la base canonique de $M$. La
première étape de la preuve consiste à démontrer qu'étant donné un
réseau $L$ de $M$ tel que $\tilde \varphi_i(L) = \tilde \varphi_i$ pour
tout $i$, il existe des éléments $v_{q,i} \in M_{k((u^{1/b}))}$ pour
$(q,i) \in \tilde V$ et des élements $a_{q,i,q',i'} \in k$ pour $(q,i,q',i')
\in \tilde A$ qui vérifient :
\begin{itemize}
\item[i)] pour tout $(q,i) \in \tilde V$, on a $\val(v_{q,i}) = (q,i)$ et
$\val(v_{q,i} - u^q e_i) > (q,i)$ ;
\item[ii)] pour tout $(q,i) \in \tilde V$, on a $v_{q,i} \in k[[u^{1/b}]]
\otimes_{k[[u]]} L$ et $w_{q,i} = u^{-\tilde \varphi_i(q)} \sigma(v_{q,i})
\in L$ ;
\item[iii)] pour tout nombre entier $\mu$, les éléments $(w_{q,i} \mod
u)$ pour $(q,i)$ parcourant $\tilde V_\mu$ forment une famille libre sur
$k$ dans $L/uL$ ;
\item[iv)] pour tout $(q,i) \in \tilde V$ tel que $(q-\frac 1 b, i) \in 
\tilde V$, on a
\begin{equation}
\label{eq:viehmann1}
v_{q, i} = u^{1/b} v_{q-\frac 1 b,i} + \sum_{(q',i') \in \tilde A(q,i)}
a_{q,i,q',i'} \cdot v_{q',i'}
\end{equation}
\item[v)] pour tout $(q,i) \in \tilde V$ tel que $\tilde \varphi_i(q) = bq - 
\tilde q_i$ (ou de façon équivalente pour un tel $(q,i)$), on a
\begin{equation}
\label{eq:viehmann2}
v_{\tilde q_i, i} = w_{q,i} +
\sum_{(q',i') \in \tilde A(q_i,i)} a_{q_i,i,q',i'}^{p^h} \cdot v_{q',i'}
- \sum_{\substack{(q',i') \in \tilde A(q,i) \\ q' = q}} 
a_{q,i,q,i'}^{p^h} \cdot u^{\tilde q_i} e_{i'}
\end{equation}
où $w_{q,i} = u^{-\tilde \varphi_i(q)} \sigma(v_{q,i})$ comme ci-dessus.
\end{itemize}

\medskip

Une famille $(v_{q,i}, a_{q,i,q',i'})$ est dite \emph{correcte} pour $L$
si elle vérifie les conditions précédentes et, étant donné un entier
$n$, elle est dite \emph{$n$-correcte} si elle vérifie i), ii), iii) et
si les égalités iv) et v) sont vraies respectivement modulo $u^{q+\frac
n b}$ et $u^{\tilde q_i + \frac n b}$. On va construire une famille
correcte en procédant par approximations successives. Pour amorcer la
construction, on se donne des éléments $v_{q,i}$ vérifiant simplement
les conditions i), ii) et iii) ; leur existence résulte de la définition
des fonctions $\tilde \varphi_i(L)$ et, en ce qui concerne iii), d'une
analyse de la démonstration du lemme 4.1 de \cite{viehmann}. On choisit
également $a_{q,i,q',i'} = 0$ pour tout $(q,i,q',i') \in \tilde A$. La famille
$(v_{q,i}, a_{q,i,q',i'})$ est alors $0$-correcte. L'étape d'itération
est donnée par le lemme suivant duquel il résulte directement
l'existence souhaitée après un passage à la limite.

\begin{lemme}
Soit $(v_{q,i}, a_{q,i,q',i'})$ une famille $n$-correcte pour $L$ pour
un certain entier $n\geq 0$. On pose $m = n+1$. Alors, il existe
$(v'_{q,i}, a'_{q,i,q',i'})$ une famille $m$-correcte pour $L$ telle que
\begin{itemize}
\item on ait $v'_{q,i} \equiv v_{q,i} \pmod {u^{q + \frac n b}}$ pour
tout $(q,i) \in \tilde V$;
\item si $n$ est suffisamment grand, on ait aussi $a'_{q,i_0,q',i'} =
a_{q,i_0,q',i'}$ pour tout $(q,i_0,q',i') \in \tilde A$.
\end{itemize}
\end{lemme}

\begin{proof}
On construit les $v'_{q,i}$ par récurrence sur $q$ et à $q$ fixé par
récurrence descendante sur $i$. Autrement dit, on considère $(q,i) \in
\tilde V$, on suppose que tous les $v'_{q',i'} \in \tilde V$ avec $q' <
q$ ou $q' = q$ et $i' > i$ sont construits et on cherche à construire
$v'_{q,i}$.  On suppose d'abord que $(q-\frac 1 b, i)$ appartient à 
$\tilde V$ et on regarde dans ce cas l'équation \eqref{eq:viehmann1} 
modulo $u^{q+\frac m b}$, \emph{i.e.} la congruence
\begin{equation}
\label{eq:congrcorrect}
v'_{q, i} \equiv u^{1/b} \cdot v'_{q-\frac 1 b,i} + \sum_{(q',i') \in
\tilde A(q,i)} a'_{q,i,q',i'} \cdot v'_{q',i'} \pmod{u^{q+\frac m b}}
\end{equation}
qui doit être satisfaite par l'élément $v'_{q,i}$ que l'on veut
construire. Dans l'expression précédente, les $v'_{q-\frac 1 b,i}$ ont
déjà été construits de même que les $v'_{q',i'}$ pour $q = q'$ car on a
alors nécessairement $i'> i$. Si $q' > q$, en revanche, on n'a pas
encore construit $v'_{q', i'}$ mais on souhaite le faire de façon à ce
que $v'_{q',i'} \equiv v_{q',i'} \pmod {u^{q' + \frac n b}}$ et donc
\emph{a fortiori} $v'_{q',i'} \equiv v_{q',i'} \pmod {u^{q + \frac m
b}}$. On cherche donc à ce que $v'_{q',i'}$ satisfasse la congruence
\eqref{eq:congrcorrect} où on a remplacé $v'_{q',i'}$ pour $q' > q$ par
$v_{q',i'}$. En fait, on va chercher $v'_{q',i'}$ de sorte que cette
nouvelle congruence soit une égalité, c'est-à-dire de sorte que
\begin{equation}
\label{eq:egalcorrect}
v'_{q, i} = u^{1/b} \cdot v'_{q-\frac 1 b,i} + \sum_{(q',i') \in \tilde
A(q,i)} a'_{q,i,q',i'} \cdot x_{q',i'}
\end{equation}
où, pour unifier les écritures, on a posé, pour $(q',i') \in \tilde
A(q,i)$, $x_{q',i'} = v'_{q',i'}$ si $q' = q$ et $x_{q',i'} = v_{q',i'}$
sinon. Par ailleurs, en plus de cela, on doit avoir $\sigma(v'_{q,i}) 
\in u^{\tilde \varphi_i(q)} L$. En reportant la valeur désirée de 
$v'_{q,i}$ donnée par l'égalité \eqref{eq:egalcorrect}, on est amené 
à démontrer qu'il existe $a'_{q,i,q',i'} \in k$ tels que :
\begin{equation}
\label{eq:appcorrect}
\sigma\Bigg(
u^{1/b} \cdot v'_{q-\frac 1 b,i} + \sum_{(q',i') \in \tilde
A(q,i)} a'_{q,i,q',i'} \cdot x_{q',i'}
\Bigg) \in u^{\tilde \varphi_i(q)} L
\end{equation}
Par définition de $\tilde \varphi_i(q)$, on sait qu'il existe dans $L$
un élément $x$ de valuation $(q,i)$ tel que $\sigma(x) \in u^{\tilde
\varphi_i(q)}$. Quitte à multiplier $x$ par une constante dans $k$, on
peut en outre supposer que $\val(x - u^{1/b} v'_{q-\frac 1 b,i}) >
(q,i)$. On en déduit, en utilisant la condition i), que $x - u^{1/b}
v'_{q-\frac 1 b,i}$ s'écrit comme une somme infinie, portant sur tous
les couples $(q',i')$ strictement supérieurs à $(q,i)$, de termes de la
forme $a'_{q,i,q',i'} v_{q',i'}$ avec $a'_{q,i,q',i'} \in k$. Puisque
$\sigma(x_{q',i'}) \in u^{\tilde \varphi_{i'}(q')} L$, on peut, quitte à
changer $x$ en un autre élément de valuation $(q,i)$ tel que $\sigma(x)
\in u^{\tilde \varphi_i(q)}$, retirer de la somme précédente les
contributions apportées par les couples $(q',i')$ tels que $\tilde
\varphi_{i'}(q') \geq \tilde \varphi_i(q)$. La somme restante est alors
finie, car il n'existe de toute façon qu'un nombre fini de couples
$(q',i') \in \tilde V$ tels que $\tilde \varphi_{i'}(q') < \tilde
\varphi_i(q)$. Pour conclure, il ne reste plus qu'à démontrer que les
$a'_{q,i,q',i'}$ sont nécessairement nuls dès que $\tilde
\varphi_{i'}(q') \leq \tilde \varphi_i(q - \frac 1 b)$. On part pour
cela de la relation
\begin{equation}
\label{eq:appcorrect2}
\sigma(x) = u \cdot \sigma(v'_{q-\frac 1 b,i}) + 
\sum_{\substack{(q',i') > (q,i) \\ 
\tilde\varphi_{i'}(q') < \tilde \varphi_i(q)}} 
(a'_{q,i,q',i'})^{p^h} \cdot \sigma(x_{q',i'})
\in u^{\tilde \varphi_i(q)} L.
\end{equation}
obtenue simplement en développant. On définit également la valuation
$L$-adique $\val_L(v)$ d'un élément $v \in M$ comme le plus grand entier
$n$ tel que $v \in u^n L$. La valuation $L$-adique d'un élément de la
forme $\sigma(v_{q',i'})$ ou $\sigma(v'_{q',i'})$ est alors égale à
$\tilde \varphi_{i'}(q')$ : en effet, elle est supérieure ou égale à
cette valeur d'après la condition ii) et l'inégalité ne peut être
stricte par définition de $\tilde \varphi_{i'}(q')$. Par ailleurs, à
partir de la condition iii), il est facile de montrer que la valuation
$L$-adique d'une somme de termes de la forme $c_{q',i'}
\sigma(x_{q',i'})$ (avec $c_{q',i'} \in k$) est toujours égale au
minimum des valuations $L$-adiques des $c_{q',i'} \sigma(x_{q',i'})$.
Ainsi aucun terme de la somme qui apparaît dans \eqref{eq:appcorrect2}
ne peut avoir une valuation $L$-adique strictement inférieure à $\val_L(
u \cdot \sigma(v'_{q-\frac 1 b,i})) = 1 + \tilde \varphi_i(q - \frac 1
b)$, ce qui implique finalement ce qu'il fallait démontrer.

Si maintenant, au contraire, $(q-\frac 1 b, i) \not\in \tilde V$, on
raisonne de manière similaire sauf que l'on part désormais de l'équation
\eqref{eq:viehmann2} et comme précédemment on remplace $w'_{q,i} =
u^{-\tilde \varphi_i(q)} \sigma(v'_{q,i})$ par $w_{q,i} = u^{-\tilde
\varphi_i(q)} \sigma(v_{q,i})$. Après cela, il n'est plus difficile de
vérifier que $v'_{q,i} \equiv v_{q,i} \pmod {u^{q + \frac n b}}$ pour
tout $(q,i) \in \tilde V$ et que la famille $(v'_{q,i}, a'_{q,i,q',i'})$ que
nous avons construite est bien $m$-correcte. Il reste à montrer que si
$n$ est suffisamment grand, on a $a_{q,i,q',i'} = a'_{q,i,q',i'}$ pour
tout quadruplet $(q,i,q',i') \in \tilde A$. Mais cela résulte
directement du fait que, en vertu des congruences $v'_{q,i} \equiv
v_{q,i} \pmod {u^{q + \frac n b}}$ pour tout $(q,i) \in \tilde V$, l'assertion
\eqref{eq:appcorrect}, de même que son analogue dans le cas où $(q-\frac
1 b, i) \not\in \tilde V$, est vraie avec $a'_{q,i,q',i'} = a_{q,i,q',i'}$ si
$n$ est suffisamment grand.
\end{proof}

% Le dernier paragraphe de la preuve est peut-être à détailler un peu
% plus

\begin{rem}
\label{rem:unicite}
En reprenant l'argument de la démonstration de l'existence des
$a'_{q,i,q',i'}$, on voit que ceux-ci sont en fait uniquement
déterminés. Il résulte de cette remarque, par passage à la limite, que
les éléments $a_{q,i,q',i'}$ dans une famille correcte pour $L$ sont,
entièrement déterminés par $L$.
\end{rem}

Il est également possible à partir d'une famille correcte $(v_{q,i},
a_{q,i,q',i'})$ de retrouver le réseau $L$ : en effet, à partir des
conditions i) et ii), on démontre directement que $L$ est le module
engendré par les $v_{q,i}$ pour $(q,i)$ parcourant $V$. Le lemme suivant
montre qu'en fait $L$ est déjà engendré par les $d$ vecteurs $v_{\tilde
q_i, i}$ (qui en forment donc une base).

\begin{lemme}
\label{lem:vqss}
Soit $(v_{q,i}, a_{q,i,q',i'})$ une famille correcte pour $L$.
Pour tout $(q,i) \in \tilde V$, il existe des $\lambda_s \in k[[u]]$ ($1 \leq 
s \leq d$) tels que :
$$v_{q,i} = \sum_{s=1}^d \lambda_s \cdot u^{q-\tilde q_s} v_{\tilde
q_s, s}$$
et $\lambda_s = 0$ si $q < \tilde q_s$ ou si $q = \tilde q_s$ et 
$s < i$.
\end{lemme}

\begin{proof}
Si $(q-\frac 1 b,i) \not \in \tilde V$, alors $q = \tilde q_i$ et le résultat
est clair. Dans le cas contraire, la relation \eqref{eq:viehmann1} 
assure que $v_{q,i}$ s'exprime en termes de $v_{q-\frac 1 b,i}$ et des
$v_{q',i'}$ pour $(q',i') \in \tilde A(q,i)$. Or, les nombres $\tilde
\varphi_i(q- \frac 1 b)$ et $\varphi_{i'}(q')$ sont tous strictement
plus petits que $\varphi_i(q)$. Une récurrence sur le nombre $\tilde 
\mu = \tilde \varphi_i(q)$ permet donc de terminer la démonstration
(on remarque que l'initialisation ne pose pas de problème car si
$\tilde \mu$ est suffisamment petit, aucun couple $(q,i)$ ne
convient).
\end{proof}

\begin{rem}
\label{rem:vqss}
Un examen de la démonstration précédente indique, en outre, que les
$\lambda_s$ s'expriment uniquement en fonction des $a_{q,i,q',i'}$ et
des propriétés combinatoires des fonctions $\tilde \varphi_i$.
\end{rem}

\paragraph{L'espace des familles correctes}

Dans ce paragraphe, on explique comment l'invariant \og famille correcte
\fg\ permet de paramétrer les réseaux $L$. À partir de maintenant, nous
ne fixe donc plus un réseau $L$ mais, au contraire, on considère
l'ensemble $\calC(k)$ des familles $(v_{q,i}, a_{q,i,q',i'})$
satisfaisant les conditions i), iv) et v) précédemment énoncées.
Manifestement, $\calC(k)$ est l'ensemble des $k$-points d'une variété
algébrique définie sur $k$ que l'on note $\calC$.

\begin{lemme}
\label{lem:henselvqi}
On suppose que $(v_{q,i}, a_{q,i,q',i'})$ et $(v'_{q,i},
a_{q,i,q',i'})$ (avec les mêmes $a_{q,i,q',i'}$) vérifient
les conditions i), iv) et v), et que pour tout $i \in \{1, \ldots,
d\}$, on a $v_{\tilde q_i,i} \equiv v'_{\tilde q_i,i} \pmod {u^{\tilde
q_i + \frac 1 b}}$.
Alors $v_{q,i} = v'_{q,i}$ pour tout $(q,i) \in \tilde V$.
\end{lemme}

\begin{proof}
On pose, pour simplifier les écritures, $v_i = u^{-\tilde q_i} v_{\tilde
q_i,i}$ et de même $v'_i = u^{-\tilde q_i} v'_{\tilde q_i,i}$.
L'hypothèse s'écrit alors $v_i \equiv v'_i \pmod {u^{1/b}}$. En
utilisant l'égalité \eqref{eq:viehmann2}, le lemme \ref{lem:vqss} 
ainsi que la remarque \ref{rem:vqss}, on
voit qu'il existe des matrices $G$ et $H$ à coefficients dans $k[[u]]$
de taille respectivement $d \times d$ et $1 \times d$ telles que
$$(\sigma(v_1), \ldots, \sigma(v_d)) = (v_1, \ldots, v_d) G + H
\quad \text{et} \quad
(\sigma(v'_1), \ldots, \sigma(v'_d)) = (v'_1, \ldots, v'_d) G + H.$$
De plus, on vérifie que la matrice $G$ s'écrit $I_d + G'$ où $G'$ est
topologiquement nilpotente ; en particulier $G$ est inversible. En
posant $w_i = v_i - v'_i$, on a par hypothèse $w_i \equiv 0 \pmod
{u^{1/b}}$ et, d'après ce qui précède, $(\sigma(w_1), \ldots, \sigma(w_d)) =
(w_1, \ldots, w_d) G$. Comme $G$ est inversible, cela implique que $w_i
\equiv 0 \pmod u$. En répétant l'argument, on obtient $w_i \equiv 0
\pmod {u^{b^n}}$ pour tout $n$, c'est-à-dire $w_i = 0$. Finalement, $v_i
= v'_i$ et une nouvelle application du lemme \ref{lem:vqss} permet de
conclure.
\end{proof}

\begin{rem}
Le lemme précédent reste vrai si $k$ est remplacé par une $k$-algèbre
quelconque.
\end{rem}

Soit $\calA = \A_k^{\tilde A}$ l'espace affine standard sur $k$ dont les
coordonnées sont indicées par l'ensemble $\tilde A$ ; c'est une variété
algébrique de dimension $\card\,\tilde A$. On dispose par ailleurs d'un
morphisme naturel $f : \calC \to \calA$ qui à une famille $(v_{q,i},
a_{q,i,q',i'})$ associe le vecteur de $a_{q,i,q',i'}$. Le lemme
précédent et la remarque qui le suit montrent que les fibres de $f$ sont
de dimension inférieure ou égale à $\frac{d(d-1)} 2$. En outre, lorsque
$h \neq 0$, un examen de la preuve du lemme \ref{lem:henselvqi} montre
même que $f$ est étale. Ainsi, si l'on pose $\varepsilon = 1$ si $h = 0$
et $\varepsilon = 0$ dans le cas contraire, on obtient :
\begin{equation}
\label{eq:majdim}
\dim_k \calC \leq \dim_k \calA + \varepsilon \cdot \frac{d(d-1)} 2
= \card\, A + \varepsilon \cdot \frac{d(d-1)} 2 \leq \dim(\varphi) +
\varepsilon \cdot \frac{d(d-1)} 2.
\end{equation}
Par ailleurs, l'application qui à une famille $(v_{q,i}, a_{q,i,q',i'}) 
\in \calC(k)$ associe le réseau engendré par les $v_{q,i}$
définit un morphisme algébrique $g$ de $\calC$ dans la grassmanienne
affine sur $k$. Le fait que tout réseau $L$ appartenant à $\tilde
\calX_{\tilde \varphi}(k)$ admette une famille correcte signifie que
l'image de $g$ contient $\tilde \calX_{\tilde \varphi}$. Ainsi 
on obtient $\dim_k \tilde \calX_{\tilde \varphi} \leq \dim_k \calC$
et la majoration que l'on voulait suit alors de \eqref{eq:majdim}.

\paragraph{Démonstration de la minoration}

On se place ici dans le cas où $\tilde \varphi_1 \geq \tilde \varphi_2
\geq \cdots \geq \tilde \varphi_d$ et on souhaite montrer qu'alors
$\dim_k \tilde \calX_{\tilde \varphi} \geq \dim(\varphi)$. Dans ce cas
particulier, les ensembles $A$ et $\tilde A$ coïncident et, d'après le
lemme \ref{lem:cardA}, leur cardinal vaut $\dim(\varphi)$.

\begin{lemme}
Le morphisme $f : \calC \to \calA$ défini précédemment est un
isomorphisme.
\end{lemme}

\begin{proof}
Soit $R$ une $k$-algèbre. Il s'agit de montrer que pour tout
$(a_{q,i,q',i'}) \in \calA(R)$. Il existe une unique famille
$(v_{q,i})_{(q,i) \in \tilde V}$ d'éléments de $M \otimes_{k((u))}
R((u^{1/b}))$ telle que $(v_{q,i}, a_{q,i,q',i'}) \in \calC(R)$.

On construit les $v_{q,i}$ et on démontre leur unicité par récurrence
descendante sur $i$. En reprenant la démonstration du lemme
\ref{lem:henselvqi}, on voit que l'élément $v_i = u^{-\tilde q_i}
v_{\tilde q_i,i}$ doit satisfaire une équation de la forme
$$\sigma(v_i) = v_i + \sum_{s = i+1}^d \lambda_s v_s$$
où les $\lambda_s$ sont de valuation $u$-adique strictement positive
(les éventuels termes de valuation $0$ s'annulent avec le terme
$\sum_{(q',i') \in \tilde A(q,i), \, q' = q} a_{q,i,q,i'}^{p^h} \cdot
u^{\tilde q_i} e_{i'}$ de la formule \eqref{eq:viehmann2}). Comme les
$v_s$ pour $s > i$ sont déjà connus, on a à résoudre une équation de la
forme $\sigma(v_i) = v_i - c$ où $c$ est un élément connu de valuation
strictement positive de $R[[u^{1/b}]]$. Une telle équation a bien une
unique solution, à savoir $v_i = \sum_{n=0}^\infty \sigma^n(c)$. Les
$v_{q,i}$ pour $q > q_i$ s'obtiennent alors à partir d'une variante du
lemme \ref{lem:vqss}.
\end{proof}

À présent, il est aisé de conclure. La remarque \ref{rem:unicite} montre
que le morphisme $g : g^{-1}(\tilde \calX_{\tilde \varphi}) \to \tilde
\calX_{\tilde \varphi}$ est également un isomorphisme. Il suffit donc de
démontrer que $g^{-1}(\tilde \calX_{\tilde \varphi})$ est un ouvert non
vide de $\calC$, ce qui se fait comme dans la preuve du \emph{Claim} 3
de \cite{viehmann}.

\section{Mise en place de la méthode}
\label{sec:methode}

Les théorèmes \ref{theo:parametrisation} et \ref{theo:dim} permettent de 
reformuler le problème de calculer --- ou disons, plutôt d'estimer --- 
la dimension de $\calX_{\leq e}$ en un problème de programmation 
linéaire. Dans cette section, nous mettons en place les outils 
nécessaires à la résolution de ce dernier problème puis, en guise 
d'exemple, nous illustrons la méthode proposée en démontrant le théorème 
\ref{theo:dimcaruso} de l'introduction sous une hypothèse additionnelle. 
La démonstration complète de ce théorème est reportée à la section 
suivante, \S \ref{subsec:demcaruso}.

\subsection{Préliminaires de programmation linéaire}
\label{subsec:opticonv}

On considère un espace euclidien $E$ dont on
note $\left< \cdot | \cdot \right>_E$ le produit scalaire. On se donne :
\begin{itemize}
\item un cône convexe $Q \subset E$, c'est-à-dire un sous-ensemble non
vide de $E$ stable par addition et par multiplication par les nombres
réels positifs ou nuls ;
\item une application linéaire $f : E \to \R^n$ où $n$ est un certain
entier naturel ;
\item une forme linéaire $\ell : E \to \R$.
\end{itemize}
Étant donné que $E$ est un espace euclidien, il existe des vecteurs
$\vec \ell, \vec f_1, \ldots, \vec f_n$ tels que $\ell(x) = \left<
x | \vect \ell \right>_E$ et $f(x) = (\left< x | \vect f_1
\right>_E, \ldots, \left< x | \vect f_n \right>_E)$ pour tout $x \in
E$.
On souhaite étudier la fonction $a_{Q,f,\ell} : \R^n \to
\R \cup \{\pm \infty\}$ définie par :
$$a_{Q,f,\ell}(y) = \sup_{\substack{x \in Q \\ f(x) = y}} \ell (x) =
\sup_{\substack{x \in Q \\ \left< x | \vect f_i \right>_E = y_i}}
\left< x | \vect \ell \right>_E \qquad \text{où } y = (y_1,
\ldots, y_n)$$
en convenant, comme d'habitude, que la borne supérieure de l'ensemble
vide est $-\infty$ et celle d'un ensemble non majorée est $+\infty$.
Soit $Q^\star$ le cône dual de $Q$ :
$$Q^\star = \big\{ \, x \in E \,\, | \,\, \left< x | x' \right>_E 
\geq 0, \quad \forall x' \in Q \, \big\}.$$
Si $Q$ est défini par les inégalités $\left< x |
\vec v_i \right>_E \geq 0$ ($1 \leq i \leq N$), alors $Q^\star$ est
le cône convexe engendré par les vecteurs $v_i$, c'est-à-dire 
l'ensemble des vecteurs de la forme $\lambda_1 \vec v_1 + \cdots +
\lambda_N \vec v_N$ pour des scalaires $\lambda_i \in \R^+$.
On introduit $A_{Q,f,\ell}$ l'ensemble convexe défini par :
\begin{equation}
\label{eq:AQfl}
A_{Q,f,\ell} = \big\{ \, y = (y_1, \ldots, y_n) \in \R^n \,\, | \,\, 
(y_1 \vec f_1 + \cdots + y_n \vec f_n) - \vec \ell \in Q^\star 
\,\big\}.
\end{equation}
Le théorème suivant établit un lien de dualité entre maximisation
sur $Q$ et minimisation sur $A_{Q,f,\ell}$.

\begin{theo}
\label{theo:dualconv1}
Avec les notations précédentes, on a :
$$a_{Q,f,\ell}(y) = \inf_{\alpha \in A_{Q,f,\ell}} \left< \alpha | y
\right>_n$$
où $\left< \cdot | \cdot \right>_n$ désigne le produit scalaire usuel
sur $\R^n$.
\end{theo}

\begin{proof}
Il s'agit un résultat classique de dualité en programmation linéaire.
On rappelle quand même brièvement comment on l'établit. On remarque tout
d'abord que la fonction $a_{Q,f,\ell}$ est concave. Le théorème
de Hahn Banach assure qu'elle s'écrit comme la borne inférieure
des fonctions affines qui la majorent. Or un calcul immédiat montre 
que la fonction affine $\R^n \to \R$, $(y_1, \ldots, y_n) \mapsto 
\alpha_1 y_1 + \cdots + \alpha_n y_n + \beta $ ($\alpha_i, \beta \in 
\R$) majore $a_{Q,f,\ell}$ si, et seulement si 
$$\forall x \in Q, \qquad \left< x | \vect \ell - (\alpha_1 
\vect f_1 + \cdots + \alpha_n \vect f_n) \right>_E \leq \beta.$$
Comme $Q$ est un cone convexe, ceci est encore équivalent à $\beta \geq
0$ et $(\alpha_1 \vec f_1 + \cdots + \alpha_n \vec f_n) - \vec 
\ell \in Q^\star$. Le théorème en résulte.
\end{proof}

Pour ce que l'on veut faire, on aura besoin de travailler dans une
situation légèrement plus générale que celle qui vient d'être étudiée.
Précisément, en plus de $Q$, $f$ et $\ell$, on se donne maintenant deux
cônes convexes $C$ et $D$ inclus dans $\R^n$, et on considère la
fonction $b_{Q,f,\ell,C,D} : \R^n \to \R \cup \{\pm \infty\}$ définie
par :
$$\begin{array}{rcll}
b_{Q,f,\ell,C,D}(y) & = & \sup_{x \in Q, \, f(x) \in y - C} \ell (x) & 
\text{si } y \in D \\
& = & -\infty & \text{sinon}
\end{array}$$
où par définition $y - C$ est l'ensemble des vecteurs $y' \in \R^n$ pour
lesquels il existe $c \in C$ tel que $y - c = y'$, ou autrement dit $y -
y' \in C$. Dans la suite, lorsque $D = \R^n$, on s'autorisera à ne pas
le noter en indice. On note $C^\star$ et $D^\star$ les
cônes duaux respectifs de $C$ et $D$. On pose : $$B_{Q,f,\ell,C,D} =
(A_{Q,f,\ell} \cap C^\star) + D^\star$$ où la notation précédente
signifie que les éléments de $B_{Q,f,\ell,C,D}$ sont ceux qui s'écrivent
sous la forme $y_1 + y_2$ avec $y_1 \in A_{Q,f,\ell} \cap C^\star$ et
$y_2 \in D^\star$. 

\begin{prop}
\label{prop:dualconv2}
Avec les notations précédentes, on a :
$$b_{Q,f,\ell,C,D}(y) = \inf_{\alpha \in B_{Q,f,\ell,C,D}}
\left< \alpha | y \right>_n.$$
\end{prop}

\begin{proof}
On commence par traiter le cas où $D = \R^n$. Alors $D^\star = \{0\}$ et
$B_{Q,f,\ell,C,D} = A_{Q,f,\ell} \cap C^\star$. La démonstration suit les
mêmes idées que celle du théorème \ref{theo:dualconv1}. On
commence par montrer que la fonction $b_{Q,f, \ell,C,D}$ est concave.
Par le théorème de Hahn Banach, elle s'écrit donc comme la borne
inférieure des fonctions affines qui la majorent. Or, la fonction affine
$\R^n \to \R$, $(y_1, \ldots, y_n) \mapsto \alpha_1 y_1 + \cdots +
\alpha_n y_n + \beta $ ($\alpha_i, \beta \in \R$) majore $b_{Q,f,\ell,
C,D}$ si, et seulement si
\begin{equation}
\label{eq:dualconv}
\forall x \in Q, \, \, \forall c \in C, \qquad
\left< x | \vect \ell - (\alpha_1 \vect f_1 + \cdots + 
\alpha_n \vect f_n) \right>_E \leq \beta + \left< \alpha | c \right>_n.
\end{equation}
Si $\alpha \in C^\star$, le produit scalaire $\left< \alpha | c 
\right>_n$ est par définition toujours positif ou nul. La valeur 
minimale qu'il prend lorsque $c$ décrit $C$ est donc $0$. Ainsi, la
condition précédente est équivalente à celle qui apparaissait dans la
démonstration du théorème \ref{theo:dualconv1}, soit encore à
$\beta \geq 0$ et $\alpha \in A_{Q,f,\ell}$.
Si, au contraire, $\alpha \not\in C^\star$, alors il existe un 
vecteur $c_0 \in C$ tel que $\left< \alpha | c_0 \right>_n < 0$. Comme 
$C$ est supposé stable par multiplication par les éléments de $\R^+$, 
la quantité $\left< \alpha | c \right>_n$ est non minorée lorsque $c$ 
décrit $C$ et la condition \eqref{eq:dualconv} n'est jamais satisfaite 
dans ce cas. La proposition, dans le cas particulier $D = \R^n$, 
résulte de ces considérations.

On en vient maintenant au cas général. D'après ce que l'on vient de
faire, il suffit d'établir que :
$$\begin{array}{rcll}
\inf_{\alpha \in B_{Q,f,\ell,C,D}} \left< \alpha | y \right>_n 
& = & \inf_{\alpha \in A_{Q,f,\ell} \cap C^\star} \left< \alpha | y 
\right>_n & \text{si } y \in D \\
& = & -\infty & \text{sinon}.
\end{array}$$
On suppose d'abord que $y \in D$. Alors, si $\alpha$ est un élément de
$B_{Q,f,\ell,C,D}$, il s'écrit $\alpha = \alpha_1 + \alpha_2$
avec $\alpha_1 \in A_{Q,f,\ell} \cap C^\star$ et $\alpha_2 \in
D^\star$, d'où il suit 
$\left< \alpha | y \right>_n = \left< \alpha_1 | y \right>_n +
\left< \alpha_2 | y \right>_n \geq \left< \alpha_1 | y \right>_n$.
En passant à la borne inférieure, on obtient l'inégalité
$\inf_{\alpha \in B_{Q,f,\ell,C,D}} \left< \alpha | y \right>_n
\geq \inf_{\alpha \in A_{Q,f,\ell} \cap C^\star} \left< \alpha | y 
\right>_n$. Mais l'inégalité dans l'autre sens est évidente puisque
$B_{Q,f,\ell,C,D}$ contient $A_{Q,f,\ell} \cap C^\star$.
Si maintenant $y \not\in D$, le théorème d'Hahn Banach assure qu'il
existe $z \in D^\star$ tel que $\left< z | y \right>_n < 0$.
Les vecteurs $\lambda z$, pour $\lambda \in \R^+$
appartiennent alors tous à $B_{Q,f,\ell,C,D}$, ce qui assure que la
quantité $\left< \alpha | y \right>_n$ est non minorée lorsque $\alpha$
parcourt cet ensemble. On a donc bien démontré ce que l'on voulait dans
tous les cas.
\end{proof}

Un cas important est celui où le cône convexe $Q$ est défini comme
l'intersection d'un nombre \emph{fini} de demi-espaces, ce qui est la
situation que l'on considèrera dans la suite. L'ensemble
$B_{Q,f,\ell,C,D}$ est alors un polytope (éventuellement non borné) qui
n'a, en tout cas, qu'un nombre fini de sommets. En outre, si on note
$\alpha_1, \ldots, \alpha_N$ ceux qui restent à distance finie, il
découle de la proposition précédente que :
\begin{equation}
\label{eq:bfini}
\begin{array}{rcll}
b_{Q,f,\ell,C,D}(y) & = & \inf_{1 \leq i \leq N} \left< \alpha_i |
y \right>_n & \text{si } y \in D \cap (f(Q)+C) \\
& = & -\infty & \text{sinon}.
\end{array}
\end{equation}
Ainsi, déterminer la fonction $b_{Q,f,\ell,C,D}$ revient à déterminer
les $\alpha_i$.

\subsubsection*{Un peu de réseaux pour pimenter}

On conserve les notations introduites précédemment, et on se donne en
outre $R$ un réseau de $E$, c'est-à-dire un sous-groupe additif de $E$
engendré par une base de $E$. On définit une nouvelle fonction
$b'_{Q,R,f,\ell,C,D} : \R^n \to \R$ par
$$\begin{array}{rcll}
b'_{Q,R,f,\ell,C,D}(y) & = & \sup_{x \in Q \cap R, \, f(x) \in y - C}
\ell (x) & \text{si } y \in D \\
& = & -\infty & \text{sinon}.
\end{array}$$ 
De même que précédemment, lorsque $D = \R^n$, on l'omettra dans la
notation. Il est évident que la fonction $b'_{Q,R,f,\ell,C,D}$ est
majorée par $b_{Q,f,\ell,C,D}$ puisque la borne supérieure pour définir
cette dernière fonction est prise sur un ensemble plus gros.

\begin{prop}
\label{prop:bprime}
On suppose que $Q$ engendre $E$ en tant qu'espace vectoriel, que
l'application $f$ est surjective, et finalement que $R \cap f^{-1}(C)$
et $f^{-1}(C)$ engendrent le même espace vectoriel dans $E$.

Alors, il existe un vecteur $y_0 \in \R^n$ et une constante $c \in \R$
tels que, pour tout $y \in f(R) + C$, on ait :
$$b_{Q,f,\ell,C}(y-y_0) - c \leq b'_{Q,R,f,\ell,C}(y) \leq
b_{Q,f,\ell,C}(y)$$
où, bien sûr, on convient que $-\infty - c = -\infty$ et $+\infty - c =
+\infty$.
\end{prop}

\begin{rem}
Le fait que $Q$ engendre $E$ n'est évidemment pas vraiment
contraignant puisque, dans le cas où cela ne serait pas vérifié, il
suffit de remplacer $E$ par le sous-espace vectoriel $Q_\R$ engendré par
$Q$ et le réseau $R$ par $R \cap Q_\R$. On attire toutefois l'attention
du lecteur sur le fait qu'il se peut que $R \cap Q_\R$ ne soit pas un
réseau de $Q_\R$ ; il s'agit donc d'une question qu'il ne faudra pas
oublier de se poser le cas échéant. 
Toutefois, dans le cas où le réseau $R$ et le cône $Q$ sont tous les
deux définis sur le corps des nombres rationnels, il est facile de
vérifier que $R \cap Q_\R$ est toujours un réseau dans $R$ ; il n'y a
donc dans cette situation particulière pas de vérification 
supplémentaire à faire.
De la même façon, la troisième hypothèse de la proposition (à savoir que
$R \cap f^{-1}(C)$ et $f^{-1}(C)$ engendrent le même espace vectoriel
dans $E$) est automatiquement satisfaite dès que $R$, $C$ et $f$ sont
définis sur $\Q$. Dans les applications à suivre, ce sera toujours le
cas, et il ne sera donc pas nécessaire de vérifier la troisième 
hypothèse, de même que l'on pourra appliquer la proposition même si $Q$ 
n'engendre pas $E$.

Un mot enfin en rapport avec l'hypothèse de surjectivité de $f$.
Bien entendu, elle n'est pas véritablement contraignante car on peut
toujours appliquer la proposition en remplaçant $\R^n$ par l'image de
$f$. La conclusion du théorème n'est alors bien sûr plus valable que
pout les $y$ qui appartiennent à l'intersection de $f(R)+C$ avec l'image
de $f$.
\end{rem}

\begin{proof}
On définit $C' = f^{-1}(C)$ et de façon générale, si $X$ est un
sous-ensemble de $E$ ou de $\R^n$, on note $X_\R$ le sous-espace
vectoriel qu'il engendre.
Soit $M \subset C'_\R$ une maille du réseau $C'_\R \cap R$. C'est un
ensemble compact qui vérifie la propriété suivante : pour tout $x \in
C'_\R$, il existe $m \in M$ tel que $x + m \in R$. Du fait que $C'$ est
d'intérieur non vide dans $C'_\R$ et qu'il est stable par multiplication
par les réels positifs, on déduit qu'il existe un translaté
de $M$ entièrement inclus dans $(-C')$. Soit $K$ un tel translaté. 
C'est encore un ensemble compact qui vérifie une propriété analogue à 
celle satisfaite par $M$.
De même que précédemment, étant donné que $K$ est compact et que $Q$
est un cône convexe d'intérieur non vide dans $Q_\R$, il existe $x_0
\in Q_\R$ tel que $x_0 + K \subset Q$. On définit $y_0 = f(x_0)$.

Soit $y \in f(R)+C$.
Si $y - y_0 \not\in f(Q)+C$, on a $b_{Q,f,\ell,C}(y-y_0) = -\infty$ et
la proposition est évidente dans ce cas. On suppose donc que $y - y_0
\in f(Q)+C$. Alors $b_{Q,f,\ell,C}(y - y_0)$ est fini, et pour tout
$\varepsilon > 0$, il existe $x_1 \in Q$ tel que $f(x_1) \in (y - y_0) -
C$ et $\ell(x_1) \geq b_{Q,f,\ell,C}(y - y_0) - \varepsilon$. On a alors
$f(x_0 + x_1) \in y - C \subset f(R) + C_\R$, d'où on déduit que $x_0 + 
x_1 \in R + C'_\R$ et, de
là, qu'il existe $x_2 \in K$ tel que $x = x_0 + x_1 + x_2$ soit élément
de $R$. Comme $x_1 \in Q$ et $x_0 + K \subset Q$, l'élément $x$
appartient aussi à $Q$. Par ailleurs, $f(x) = y_0 + f(x_1) + f(x_2) \in
y_0 + (y - y_0) - C + f(K) = y - C$ car $f(K)$ est inclus dans $(-C)$ 
par construction de $K$. Ainsi, trouve-t-on :
$$b'_{Q,R,f,\ell,C}(y) \geq \ell(x) \geq \ell(x_0) +  
b_{Q,f,\ell,C}(y-y_0) - \varepsilon \inf_{x' \in K} \ell(x')$$
et la borne inférieure est finie étant donné que $K$ est compact.
\end{proof}

La minoration dans la proposition précédente fait intervenir la valeur
de la fonction $b_{Q,f,\ell,C}$ en $y - y_0$, alors qu'il aurait été
sans doute plus agréable d'avoir simplement $b_{Q,f,\ell,C}(y)$. En
général, malheureusement, on ne peut pas remplacer $y - y_0$ par $y$,
même en modifiant la constante $c$.
Néanmoins dans le cas où $Q$ est défini comme l'intersection d'un nombre 
fini d'hyperplans, on peut être plus précis : il résulte alors de la 
formule \eqref{eq:bfini} que la fonction $b_{Q,f,\ell,C}$ est 
uniformément continue sur l'ensemble $f(Q)+C$ et donc, en particulier, 
qu'il existe une constante réelle $c'$ telle que :
$$b_{Q,f,\ell,C}(y) - c' \leq b_{Q,f,\ell,C}(y-y_0)$$
pour tout $y \in (f(Q)+C) \cap (y_0 + f(Q)+C)$. En fait, l'inégalité 
précédente estencore satisfaite si $y \not\in f(Q)+C$ puisque dans ce 
cas, le minorant vaut $-\infty$. Enfin, en posant $c'' = c + c'$, il 
vient :
\begin{equation}
\label{eq:bprime}
b_{Q,f,\ell,C,D}(y) - c'' \leq b'_{Q,R,f,\ell,C,D}(y) \leq
b_{Q,f,\ell,C,D}(y)
\end{equation}
pour tout $y$ sauf éventuellement ceux qui appartiennent à $D \cap
(f(Q)+C)$ mais pas à $y_0 + f(Q)+C$. Il existe donc, si l'on veut, une zone 
de trouble autour de la frontière de $f(Q)+C$ sur laquelle on ne sait
pas contrôler le comportement de la fonction $b'_{Q,R,f,\ell,C,D}$.

\subsection{Étude du cône convexe $\Phi$}
\label{subsec:conephi}

On reprend la situation du théorème \ref{theo:parametrisation} : on note 
$I$ l'ensemble des couples d'entiers $(i,j)$ tels que $1 \leq i \leq j 
\leq d$, et $K$ le sous-ensemble convexe de $(\R^2)^I$ défini par les 
égalités \eqref{eq:rele1} et \eqref{eq:rele2} et les inégalités 
\eqref{eq:reli4}. En fait, comme cela a déjà été dit, la formule 
\eqref{eq:rele3} permet de se passer des $\mu_{i,j}$ ce que nous allons 
faire à partir de maintenant. On considère donc plutôt l'espace 
vectoriel $E = \R^I$ formé des suites $q = (q_{i,j})_{(i,j) \in I}$ 
indexées par les éléments de $I$, et à l'intérieur de celui-ci, le cône 
convexe $Q$ défini par les jeux d'inégalités suivants :

\medskip

\begin{tabular}{rp{13cm}}
(Jeu I) : & si $1 \leq i \leq j < d$, 
$q_{i,j} \geq q_{i,j+1}$ \\
(Jeu II) : & si $1 \leq i \leq j < d$, 
$q_{i,j} \leq q_{i+1,j+1}$ \\
(Jeu III) : & si $1 \leq i \leq j < d$, \\
& $\displaystyle b q_{j,j} - q_{j,d} + b \cdot \sum_{s=i}^{j-1} (q_{s, j}
- q_{s,j-1}) \geq 
b q_{j+1,j+1} - q_{j+1,d} + b \cdot \sum_{s=i+1}^j (q_{s, j+1}
- q_{s,j})$
\end{tabular}

\medskip

\noindent
qui correspondent aux inégalités \eqref{eq:reli4} dans lesquelles on a
remplacé chaque apparition d'un $\mu_{i,j}$ par son expression en
fonction des $q_{i,j}$. Le théorème \ref{theo:parametrisation} dit
alors que l'application $\varphi \mapsto (q_{i,j}(\varphi))_{(i,j) \in
I}$ définit une bijection entre $\Phi$ et $Q$.

\medskip

On munit $E$ du produit scalaire usuel : si $v = (v_{i,j})$ et 
$w = (w_{i,j})$ sont deux éléments de $E$, on pose
$$\left< v | w \right>_E = \sum_{(i,j) \in I} v_{i,j} w_{i,j}.$$
Les inégalités définissant $Q$ se réécrivent alors sous la forme $\left<
\vec v_m | q \right>_E \geq 0$ pour certains vecteurs explicites $\vec
v_m \in E$ ($1 \leq m \leq M$). Une étude attentive des formules montre
en outre que tous les vecteurs $\vec v_m$ appartiennent à l'hyperplan
\og somme des coordonnées égale $0$ \fg. 

On note comme précédemment $Q^\star$ le cône dual de $Q$ (par rapport au
produit scalaire $\left< \cdot | \cdot \right>_E$) ; c'est simplement le
cône convexe engendré par les vecteurs $\vec v_m$. En particulier, il
est lui aussi inclus dans l'hyperplan \og somme des coordonnées égale
$0$ \fg. Malheureusement, la présentation à l'aide des $\vec v_m$ n'est
pas adaptée au calcul des convexes $A_{Q,f,\ell}$ qui apparaissent en 
\ref{subsec:opticonv}, et qui joueront un rôle central dans la suite 
de l'article. Pour calculer ces ensembles, il serait plus commode de
disposer d'une présentation de $Q^\star$ comme l'intersection d'un 
certain nombre de demi-espaces. Hélas, de part sa complexité, le jeu III 
d'inégalités rend la chose difficile à réaliser. Pour contourner le
problème, l'idée consiste à travailler avec certaines approximations 
$Q$. Les deux plus simples d'entre elles sont les cônes convexes 
$Q_\min$ et $Q_\max$ définis comme suit : $Q_\max$ est le cône convexe 
défini par les jeux d'inégalité I et II, tandis $Q_\min$ est celui 
défini par le jeu I et le jeu II' que voici :

\medskip

\begin{tabular}{rp{10cm}}
(Jeu II') : & si $1 \leq i \leq j < d$, $q_{i,j} = q_{i+1,j+1}$
\end{tabular}

\medskip

\noindent
Il est clair que $Q \subset Q_\max$ et un calcul aisé montre que
$Q_\min \subset Q$. Ainsi on a des inclusions renversées au niveau
des duaux, d'où on déduit que, pour toute donnée $(f,\ell,C,D)$, 
l'encadrement 
\begin{equation}
\label{eq:encadb}
b_{Q_\min,f,\ell,C,D} \leq b_{Q,f,\ell,C,D} \leq
b_{Q_\max,f,\ell,C,D}
\end{equation}
est vrai. Il est maintenant temps de donner les présentations annoncées
des cônes duaux $Q_\min^\star$ et $Q_\max^\star$. Pour cela, on
introduit la définition suivante.

\begin{deftn}
\label{def:admissible}
Une partie $J$ de $I$ est dite \emph{admissible} si pour tout couple
$(i,j) \in J$, les deux couples $(i,j+1)$ et $(i-1,j-1)$ sont dans $J$,
dès qu'ils appartiennent à $I$.
\end{deftn}

On peut remarquer que les parties admissibles dans le sens précédent
sont naturellement en bijection avec les parties de $\{1, \ldots, d\}$ :
à une telle partie $J$, on fait correspondre l'ensemble $T$ des nombres 
non nuls qui sont de la forme :
$$\card \big\{ \, j \geq i \,\, | \,\, (i,j) \in J \, \big\}$$
pour un $i \in \{1, \ldots, d\}$. Réciproquement, si $T$ est un
sous-ensemble de $\{1, \ldots, d\}$, on note $t_1 > \cdots > t_{\card\,T}$ 
ses éléments et on lui associe l'ensemble $J \subset I$ formé des couples 
$(i,j)$ tels que $i \leq \card\,T$ et $j > d - t_i$.
On vérifie que les deux applications précédentes sont des bijections 
inverses l'une de l'autre entre l'ensemble des parties admissibles de 
$I$ et l'ensemble des parties de $\{1, \ldots, d\}$. Par exemple, si $s 
\in \{1, \ldots, d\}$, la partie admissible correspondant à $\{1, 
\ldots, s\}$ est $I_s = \{ (i,j) \in I \,|\, j - i \geq d - s\}$.

\begin{prop}
\label{prop:qdual}
Une suite $v = (v_{i,j})$ de $E$ appartient à $Q_\min^\star$ si,
et seulement si :
$$\sum_{(i,j) \in I} v_{i,j} = 0
\quad \text{et} \quad \forall s \in \{1, \ldots, d\}, \,
\sum_{(i,j) \in I_s} v_{i,j} \leq 0.$$
Une suite $v = (v_{i,j})$ de $E$ appartient à $Q_\max^\star$ si,
et seulement si :
$$\sum_{(i,j) \in I} v_{i,j} = 0
\quad \text{et} \quad \forall J \subset I \text{ admissible},\,
\sum_{(i,j) \in J} v_{i,j} \leq 0.$$
\end{prop}

\noindent
En utilisant l'identification décrite précédemment entre parties
admissibles de $I$ et parties de $\{1, \ldots, d\}$, on peut réécrire
la condition d'appartenant à $Q_\max^\star$ comme suit :
\begin{equation}
\label{eq:qmaxdual}
\sum_{(i,j) \in I} v_{i,j} = 0
\quad \text{et} \quad \forall\:T \subset \{1, \ldots, d\}, \,\,
\sum_{s=1}^{\card\:T} \, \sum_{j=s+1}^{t_s} v_{j-s,j-1} \leq 0
\end{equation}
où $t_s$ est le $s$-ième plus petit élément de $T$.

\subsubsection*{Démonstration de la proposition \ref{prop:qdual} : 
un peu de théorie des flots}

Il est possible de donner une démonstration \og à la main \fg\ de la
proposition \ref{prop:qdual} mais, comme me l'a signalé Bodo Lass, la
proposition peut également se déduire du théorème
Flot-Maximal-Coupe-Minimale, classique en théorie des graphes. J'ai
choisi ici de présenter cette dernière approche qui est à la fois plus
générale et plus conceptuelle.

\medskip

Quelques rappels pour commencer au sujet du théorème 
Flot-Maximal-Coupe-Maximale.
Soit $G$ un graphe fini orienté dans lequel on a privilégié deux sommets
$D$ (comme départ) et $A$ (comme arrivée) et on a attribué à chaque
arête $a$ un nombre positif ou nul, éventuellement égal à $+\infty$,
appelé \emph{capacité de $a$}, et noté $c(a)$. Si $a$ est une arête dans
$G$, on note $s_1(a)$ (resp. $s_2(a)$) le sommet duquel elle part
(resp. auquel elle aboutit). Un flot de $D$ vers $A$ est une fonction
$f$ à valeurs réelles définie sur les arêtes de $G$ satisfaisant les
propriétés suivantes :
\begin{itemize}
\item pour toute arête $a$, on a $0 \leq f(a) \leq c(a)$ ;
\item pour tout sommet $s$ différent de $D$ et $A$, on a 
$\sum_{a | s_1(a) = s} f(a) = \sum_{a | s_2(a) = s} f(a)$.
\end{itemize}
La dernière propriété implique que :
$$\sum_{a | s_1(a) = D} f(a) - \sum_{a | s_2(a) = D} f(a) =
\sum_{a | s_2(a) = A} f(a) - \sum_{a | s_1(a) = A} f(a).$$
Cette valeur commune s'appelle la \emph{valeur} du flot $f$ et est notée
$|f|$.
Une \emph{coupe} $C$ de $G$ est la donnée d'une partition de l'ensemble
des sommets de $G$ en deux parties $\mathcal D$ et $\mathcal A$ telles
que $D \in \mathcal D$ et $A \in \mathcal A$. La \emph{capacité} de la
coupe $C$, que l'on note $|C|$, est la somme des $c(a)$ étendue à
toutes les arêtes $a$ qui ont leur origine dans $\mathcal D$ et leur
arrivée dans $\mathcal A$. 

Il est facile de voir que si $f$ est un flot et $C$ est une coupe sur le
graphe $G$ précédemment fixé, alors $|f| \leq |C|$. Ainsi, en passant
aux bornes supérieures et inférieures, on obtient $\sup_{f\,\text{flot}}
|f| \leq \min_{C\,\text{coupe}} |C|$ (notez qu'il n'y a qu'un nombre
fini de coupes possibles).

\begin{theo}[Flot-Maximal-Coupe-Minimale]
Avec les notations précédentes, on a :
$$\textstyle \sup_{f\,\text{flot}} |f| = \min_{C\,\text{coupe}} |C|$$
et la borne supérieure précédente est atteinte.
\end{theo}

\begin{proof}
Voir par exemple \cite{maxflot}.
\end{proof}

On se place à présent dans une situation un peu différente. On considère
toujours un graphe fini orienté $G$ mais on ne se donne plus de
décoration : on ne suppose plus que deux de ses sommets sont
privilégiés, ni que les arêtes de $G$ sont munies d'une capacité. On
note $S$ l'ensemble des sommets de $G$. À un tel graphe, on associe
l'espace euclidien $E_G = \R^S$ muni du produit scalaire usuel $\left<
\cdot | \cdot \right>_G$ et $Q_G$ le cône convexe regroupant les
éléments $x = (x_s)_{s \in S} \in E_G$ vérifiant
$$x_{s_2(a)} \leq x_{s_1(a)}$$
pour toute arête $a$ de $G$.

\begin{deftn}
\label{def:admissibleg}
Un sous-ensemble $S'$ de $S$ est dit \emph{admissible} si toute arête de
$G$ ayant son origine dans $S'$ a aussi son but dans $S'$.
\end{deftn}

\noindent
On a alors la proposition suivante, de laquelle il résulte facilement la
proposition \ref{prop:qdual}.

\begin{prop}
\label{prop:dualgraphe}
On conserve les notations précédentes, et on note $Q_G^\star$ le cône
dual de $Q_G$ par rapport au produit scalaire $\left< \cdot | \cdot
\right>_G$. Alors un élément $x = (x_s)_{s \in S} \in E_G$ appartient
à $Q_G^\star$ si, et seulement si :
$$\sum_{s \in S} x_s = 0 \quad \text{et} \quad
\forall S' \subset S \text{ admissible},\,
\sum_{s \in S'} x_s \leq 0.$$
\end{prop}

\begin{proof}
De la définition de $Q_G$, il résulte que $Q_G^\star$ est un cône
convexe engendré par des vecteurs de l'hyperplan \og somme des
coordonnées égale 0 \fg. Ainsi $Q_G^\star$ est inclus dans cet
hyperplan. Par ailleurs, si $S' \subset S$ est un ensemble admissible de
sommets, l'opposé de la fonction indicatrice de $S'$ définit un vecteur
$-\1_{S'} \in Q_G$. On en déduit que le produit scalaire $\left< x |
\1_{S'} \right>_G$ est négatif ou nul dès que $x \in Q_G^\star$. Les 
éléments
de $Q_G^\star$ vérifient donc bien les conditions de la proposition. 

Réciproquement, on considère un vecteur $x = (x_s) \in E_G$ vérifiant
ces conditions. Soit $M$ un nombre réel positif assez grand pour que
toutes les sommes $x_s + M$ (pour $s$ décrivant $S$) soient positives ou
nulles. On introduit le graphe $\tilde G$ obtenu à ajoutant à $G$ deux
nouveaux sommets notés $D$ et $A$ et, pour tout sommet $s$ de $G$, une
arête de $D$ vers $s$ et une arête de $s$ vers $A$. On définit une
capacité $c$ sur $\tilde G$ comme suit :
\begin{itemize}
\item si $a$ est une arête de $G$, on pose $c(a) = +\infty$ ;
\item si $a$ part de $D$ et arrive à un sommet $s$ de $G$, on pose
$c(a) = x_s + M$ ;
\item si $a$ part d'un sommet $s$ de $G$ et arrive à $A$, on pose
$c(a) = M$.
\end{itemize}
Soit $C = (\mathcal D, \mathcal A)$ une coupe de $\tilde G$. L'ensemble 
$S' = \mathcal D \backslash \{D\}$ est inclus dans $S$.  Si $S'$ n'est 
pas admissible, cela signifie qu'il existe une arête reliant un sommet 
de $S'$ à un sommet de $S \backslash S'$. Autrement dit, il existe une 
arête dans $G$ (donc de capacité infinie) reliant un sommet de $\mathcal 
D$ à un sommet de $\mathcal A$. Dans ce cas, la capacité de la coupe $C$ 
est donc infinie. Si, au contraire, l'ensemble des sommets $S'$ est 
admissible, la capacité de la coupe $C$ vérifie :
$$|C| = \sum_{s \in S'} M + \sum_{s \in S \backslash S'} (x_s + M)
= n M - \sum_{s \in S'} x_s \geq n M$$
où on a noté $n$ le nombre de sommets de $G$. Ainsi la capacité minimale
d'une coupe, notée $c_\min$ est, elle aussi, supérieure ou égale à $nM$.
D'après le théorème Flot-Maximal-Coupe-Minimale, il existe un flot $f$
sur $\tilde G$, de valeur $c_\min$. Comme la somme des capacités des
arêtes sortant de $D$ est $nM$, la capacité de ce flot est inférieure ou
égale à $nM$. On en déduit qu'elle est égale à $nM$ (et donc qu'il en 
est de même de $c_\min$) et que toute arête partant de $D$ est saturée,
c'est-à-dire que pour toute arête $a$ partant de $D$, on a $f(a) = c(a) 
= M + x_{s_2(a)}$. De même, on démontre que, pour toute arête $a$ 
partant d'un sommet $s \in S$ et aboutissant en $A$, on a $f(a) = c(a) = 
M$. Par la condition de flot, on a pour tout sommet $s \in S$ :
$$\sum_{a | s_1(a) = s} f(a) - \sum_{a | s_2(a) = s} f(a) = x_s$$
où, dans les sommes précédentes, $a$ désigne une arête de $G$.
Ainsi, pour tout $y = (y_s) \in Q_G$,
on a :
$$\left< x | y \right>_E 
= \sum_{s \in S} y_s \cdot \Bigg( \sum_{a | s_1(a) = s} f(a) - 
\sum_{a | s_2(a) = s} f(a) \Bigg)
= \sum_a (y_{s_1(a)} - y_{s_2(a)}) f(a) \geq 0$$
d'où on déduit finalement que $x \in Q_G^\star$ comme voulu.
\end{proof}

\subsection{Un premier exemple d'application}
\label{subsec:applimethode}

À titre d'exemple, et afin de familiariser le lecteur avec les méthodes 
de cet article, j'aimerais montrer comment la machinerie qui vient 
d'être introduite permet de démontrer la majoration du théorème 
\ref{theo:dimcaruso} sous l'hypothèse supplémentaire $b \geq d-1$. Ce 
cas est intéressant car de nombreuses difficultés techniques 
s'évanouïssent, mais il permet tout de même de donner une idée correcte 
de la nature des raisonnements qui apparaîtront dans la section 
suivante.

Pour ne pas avoir, dans la suite, à distinguer systématiquement les cas
selon que $h$ soit ou non égal à $0$, on se restreint à partir de
maintenant à $h \neq 0$, le cas contraire se traitant de façon
complètement analogue en ajoutant $\frac{d(d-1)} 2$ à tous les
majorants.

\subsubsection{Positionnement du problème}
\label{subsec:positionnement}

Un réseau $L$ de $M$ définit un $k$-point de la variété $\calX_{\leq e}$
si, et seulement si les exposants des diviseurs élémentaires du
$k[[u]]$-module engendré par $\sigma(L)$ par rapport à $L$ sont tous
compris entre $0$ et $e$. Si on note ceux-ci $\mu_1(L) \geq \cdots \geq
\mu_d(L)$ comme dans la proposition \ref{prop:varphi}, cela se réécrit
$\mu_1(L) \leq e$ et $\mu_d(L) \geq 0$. Par ailleurs, si on note
$q_{i,j}(L)$ et $\mu_{i,j}(L)$ les nombres réels associés au $d$-uplet
de fonctions $\varphi(L) = (\varphi_1(L), \varphi_2(L), \ldots,
\varphi_d(L))$, l'assertion 4 de la proposition \ref{prop:varphi}
combinée à la proposition \ref{prop:bijphipsi} assure que $\mu_j(L) =
\mu_{1,j}(L)$ pour tout indice $j \in \{1, \ldots, d\}$. Ces
observations conduisent à une décomposition de la variété $\calX_{\leq
e}$ en union disjointe de parties localement fermées comme suit :
$$\calX_{\leq e} = \bigsqcup_{\varphi \in \Phi_{\leq e}}
\calX_{\varphi}$$
où l'ensemble $\Phi_{\leq e}$ réunit les $\varphi \in \Phi_\Z$ tels que
$\mu_{1,1}(\varphi) \leq e$ et $\mu_{1,d}(\varphi) \geq 0$. On déduit
directement de cette écriture, une formule au niveau des dimensions :
$$\dim_k \calX_{\leq e} = \sup_{\varphi \in \Phi_{\leq e}} \dim_k
\calX_{\varphi}.$$

Comme précédemment, on note $E$ l'espace des suites indicées par
l'ensemble $I$, et $Q$ le cône convexe défini par les jeux d'inégalités
I, II et III. Soit $f : E \to \R^2$ l'application linéaire qui à une
famille $(q_{i,j})$ associe le couple $(\mu_{1,1}, \mu_{1,d})$ où ces
réels sont définis comme d'habitude par la formule \eqref{eq:rele3}. Par
ailleurs, par le lemme \ref{lem:dim}, la fonction $\dim$ s'étend en une
forme linéaire $E \to \R$, notée $\ell$. On pose encore $C = \R^+ \times
\R^-$, $D = \R^2$ et on note $R$ le réseau formé des éléments $(q_{i,j})
\in E$ tels que $q_{i,j} \in \frac 1 b \Z$ pour tout $(i,j) \in I$ et
$q_{i,d} \in \Z$ pour tout $i \in \{1, \ldots, d\}$. Les théorèmes
\ref{theo:parametrisation} et \ref{theo:dim}, couplés à la proposition
\ref{prop:integrite}, fournissent alors l'égalité (on rappelle que 
l'on suppose $h \neq 0$) :
\begin{equation}
\label{eq:encadrdimve}
\dim_k \calX_{\leq e} = b'_{Q,R,f,\ell,C} (e,0)
\end{equation}
et donc en particulier $\dim_k \calX_{\leq e} \leq b_{Q,R,f,\ell,C}
(e,0)$. Pour estimer ce dernier nombre, on utilise l'encadrement 
\eqref{eq:encadb}, et les propositions \ref{prop:dualconv2} et
\ref{prop:qdual} pour évaluer les fonctions $b_{Q_\min,f,\ell,C}$ et 
$b_{Q_\max,f,\ell,C}$. 

\subsubsection{Les vecteurs $\vec \mu_1, \ldots, \vec \mu_d$ et $\vec
\delta$}
\label{subsec:vecmui}

Pour pouvoir suivre la méthode indiquée ci-dessus, il faut commencer
par exprimer les coordonnées de la fonction $f$ ainsi que la forme 
linéaire $\ell$ comme des produits scalaires contre certains vecteurs de 
$E$. C'est l'objet de ce numéro.

Pour un indice $i$ compris entre $1$ et $d$, on appelle $\vec \mu_i$ le
vecteur de $E$ tel que pour toute suite $q = (q_{i,j}) \in E$, on ait
$\mu_{1,i} = \left< q | \vec \mu_i \right>_E$ où $\mu_{1,i}$ est, comme
précédemment, le réel calculé par la formule
\eqref{eq:rele3}\footnote{Pour le calcul de la dimension de $\calX_{\leq
e}$, seuls les vecteurs $\vec \mu_1$ et $\vec \mu_d$ seront utiles.
Cependant, les autres vecteurs $\vec \mu_i$ serviront dans la suite pour
estimer la dimension d'autres variétés, et nous avons pensé qu'il était
préférable de les introduire tous en même temps.}. 
Si l'on décide de représenter les éléments de $E$ comme des matrices
triangulaires supérieures (le terme d'indice $(i,j)$ étant placé à
l'intersection de la $i$-ième ligne et de la $j$-ième colonne), on a
ainsi :
$$\vec \mu_i = \left(
\raisebox{0.5\depth}{\xymatrix @R=3pt @C=3pt {
0 \ar@{.}[rrdd] \ar@{.}[rr] & & 0 \ar@{.}[dd] & -b \ar@{.}[ddd] & 
b \ar@{.}[dddd] & 0 \ar@{.}[rr] \ar@{.}[dddd] & & 0 \ar@{.}[ddd] \\
\\
& & 0 \\
& & & -b & & & & 0 \\
& & & & b & 0 \ar@{.}[r] & 0 \ar@{.}[ur] & -1 \\
& & & & & 0 \ar@{.}[rrdd] \ar@{.}[rr] & & 0 \ar@{.}[dd] \\
\\
& & & & & & & 0 }} \right)$$
où le $b$ sur la diagonale est à l'intersection de la $i$-ième ligne et
de la $i$-ième colonne. De même, on définit le vecteur $\vec \delta$
par $\left< q | \vect \delta \right>_E = \ell(q)$ pour tout $q \in E$.
Le lemme \ref{lem:dim} montre que $\vec \delta$ s'exprime comme suit :
$$\vec \delta = b \cdot \left(
\raisebox{0.5\depth}{\xymatrix @R=3pt @C=7pt {
1 \ar@{.}[rrrddd] \ar@{.}[rrr] & & & 1 \ar@{.}[ddd] & 0 \\
& & & & -1 \ar@{.}[dd] \\ \\
& & & 1 & 2-d \\
& & & & 1-d }} \right) + \left(
\raisebox{0.5\depth}{\xymatrix @R=3pt @C=7pt {
0 \ar@{.}[rrrddd] \ar@{.}[rrr] & & & 0 \ar@{.}[ddd] & 1-d \\
& & & & 3-d \ar@{.}[dd] \\ \\
& & & 0 & d-3 \\
& & & & d-1 }} \right).$$
Si $J \subset I$ est un sous-ensemble admissible (voir définition
\ref{def:admissible}) et si $x = (x_{i,j}) \in E$, on note $S_J (x)$
la somme des coordonnées $x_{i,j}$ pour $(i,j)$ parcourant $J$. Les
fonctions $S_J$ ainsi définies sont clairement des formes linéaires sur
$E$. Si $T$ est la partie de $\{1, \ldots, d\}$ correspondant à
l'ensemble admissible $J$, on s'autorise à noter $S_T$ à la place de
$S_J$. Pour les calculs à suivre, le lemme suivant nous sera fort
utile.

\begin{lemme}
\label{lem:calculST}
Si $T$ est un sous-ensemble de $\{1, \ldots, d\}$, on a :
$$S_T(\vec \mu_i) = b \cdot \1_T(d+1-i) - [(\card \: T) \geq i]$$
où $\1_T$ est la fonction indicatrice de $T$ et l'expression
$[(\card \: T) \geq i]$ vaut $1$ si $\card\:T \geq i$ et $0$ sinon.
On a aussi :
$$S_T(\vec \delta) = b \cdot \Bigg( \sum_{t \in T} t - \frac{\card\,T
\cdot (\card\, T + 1)} 2 \Bigg) - \card\,T  \cdot (d - \card\, T).$$
\end{lemme}

\begin{proof} 
C'est un simple calcul à partir des descriptions précédentes.
\end{proof}

\subsubsection{Calcul du majorant}

On commence par calculer la fonction $b_{Q_\min,f,\ell,C}$ en
utilisant la proposition \ref{prop:dualconv2}. Il s'agit donc de 
déterminer l'ensemble $B_{Q_\min,f,\ell,C} = A_{Q_\min,f,\ell}
\cap C^\star$. Il est facile de voir que $C^\star = \R^+ \times \R^-$ ;
reste donc à calculer $A_{Q_\min,f,\ell}$. Comme $f = (\left< \vec
\mu_1, \cdot \right>_E, \left< \vec \mu_d, \cdot \right>_E)$ et $\ell =
\left< \vect \delta | \cdot \right>_E$, la formule \eqref{eq:AQfl}
s'écrit :
$$A_{Q_\min,f,\ell} = \big\{ \, (y_1, y_d) \in \R^2 \,\, | \,\, 
y_1 \vec \mu_1 + y_d \vec \mu_d - \vec \delta \in Q_\min^\star
\,\big\}.$$
On cherche donc les couples $(y_1, y_d)$ tels que $y_1 \geq
0$, $y_d \leq 0$, et $y_1 \vec \mu_1 + y_d \vec \mu_d - \vec \delta \in
Q_\min^\star$. D'après la proposition \ref{prop:qdual}, la dernière
condition se réécrit :
\begin{equation}
\label{eq:systqmin}
\left\{ \begin{array}{ll}
y_1 S_{I_s}(\vec \mu_1) + y_d S_{I_s}(\vec \mu_d) \leq S_{I_s}
(\vec \delta) & \forall s \in \{1, \ldots, d\} \\
y_1 S_{I_d}(\vec \mu_1) + y_d S_{I_d}(\vec \mu_d) = S_{I_d}
(\vec \delta)
\end{array} \right.
\end{equation}
où on rappelle que $I_s$ désigne l'ensemble admissible formé des couples
$(i,j) \in I$ tels que $j-i \geq d-s$. Comme cet ensemble correspond à la
partie $T = \{1, \ldots, s\}$, le lemme \ref{lem:calculST} assure que
si $s < d$, on a $S_{I_s}(\vec \mu_1) = -1$, $S_{I_s}(\vec \mu_d) = b$ et
$S_{I_s}(\vec \delta) = -s(d-s)$, tandis que pour $s = d$, on a $S_{I_d}
(\vec \mu_1) = S_{I_d}(\vec \mu_d) = b-1$, $S_{I_d}(\vec \delta) = 0$.
Ainsi, le système \eqref{eq:systqmin} est équivalent à :
$$y_1 + y_d = 0 \qquad \text{et} \qquad
y_1 \geq \frac{s(d-s)}{b+1}, \quad \forall s \in \{1, \ldots, d-1\}.$$
Comme on voit aisément que le maximum de $s(d-s)$ vaut
$[\frac {d^2} 4]$, l'ensemble $B_{Q_\min,f,\ell,C}$ est formé des 
couples $(y,-y)$ avec $y \geq [\frac {d^2} 4] \cdot \frac 1{b+1}$. 
Ainsi, si $y_1 \geq y_d$, on obtient :
$$b_{Q_\min,f,\ell,C}(y_1, y_d) = \cro{\frac {d^2} 4} \cdot
\frac{y_1 - y_d}{b+1}.$$
On souhaite à présent montrer que l'expression ci-dessus vaut encore
pour la fonction $b_{Q_\max,f,\ell,C}$. Comme on sait que 
$b_{Q_\min,f,\ell,C} \leq b_{Q_\max,f,\ell,C}$, il suffit, pour
établir ce que l'on veut, de montrer que $( [\frac {d^2} 4] \cdot 
\frac 1{b+1}, - [\frac {d^2} 4] \cdot \frac 1{b+1}) \in 
B_{Q_\max,f,\ell,C}$, c'est-à-dire que pour toute partie $T \in 
\{1, \ldots, d\}$, on a :
$$\cro{\frac {d^2} 4} \cdot (S_T(\vec \mu_1) - S_T(\vec \mu_d))
\leq (b+1) \cdot S_T(\vec \delta).$$
(Notez que le cas d'égalité, qui doit être obtenu pour $T = \{1, \ldots,
d\}$, a déjà été vérifié.) On commence par éliminer les cas triviaux $T
= \emptyset$ et $T = \{1, \ldots, d\}$. Ceci permet d'écrire 
$S_T(\vec\mu_1) = b \1_T(d) - 1$ et $S_T(\vec \mu_d) = \1_T(1)$. Par
ailleurs, si $T = \{t_1, \ldots, t_s\}$ avec $t_1 < \cdots < t_s$, la 
formule du lemme \ref{lem:calculST} se réécrit 
\begin{equation}
\label{eq:STdelta}
S_T(\vec\delta) = -s(d-s) + b \sum_{i=1}^s (t_i - i).
\end{equation}
On distingue à présent quatre cas selon que les entiers $1$ et $d$ 
appartiennent ou n'appariennent pas à $T$. 
Si $1 \in T$ et $d \not\in T$, alors $S_T(\vec \mu_1) - S_T(\vec \mu_d)
= -b-1$ et l'inégalité à démontrer devient $S_T(\vec \delta) \geq -
[\frac{d^2} 4]$. Or les $t_i - i$ étant tous positifs ou nuls, la
formule \eqref{eq:STdelta} montre que $S_t(\vec\delta) \geq -s(d-s)$
d'où il suit ce que l'on veut.
Pour les autres cas, il est important de remarquer que si $1 \not\in
T$, alors tous les $t_i - i$ sont supérieurs ou égaux à $1$, alors que
si $d \in T$, on a $t_s - s = d-s$. En particulier, si l'on excepte le
cas que l'on vient de traiter, c'est-à-dire si l'on suppose que $1 
\not\in T$ ou que $d \in T$, on a 
$$S_T(\vec\delta) \geq b \cdot \min(s,d-s) - s(d-s).$$
Comme on a supposé $b \geq d-1$, on en déduit que $S_T(\vec\delta) \geq
0$, ce qui suffit à conclure dans le cas où les deux nombres $1$ et $d$
appartiennent (resp. n'appartiennent pas) à $T$ puisqu'alors $S_T(\vec
\mu_1) - S_T(\vec \mu_d) = -1 \leq 0$. Si finalement, $1 \not\in T$ et
$d \in T$, la minoration de $S_T(\delta)$ est renforcée comme suit :
$$S_T(\vec\delta) \geq b (s-1 + d-s) - s(d-s) = b(d-1) - s(d-s) \geq 
b(d-1) - \cro{\frac {d^2} 4}.$$
Étant donné que $S_T(\vec \mu_1) - S_T(\vec \mu_d) = b-1$ dans ce cas,
il suffit d'établir 
$$(b+1)\Big( b(d-1) - \cro{\frac {d^2} 4} \Big) \geq (b-1) \cdot 
\cro{\frac {d^2} 4},$$
ce qui se simplifie encore en $(b+1)(d-1) \geq 2 [\frac{d^2} 4]$. 
Finalement, la condition $b \geq d-1$ assure que cette inégalité est 
bien satisfaite. On a ainsi démontré que
$$b_{Q_\max,f,\ell,C}(y_1, y_d) =
\cro{\frac {d^2} 4} \cdot \frac{y_1 - y_d}{b+1} \quad \text{pour } y_1
\geq y_d.$$
à partir de quoi la majoration du théorème \ref{theo:dimcaruso} s'ensuit 
puisque l'on sait que la dimension de $\calX_{\leq e}$ est majorée par 
$b_{Q,f,\ell,C}(e,0)$ lui même majoré par $b_{Q_\max,f,\ell,C}(e,0)$.

\section{Dimension des variétés $\calX_{\leq e}$, $\calX_\mu$ et
$\calX_{\leq \mu}$}
\label{sec:dimkisin}

Pour estimer les dimensions des variérés $\calX_{\leq e}$, $\calX_\mu$ 
et $\calX_{\leq \mu}$, on suit la méthode introduite précédemment,
sauf que l'on remplace la fonction $f$ par la fonction
$$g : E \to \R^d, \quad q \mapsto (
\left< \vec \mu_1 | q \right>_E, \left< \vec \mu_2 | q \right>_E,
\ldots, \left< \vec \mu_d | q \right>_E)$$
et que la lettre $C$ fait désormais référence au nouveau cône convexe 
$\{ (y_1, \ldots, y_d) \in \R^d \, | \, y_1 \geq \cdots \geq y_d \}$
Toutes les autres notations ($Q$, $R$, $Q_\min$, $Q_\max$, $\ell$, $\vec
\mu_i$, $\vec \delta$, \emph{etc.}) sont conservées ; on se contente 
donc de renvoyer le lecteur aux \S\S \ref{subsec:conephi} et 
\ref{subsec:vecmui} s'il souhaite se remémorer les définitions. Le
cône dual de $C$ jouera un rôle particulier dans la suite ; on remarque
d'ores et déjà qu'il est engendré par les vecteurs de la forme
$(0,\ldots, 0, 1,-1,0, \ldots, 0)$. On en déduit facilement une
description en termes d'inégalités :
$$C^\star = \left\{ \, (z_1, \ldots, z_d) \in \R^d \,\,\left|\,\, 
\begin{array}{l}
z_1 + \cdots + z_d = 0 \\
z_1 + \cdots + z_i \geq 0, \, \forall i \in \{1, \ldots, d\}
\end{array} \right. \right\}.$$
D'autre part, étant donné que les vecteurs $\vec \mu_i$ forment
trivialement une famille libre dans $E$, l'application linéaire
$g$ est surjective. On peut ainsi appliquer la proposition 
\ref{prop:bprime} avec les données précédentes.

À partir de maintenant, on se restreint à nouveau au cas où $h \neq 0$, 
le cas contraire se traitant exactement de la même façon sauf qu'il
faut ajouter $\frac{d(d-1)} 2$ à tous les majorants.
Par le théorème \ref{theo:dim}, pour $\mu = (\mu_1, \ldots, \mu_d)$
vérifiant les bonnes hypothèses, la dimension de $\calX_\mu$ (resp.
$\calX_{\leq \mu}$) est majorée (et même bien approchée en vertu de la
proposition \ref{prop:bprime}) par le nombre $b_{Q,g,\ell,0}(\mu)$ (resp.
$b_{Q,g,\ell,C^\star}(\mu)$). On est ainsi ramené à calculer
$b_{Q,g,\ell,0}(\mu)$ et $b_{Q,g,\ell, C^\star} (\mu)$, et pour cela, à
étudier les ensembles $B_{Q,g,\ell,0}$ et $B_{Q,g,\ell,C^\star}$.

\subsection{Les points extrémaux de $A_{Q_\max,g,\ell}$}
\label{subsec:extrAQmaxgl}

Soit $\mathfrak S_d$ l'ensemble des permutations de l'ensemble $\{1,
\ldots, d\}$. Pour tout $\weyl \in \mathfrak S_d$, on définit le
vecteur $\vec \rho_\weyl \in \R^d$ par :
$$\vec \rho_\weyl =
\left((b-1) \cdot \sum_{n=1}^\infty \frac{d+1-i-\weyl^n(i)}{b^n} \right)
_{1 \leq i \leq d}.$$
Il est clair que la somme infinie que l'on vient d'écrire converge. En
fait, la suite des $\weyl^n(i)$ étant périodique, elle converge même
vers un nombre rationnel (qui s'exprime comme l'évaluation en $b$ d'une
fraction rationnelle à coefficients entiers). De plus, lorsque $\weyl = 
\weyl_0$ est la permutation $i \mapsto d+1-i$, le vecteur $\vec 
\rho_{\weyl_0}$ s'exprime facilement : il vaut $\frac {2 \vec 
\rho}{b+1}$ où on rappelle que $\vec \rho = (\frac{d-1} 2, \frac{d-3} 2, 
\ldots, \frac {1-d} 2)$. Enfin, à partir de maintenant, on pose $b_0 = 1 
+ [\frac{(d-1)^2} 4]$.

\begin{prop}
\label{prop:extrAQmaxgl}
On suppose $b \geq b_0$.
Les points extrémaux de $A_{Q_\max,g,\ell}$ sont exactement les
vecteurs $\vect \rho_\weyl$ pour $\weyl$ parcourant $\mathfrak S_d$.
\end{prop}

La fin de cette partie est consacrée à la démonstration de la 
proposition. On peut d'ores et déjà présenter le convexe 
$A_{Q_\max,g,\ell}$ comme une intersection de demi-espaces ; il suffit 
pour cela d'injecter la description de $Q_\max^\star$ donnée par la 
proposition \ref{prop:qdual} dans la formule \eqref{eq:AQfl}. La
condition qui en découle s'exprime simplement à l'aide des fonctions
$S_T$ introduites juste avant le lemme \ref{lem:calculST} ; elle dit 
qu'un élément $y = (y_1, \ldots, y_d) \in \R^d$ appartient à 
$A_{Q_\max,g,\ell}$ si, et seulement si
$$y_1 S_T(\vec \mu_1) + y_2 S_T (\vec \mu_2) + \cdots +
y_d S_T (\vec \mu_d) \leq S_T (\vec \delta)$$
pour toute partie $T \subset \{1, \ldots, d\}$ et l'égalité est atteinte
lorsque $T = \{1, \ldots, d\}$. Avec les expressions du lemme
\ref{lem:calculST}, ceci se réécrit :
\begin{equation}
\label{eq:AQmaxgl}
\left\{\begin{array}{l}
y_1 + y_2 + \cdots + y_d = 0 \\
f_T(y) \geq 0, \quad \forall\, T \subsetneq \{1, \ldots, d\}
\end{array} \right.
\end{equation}
où on a posé $s = \card\,T$ et :
\begin{equation}
\label{eq:fT}
f_T(y) = (y_1 + \cdots + y_s) - b \cdot \sum_{t \in T} y_{d+1-t} -
s(d-s) + b \cdot \Bigg( \sum_{t \in T} t - \frac {s(s+1)} 2\Bigg).
\end{equation}
Ceci étant dit, un point extrémal de $A_{Q_\max,g,\ell}$ n'est autre
qu'un point de $A_{Q_\max,g,\ell}$ qui se situe à l'intesection de
$d-1$ hyperplans affines indépendants parmi ceux d'équation $f_T(x) = 0$
avec $T \subsetneq \{1, \ldots, d\}$.
Calculer ces points revient donc à déterminer une condition nécessaire
et suffisante sur les parties $T_1, \ldots, T_{d-1}$ de 
$\{1, \ldots, d\}$ pour que :
\begin{itemize}
\item[i)] en posant $T_d = \{1, \ldots, d\}$, les hyperplans affines
d'équation $f_{T_i}(x) = 0$ ($1 \leq i \leq d$) aient un unique point
d'intersection, et
\item[ii)] ce point d'intersection appartienne à $A_{Q_\max,g,\ell}$.
\end{itemize}
On commence par deux lemmes.

\begin{lemme}
\label{lem:sommeyi}
On considère des entiers $m$ et $n$ tels que $0 \leq m < n \leq d$
et un élément $y = (y_1, \ldots, y_d) \in A_{Q_\max,g,\ell}$. Alors :
$$-m(n-m) \leq y_{m+1} + y_{m+2} + \cdots + y_n \leq (n-m)(d-n).$$
\end{lemme}

\begin{proof}
On commence par prouver la majoration. Dans le cas où $m = 0$, 
celle-ci provient directement de l'inégalité $f_T(y) \geq 0$ où
$T = \{d+1-n, \ldots, d\}$. Dans le cas où $m > 0$, on considère
la partie $T = \{d+1-n, \ldots, d-m\}$. L'inégalité $f_T(y) \geq 0$
donne alors :
$$(y_1 + \cdots + y_{n-m}) - b (y_{m+1} + \cdots + y_n) \geq
(n-m)(d-n+m) - b (n-m) (d-n).$$
Or, on sait déjà par ailleurs que $y_1 + \cdots + y_{n-m} \leq
(n-m)(d-n+m)$. Le résultat s'ensuit. Pour la minoration, le plus rapide 
est de remarquer que l'application $(y_1, \ldots, y_d) \mapsto (-y_d, 
\ldots, -y_1)$ définit une bijection de $A_{Q_\max,g,\ell}$ dans 
lui-même et que, \emph{via} cette bijection, la majoration qui vient 
d'être établie donne la minoration.
\end{proof}

\begin{lemme}
On suppose $b > d$.
Soient $T_1$ et $T_2$ deux parties de $\{1, \ldots, d\}$, et soit
$y \in A_{Q_\max,g,\ell}$. On a :
$$f_{T_1}(y) + f_{T_2}(y) \geq f_{T_1 \cap T_2}(y) + 
f_{T_1 \cup T_2}(y)$$
et l'égalité a lieu si, et seulement si $T_1 \subset T_2$ ou $T_2
\subset T_1$.
\end{lemme}

\begin{proof}
Soient $s_1$, $s_2$, $s_{1 \cap 2}$ et $s_{1 \cup 2}$ les cardinaux
respectifs de $T_1$, $T_2$, $T_1 \cap T_2$ et $T_1 \cup T_2$. On a
évidemment $s_1 + s_2 = s_{1 \cap 2} + s_{1 \cup 2}$ et $s = s_1 - s_{1
\cap 2} = s_{1 \cup 2} - s_2 \geq 0$. Par ailleurs, quitte à intervertir
$T_1$ et $T_2$, on peut supposer $s_1 \leq s_2$. Un calcul immédiat
conduit alors à
$$\begin{array}{l}
f_{T_1}(y) + f_{T_2}(y) - f_{T_1 \cap T_2}(y) - f_{T_1 \cup T_2}(y) \\
\hspace{1cm} \textstyle = (y_{s_{1\cap 2} + 1} + \cdots + y_{s_1}) -
(y_{s_2 + 1} + \cdots + y_{s_{1 \cup 2}}) + \frac 1 2 \cdot s (b-2)
\cdot (s_{1 \cup 2} + s_2 - s_1 - s_{1 \cap 2}).
\end{array}$$
à partir de quoi, on trouve en utilisant le lemme précédent :
$$\begin{array}{l}
f_{T_1}(y) + f_{T_2}(y) - f_{T_1 \cap T_2}(y) - f_{T_1 \cup T_2}(y) \\
\hspace{1cm}
\geq s \cdot \big[ (b-1)(s_2-s_1) + bs - d + \frac b 2 (s_{1\cup 2}
- s_{1 \cap 2} -s_2 + s_1) \big] \\
\hspace{1cm}
\geq s \cdot \big( (b-1)(s_2 - s_1) + b s - d \big).
\end{array}$$
Si $s > 0$, il est clair que la dernière expression est elle-aussi
strictement positive. L'inégalité du lemme est donc vraie, et même 
stricte, dans ce cas. Si au contraire $s = 0$, le majorant $s \cdot ( 
(b-1)(s_2 - s_1) + b s - d)$ s'annule et l'inégalité du lemme en résulte 
encore. Dans ce dernier cas, l'égalité peut se produire, mais dire que 
$s = 0$ revient à dire que $T_1 \cap T_2 = T_1$, \emph{i.e.} $T_1 
\subset T_2$. Comme il est clair que, réciproquement, si $T_1 \subset 
T_2$, l'inégalité du lemme est une égalité, on a bien démontré ce qui 
avait été annoncé.
\end{proof}

On est maintenant prêt à démontrer la proposition 
\ref{prop:extrAQmaxgl}. On suppose donc $b \geq b_0$. Si l'on excepte le 
cas $b = d = 2$ que l'on vérifie à part à la main, on a $b > d$ et donc 
le lemme précédent s'applique. On considère des parties $T_1, \ldots, 
T_{d-1}$ de $\{1, \ldots, d\}$ vérifiant les hypothèses i) et ii) 
énoncées en amont des lemmes. On pose également $T_d = \{1, \ldots, d\}$ 
et on appelle $y$ l'unique point d'intersection des hyperplans affines
d'équation $f_{T_i}(y) = 0$ pour $i$ variant entre $1$ et $d$. Quitte à 
renuméroter les $T_i$, on peut supposer que leurs cardinaux sont rangés 
par ordre croissant. Comme $y \in A_{Q_\max,g,\ell}$ par hypothèse, les 
nombres $f_{T_i \cup T_j} (y)$ et $f_{T_i \cap T_j} (y)$ sont positifs 
ou nuls pour tous indices $i$ et $j$. Ainsi, on est nécessairement dans 
le cas d'égalité du lemme précédent, ce qui signifie, dans le cas où $i 
< j$, que $T_i \subset T_j$ d'après l'hypothèse supplémentaire que nous 
avons faite sur les cardinaux. Ainsi, on obtient
$$T_1 \subset T_2 \subset \cdots \subset T_{d-1} \subset T_d = \{1, 
\ldots, d\}.$$
Par ailleurs, toutes les inclusions sont strictes car les hyperplans
d'équation $f_{T_i}(y) = 0$ ont un unique point d'intersection et que
deux d'entre eux ne peuvent être confondus. On en déduit qu'il existe 
une permutation $\weyld \in \mathfrak S_d$ telle que $T_i = \{ 
\weyld(1), \weyld(2), \ldots, \weyld(i) \}$ pour tout $i$. Il reste à 
calculer les coordonnées $(y_1, \ldots, y_d)$ du point $y$. Celles-ci 
sont solutions du système d'équations suivant :
\begin{equation}
\label{eq:systAQwgl}
\left\{ \begin{array}{l}
y_1 - b y_{d+1-\weyld(1)} = d-1 - b(\weyld(1) - 1) \\
(y_1 + y_2) - b (y_{d+1-\weyld(1)} + y_{d+1-\weyld(2)}) = d^2-4 - 
b(\weyld(1) + \weyld(2) - 3) \\
\hspace{1cm} \vdots \\
(y_1 + \cdots + y_{d-1}) - b (y_{d+1-\weyld(1)} + \cdots + 
y_{d+1-\weyld(d-1)}) \\
\hspace{2cm} = d^2-(d-1)^2 - b(\weyld(1) + \cdots + \weyld(d-1) 
- \frac{d(d-1)} 2) \\
y_1 + \cdots + y_d = 0 
\end{array} \right.
\end{equation}
En retranchant chaque équation de sa précédente, le système fournit
les équations
$$y_i - b y_{d+1-\weyld(i)} = d + 1 - 2i - b (\weyld(i) - i),
\quad \text{pour } 1 \leq i \leq d.$$
En notant $\weyl$ la permutation inverse de $i \mapsto d+1-\weyld(i)$,
et en posant $z_i = y_i - d - 1 + 2i$ pour tout $i$, celles-ci se
réécrivent $z_i = \weyl(i) - i + \frac{z_{\weyl(i)}} b$ et donnent 
donc :
$$z_i = \weyl(i) - i + \frac{z_{\weyl(i)}} b = \weyl(i) - i + 
\frac{\weyl^2(i) - \weyl(i)} b + \frac{z_{\weyl^2(i)}} {b^2} = \cdots =
\sum_{n=0}^\infty \frac {\weyl^{n+1}(i) - \weyl^n(i)} {b^n}.$$
Un calcul facile montre alors que $y = \vec \rho_\weyl$. Réciproquement,
pour montrer que $\vec \rho_\weyl$ est un point extrémal, si l'on prend
en considération ce que l'on a déjà dit, il reste à montrer qu'il
appartient à $A_{Q_\max,g,\ell}$. Or, pour $T \subsetneq \{1, \ldots, 
d\}$, un nouveau calcul conduit à l'égalité suivante :
\begin{equation}
\label{eq:fTy}
\frac{f_T(\vec \rho_\weyl)}{b-1} = \sum_{n=0}^\infty \frac 1 {b^n} 
\Bigg( \sum_{t \in T} \weyl^{n+1} (d+1-t) - \sum_{i=1}^s \weyl^n(i) 
\Bigg)
\end{equation}
où, comme d'habitude, $s$ est le cardinal de $T$. Chaque facteur
$\sum_{t \in T} \weyl^{n+1} (d+1-t) - \sum_{i=1}^s \weyl^n(i)$ est, en
valeur absolue, inférieur ou égal $s(d-1-s)$, qui est lui-même inférieur
ou égal à $[\frac {(d-1)^2} 4] = b_0 - 1$. Ainsi, dans la somme infinie 
du membre de droite de l'égalité \eqref{eq:fTy}, la contribution des 
termes pour $n \geq 1$ est majorée en valeur absolue par $\frac{b_0-1} 
{b-1} \leq 1$. Par ailleurs, la contribution pour $n = 0$ vaut $\sum_{t 
\in T} \weyl (d+1-t) - \sum_{i=1}^s i$ qui est un nombre entier positif 
ou nul. S'il est strictement positif, il vaut au moins $1$ et d'après ce 
qu'on a démontré précédemment, il ne peut être compensé par le reste de 
la somme ; ainsi on a bien $f_T(\vec \rho_\weyl) \geq 0$ dans ce cas. 
Si, maintenant, $\sum_{t \in T} \weyl (d+1-t) = \sum_{i=1}^s i$, 
l'ensemble des $\weyl(d+1-t)$ pour $t$ parcourant $T$ est nécessairement 
égal à l'ensemble $\{1, \ldots, s\}$. Il en résulte que pour tout $n$, 
on a aussi $\sum_{t \in T} \weyl^{n+1} (d+1-t) = \sum_{i=1}^s 
\weyl^n(i)$, et donc que $f_T(\vec \rho_\weyl)$ est nul dans ce cas. 
Ainsi, pour tout sous-ensemble $T \subsetneq \{1, \ldots, d\}$, on a 
$f_T(y) \geq 0$. Comme on a également $y_1 + \cdots + y_d = 0$, on 
trouve bien $\vec \rho_\weyl \in A_{Q_\max,g,\ell}$, et la proposition 
\ref{prop:extrAQmaxgl} est démontrée.

\begin{lemme}
\label{lem:extrAQmaxgl}
On conserve les hypothèses de la proposition \ref{prop:extrAQmaxgl} et 
les hypothèses de sa démonstration. Alors, si $T$ n'est pas l'un des 
ensembles $T_i$, on a $f_T(\vec \rho_\weyl) > 0$.
\end{lemme}

\noindent
En d'autres termes, le lemme dit qu'il ne concourt au sommet $\vec 
\rho_\weyl$ que $d-1$ faces correspondant aux hyperplans d'équation 
$f_{T_i}(y) = 0$ pour $1 \leq y \leq d-1$. (Notez bien que l'équation 
$f_{T_d}(y) = 0$ définit l'hyperplan \og somme des coordonnées égale 0 
\fg\ dans lequel tout est plongé.)

\begin{proof}
Il suffit de remarquer que dans la somme de la formule \eqref{eq:fTy}, 
la contribution du terme correspondant à $n = d!$ est positive ou nulle 
puisque $\weyl^{d!}$ est l'identité. On en déduit que la contribution
cumulée de tous les termes obtenus avec $n \geq 1$ est en fait 
\emph{strictement} inférieure à $\frac{b_0-1}{b-1}$. Ainsi dès que
l'ensemble des $\weyl(d+1-t)$ ($t \in T$) n'est pas égal à $\{1, \ldots, 
s\}$, c'est-à-dire dès que $T$ n'est pas égal à $T_s$ avec $s = 
\card\:T$, la quantité $f_T(\rho_\weyl)$ est strictement positive.
\end{proof}

\subsubsection*{Un mot sur les points extrémaux de $A_{Q_\max,g,\ell} +
C^\star$}

On démontrera dans la suite que $A_{Q,g,\ell} = A_{Q_\max,g,\ell} + 
C^\star$ (proposition \ref{prop:AQgl}) et, donc, plutôt que les points 
extrémaux de $A_{Q_\max,g,\ell}$, ce sont ceux de la somme 
$A_{Q_\max,g,\ell} + C^\star$ que l'on aimerait décrire. Ceux-ci forment 
un sous-ensemble des points extrémaux de $A_{Q_\max,g,\ell}$ qui, 
malheureusement, semble difficile à comprendre. On peut, malgré tout, 
définir un ordre sur $\mathfrak S_d$ qui permet de mieux appréhender la
situation.

Voici comment on procède. Tout d'abord, pour tout entier $s \in \{1,
\ldots, d\}$, on définit le préordre $\preccurlyeq_s$ par :
$$\weyl_1 \preccurlyeq_s \weyl_2
\quad \text{ssi} \quad
\Bigg(\sum_{i=1}^s \weyl_1^n(i) \Bigg)_{n \geq 1} \leq_\lex 
\Bigg(\sum_{i=1}^s \weyl_2^n(i) \Bigg)_{n \geq 1}$$
où $\leq_\lex$ désigne l'ordre lexicographique sur l'ensemble des
suites, qui donne le plus de poids aux petits indices. On pose ensuite :
$$\weyl_1 \preccurlyeq \weyl_2
\quad \text{ssi} \quad
\weyl_1 \preccurlyeq_s \weyl_2
\text{ pour tout } s \in \{1, \ldots, d\}.$$
On vérifie facilement que $\preccurlyeq$ est bien un ordre sur
$\mathfrak S_d$. À titre d'exemple, la figure \ref{fig:hasse} montre son
diagramme de Hasse pour $d = 4$ (sur cette figure, la permutation $\weyl$
est notée $\perm {\weyl(1)} {\weyl(2)} {\weyl(3)} {\weyl(4)}$). Il existe
déjà un certain nombre d'ordres sur le groupe des permutations $\mathfrak
S_d$ (ordre de Bruhat, ordre faible, \emph{etc}.) mais l'ordre 
$\preccurlyeq$ ne semble coïncider avec aucun d'entre eux (en tout cas, 
pas avec ceux qui sont connus de l'auteur). Pour l'instant, il est 
encore assez mystérieux. La proposition suivante explique malgré tout
comment il est relié avec la problématique de cet article.

\begin{figure}[h]
\begin{center}
\includegraphics[scale=0.6]{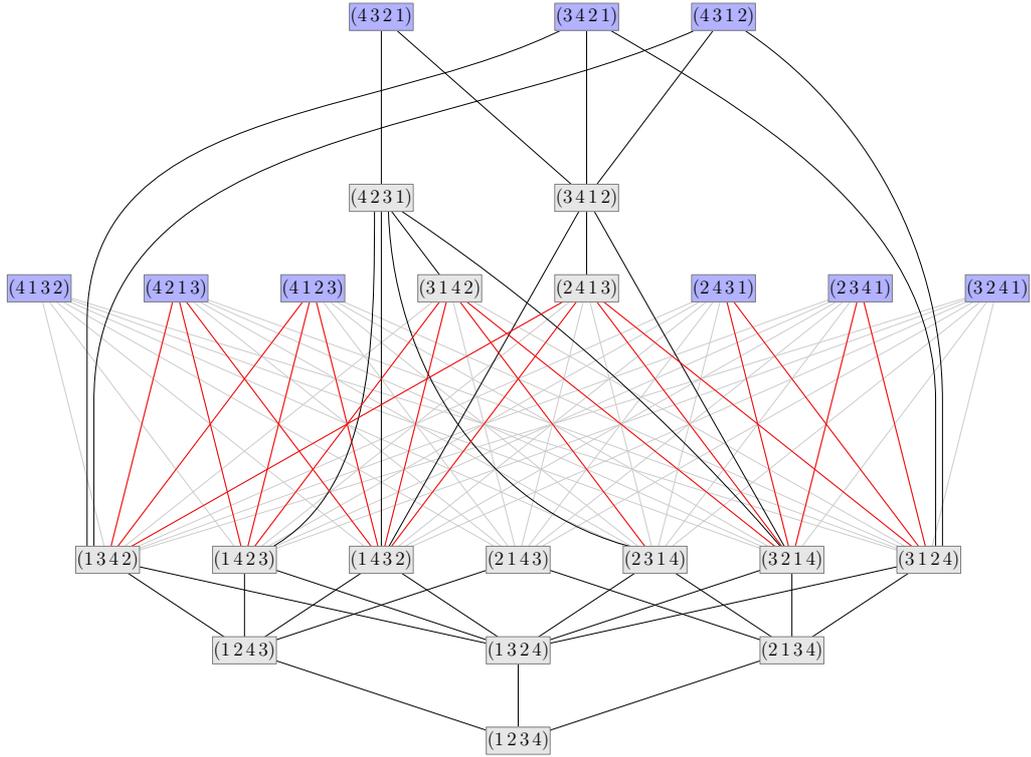}

\caption{Diagramme de Hasse de l'ordre $\preccurlyeq$ sur $\mathfrak
S_4$}
\label{fig:hasse}

\medskip

\begin{minipage}{13cm}
Les traits gris clairs sur la figure représentent, de la même façon que
les traits noirs, des arêtes dans le graphe de Hasse. Attention : afin
de faciliter la lecture, nous avons dessiné en rouge les arêtes entre une
permutation du troisième étage et une du quatrième qui \emph{ne} sont
\emph{pas} dans le diagramme de Hasse !

Les permutations coloriées en bleu sont les éléments maximaux.
\end{minipage}
\end{center}
\end{figure}

\begin{prop}
On suppose $b \geq b_0$. Pour toutes permutations 
$\weyl_1, \weyl_2 \in \mathfrak S_d$, on a $\weyl_1 \preccurlyeq \weyl_2$
si, et seulement si $\vec \rho_{\weyl_1} \in \vec \rho_{\weyl_2} + C^\star$. 
\end{prop}

\begin{proof}
Les coordonnées de $\vec \rho_{\weyl_1} - \vec \rho_{\weyl_2}$ sont
$$(b-1) \cdot \sum_{n=1}^\infty \frac{\weyl_1^n(i) - \weyl_2^n(i)}{b^n}$$
avec $1 \leq i \leq d$. Dire que ce vecteur est dans $C^\star$ revient 
à dire que, pour tout $s \in \{1, \ldots, d\}$, on a :
$$\sum_{n=1}^\infty \frac 1 {b^n} \sum_{i=1}^s (\weyl_1^n(i) - 
\weyl_2^n(i)) \geq 0$$
et que l'égalité a lieu si $s = d$. En fait, il est clair que l'égalité
est toujours vérifiée pour $s = d$ ; on peut donc oublier cette
condition. À part cela, on remarque que la somme $\sum_{i=1}^s 
\weyl_1^n(i) - \weyl_2^n(i)$ est majorée en valeur absolue par 
$b_0-1$. Comme $\frac{b_0-1} {b-1} \leq 1$, le signe de la somme 
infinie ne dépend que du signe du premier terme non nul. La proposition 
découle de cela.
\end{proof}

\noindent
Il résulte de la proposition que si deux permutations distinctes
$\weyl_1$ et
$\weyl_2$ sont telles que $\weyl_1 \preccurlyeq \weyl_2$, alors $\vec
\rho_{\weyl_1}$ ne peut être un point extrémal de $A_{Q_\max,g,\ell} +
C^\star$ (toujours sous l'hypothèse $b \geq b_0$). Ainsi,
les seuls points extrémaux envisageables correspondent aux permutations
$\weyl$ qui sont des éléments maximaux pour $\preccurlyeq$. Par contre,
il se peut que certains éléments maximaux ne définissent pas des points
extrémaux de $A_{Q_\max,g,\ell} + C^\star$ ; lorsque $d = 4$ par
exemple, c'est le cas des vecteurs correspondant aux permutations $\perm
2 4 3 1$ et $\perm 4 2 1 3$ (qui sont bien des éléments maximaux comme
on le voit sur la figure \ref{fig:hasse}).

\subsection{Le calcul de $A_{Q,g,\ell}$}
\label{subsec:AQgl}

Le but de cette partie est de démontrer la proposition suivante.

\begin{prop}
\label{prop:AQgl}
On suppose $b \geq b_0$. Alors $A_{Q,g,\ell} = A_{Q_\max,g,\ell} + 
C^\star$.
\end{prop}

L'inclusion $A_{Q_\max,g,\ell} + C^\star \subset A_{Q,g,\ell}$ est facile. 
En effet, il est clair déjà que $A_{Q_\max,g,\ell}$ est inclus dans
$A_{Q,g,\ell}$. Soient maintenant $y = (y_1, \ldots, y_d) \in
A_{Q,g,\ell}$ et $z = (z_1, \ldots, z_d) \in C^\star$. Il suffit de montrer
que $y + z \in A_{Q,g,\ell}$. Par définition de $A_{Q,g,\ell}$, on a
$y_1 \vec \mu_1 + \cdots + y_d \vec \mu_d - \vec \delta \in Q^\star$, ce
qui revient à dire que pour tout $q \in Q$,
\begin{equation}
\label{eq:yiAQgl}
y_1 \left< \vec \mu_1 | q \right>_E + \cdots + y_d \left< \vec \mu_d |
q \right>_E \geq \left< \vect \delta | q \right>_E.
\end{equation}
Par ailleurs, on sait que la suite des $\left< \vec \mu_j | q \right>_E$
est décroissante en $j$ ; autrement dit elle définit un vecteur de
$\R^d$ qui appartient à $C$. On en déduit que 
\begin{equation}
\label{eq:ciAQgl}
z_1 \left< \vec \mu_1 | q \right>_E + \cdots + z_d \left< \vec
\mu_d | q \right>_E \geq 0.
\end{equation}
Finalement en additionnant les formules \eqref{eq:yiAQgl} et
\eqref{eq:ciAQgl}, on trouve bien $y + z \in A_{Q,g,\ell}$ comme
annoncé.

\medskip
On se concentre désormais sur la démonstration de l'inclusion
réciproque. À partir de maintenant, on suppose donc $b \geq b_0$. La 
démarche générale est la suivante. Tout d'abord, pour tout point 
extrémal $\vec \rho_\weyl$ de $A_{Q_\max,g,\ell}$, on note 
$D_\weyl$ le cône convexe tel qu'au voisinage de $\vec \rho_\weyl$, on 
ait $A_{Q_\max,g,\ell} = \vec \rho_\weyl + D_\weyl$. Si $\vec 
\rho_\weyl$ est un point extrémal de $A_{Q_\max,g,\ell} + C^\star$, les 
ensembles convexes $A_{Q_\max,g,\ell} + C^\star$ et $\vec \rho_\weyl + 
D_\weyl + C^\star$ coïncident encore au voisinage de $\vec \rho_\weyl$. 
On en déduit que
$$A_{Q_\max,g,\ell} + C^\star = \bigcap_{\weyl \in \mathcal S_d} (\vec 
\rho_\weyl + D_\weyl + C^\star)$$
où l'on a noté $\mathcal S_d$ le sous-ensemble de $\mathfrak S_d$
correspondant aux points extrémaux de $A_{Q_\max,g,\ell} + C^\star$.
Ainsi, il suffit de montrer que tout élément de $A_{Q,g,\ell}$
appartient à $\vec \rho_\weyl + D_\weyl + C^\star$ pour tout
$\weyl \in \mathcal S_d$. On va en fait montrer que c'est le cas
pour tout $\weyl \in \mathfrak S_d$. Pour ce faire, l'idée est 
d'interpréter la somme $\vec \rho_\weyl + D_\weyl + C^\star$ comme des 
$A_{Q_\weyl,g,\ell}$ pour certaines cônes convexes $Q_\weyl \subset E$.
Il suffira alors pour conclure de montrer que tous les $Q_\weyl$ sont
inclus dans $Q$.

\subsubsection{Les ensembles $Q_\weyl$}

On fixe $\weyl$ une permutation $\weyl$ de l'ensemble $\{1, \ldots, 
d\}$. En s'inspirant du \S \ref{subsec:extrAQmaxgl}, on définit la
permutation $\weyld : i \mapsto d + 1 - \weyl^{-1}(i)$ et, pour tout $s$ 
compris entre $1$ et $d$, on considère l'ensemble
$$T_s(\weyl) = \{ \weyld(1), \weyld(2), \ldots, \weyld(s) \}.$$
Cet ensemble correspond à une partie admissible de $I$ notée
$I_s(\weyl)$. On rappelle que, $s$ étant fixé, si on note $t_1, \ldots,
t_s$ les nombres $\weyld(1), \ldots, \weyld(s)$ triés par ordre
décroissant (\emph{i.e.} $t_1 > t_2 > \cdots > t_s$), alors le couple
$(i,j)$ est dans $I_s(\weyl)$ si, et seulement si $i \leq s$ et 
$j > d-t_i$. Lorsque $\weyl$ est la permutation $\omega : i \mapsto
d+1-i$, on a $\weyld = \id$ et donc $T_s(\omega) = \{1, \ldots, s\}$ ;
l'ensemble admissible $T_s(\omega)$ est alors simplement l'ensemble
$I_s = \{ (i,j) \in I \, | \, j-i \geq d-s \}$ qui a déjà été
considérée.
De façon générale, on vérifie facilement que les ensembles $I_s(\weyl)$ 
définissent une partition croissante de $I$ telle que $I_d(\weyl) = I$. 
On considère la fonction $\ord_\weyl : I \to \{1, \ldots, d\}$ qui à un 
couple $(i,j)$ associe le plus petit entier $s$ tel que $(i,j) \in 
I_s(\weyl)$. De façon évidente, pour tout $s$, l'ensemble $I_s(\weyl)$ 
regroupe les éléments $x \in I$ tels que $\ord_\weyl(x) \leq s$.

Pour se représenter les constructions précédentes, il est
commode de considérer les éléments $(i,j)$ comme les cases d'un tableau
triangulaire (voir figure \ref{fig:ord}). La fonction $\ord_\weyl$
correspond alors à un remplissage des cases du tableau par les nombres
entiers compris entre $1$ et $d$. Celui-ci s'obtient en fait très
simplement à partir de la permutation $\weyl$, comme suit. On place
d'abord sur la ligne du haut et dans les $\weyl(1)$ dernières colonnes
du tableau le nombre $1$. Ensuite, on place le nombre $2$ dans les
$\weyl(2)$ dernières colonnes du tableau à chaque fois dans la case la
plus haute qui n'est pas déjà remplie. Ainsi, si $\weyl(2) < \weyl(1)$,
tous les nombres $2$ se retrouvent sur la deuxième ligne, tandis que si
$\weyl(2) > \weyl(1)$, on écrit $\weyl(1)$ nombres $2$ sous les $1$ déjà
écrits à l'étape précédente, et on met les $\weyl(2)-\weyl(1)$ nombres
$2$ restants sur la première ligne. On continue ensuite avec les $3$ :
on les place dans les $\weyl(3)$ dernières colonnes, toujours le plus
haut possible. Et ainsi de suite jusqu'à $d$. La figure \ref{fig:ord}
montre le remplissage obtenue pour la permutation $(3\,2\,4\,5\,1)$
(pour $d = 5$ donc).

\begin{figure}
\begin{center}
\includegraphics[scale=0.9]{young.ps}
\end{center}

\caption{La fonction $\ord_{(3\,2\,4\,5\,1)}$}
\label{fig:ord}
\end{figure}

On définit encore un graphe $I_\weyl$ comme suit : ses sommets sont les
éléments de $I$ et l'on convient qu'il y a une arête entre $x$ et $y$
dans ce graphe si $\ord_\weyl(x) \geq \ord_\weyl(y)$. Finalement, on
introduit l'ensemble $Q'_\weyl \subset E$ : c'est le cône convexe
$Q_{I_\weyl}$ associé au graphe $I_\weyl$ par la recette donnée juste en
dessous de la définition \ref{def:admissibleg}. Pour $\weyl_0 : i 
\mapsto d+1-i$, l'ensemble $Q_{\weyl_0}$ n'est autre que l'ensemble
$Q_\min$.

\begin{lemme}
\label{lem:AQptau}
On suppose $b \geq b_0$.
Alors, on a $A_{Q'_\weyl,g,\ell} = \vec \rho_\weyl + D_\weyl$.
\end{lemme}

\begin{proof}
Les parties admissibles du graphe $I_\weyl$ (dans le sens de la
définition \ref{def:admissibleg}) sont exactement les $I_s(\weyl)$ pour
$1 \leq s \leq d$. 
La proposition \ref{prop:dualgraphe} assure donc que le cône dual
$(Q'_\weyl)^\star$ est défini dans $E$ par les équations :
$$\sum_{(i,j) \in I} x_{i,j} = 0 \quad \text{et} \quad
\forall s \in \{1, \ldots, s\},\, \sum_{(i,j) \in I_s(\weyl)} 
x_{i,j} \leq 0.$$
où les $x_{i,j}$ sont les coordonnées canoniques sur $E = \R^I$. En
injectant cela dans la formule \eqref{eq:AQfl} et en utilisant le lemme
\ref{lem:calculST}, on trouve qu'un élément $y = (y_1, \ldots, y_d) \in
\R^d$ appartient à $A_{Q'_\weyl,g,\ell}$ si, et seulement si $y_1 +
\ldots + y_d = 0$ et $f_{I_s(\weyl)}(y)\geq 0$ pour tout $s$ (où la
fonction $f_{I_s(\weyl)}$ est défini comme précédemment, voir formule
\eqref{eq:fT}).

Par ailleurs, étant donné que $b \geq b_0$, le lemme 
\ref{lem:extrAQmaxgl} s'applique et implique
que le cône $D_\weyl$ est défini par les équations et inéquations :
$$y_1 + \cdots + y_d = 0 \quad \text{et} \quad g_{I_s(\weyl)} (y) = 0$$
où, si $T$ est une partie de $\{1, \ldots, d\}$, on a noté $g_T$ la
fonction linéaire associé à $f_T$ : pour tout $y = (y_1, \ldots, y_n)$
dans $\R^d$, on a $g_T(y) = y_1 + \cdots + y_s - b \cdot \sum_{t \in T}
y_{d+1-t}$ avec $s = \card\,T$. Comme $\vec \rho_\weyl$ a été justement
construit pour vérifier $f_{I_s(\weyl)}(\vec \rho_\weyl) = 0$ pour tout
$s$, l'égalité du lemme en découle.
\end{proof}

Soit $D$ le cône convexe de $E$ défini par $D = \{ (q_{i,j}) \in E \, | 
\, \mu_{1,1} \geq \mu_{1,2} \geq \cdots \geq \mu_{1,d} \, \}$ où les 
$\mu_{i,j}$ sont définis à partir des $q_{i,j}$ par la formule 
\eqref{eq:rele3}. On rappelle que l'on a défini dans le \S 
\ref{subsec:vecmui} des vecteurs $\vec \mu_j \in \R^d$ tels que 
$\mu_{1,j} = \left< \vec \mu_j | q \right>_E$ où $q = (q_{i,j}) \in E$. 
Ainsi $D^\star$, le cône dual de $D$, n'est autre que le cône convexe 
engendré par les vecteurs $\vec \mu_j - \vec \mu_{j+1}$ pour $j$ 
parcourant l'ensemble d'indices $\{1, \ldots, d-1\}$. Par ailleurs, 
l'application $g$ est définie par $q \mapsto (\left< \vec \mu_1 | q 
\right>_E, \ldots, \left< \vec \mu_1 | q \right>_E)$ et donc envoie 
$D^\star$ sur l'ensemble des $y = (y_1, \ldots, y_d) \in \R^d$ tels que 
$y_1 \geq \cdots \geq y_d$, c'est-à-dire $C$. On pose enfin $Q_\weyl = 
Q'_\weyl \cap D \subset E$. Le cône dual de $Q_\weyl$ est alors égal à 
$(Q'_\weyl)^\star + D^\star$. À partir de là et de ce qui a été dit 
précédemment (et notamment du lemme \ref{lem:AQptau}), il suit, en 
déroulant les définitions, que l'ensemble $A_{Q_\weyl,g,\ell}$ est égal 
à $\vec \rho_\weyl + D_\weyl + C^\star$ comme souhaité. Il ne reste donc 
plus qu'à démontrer que $Q_\weyl \subset Q$.

\subsubsection{Intermède : la permutation des perdants}
\label{subsec:perdants}

On constate sur l'exemple de la figure \ref{fig:ord} que si l'on retire
la ligne du haut du damier triangulaire, et que l'on soustrait $1$ à
tous les nombres restants, on obtient une numérotation qui correspond à
une nouvelle permutation $\weyl'$, qui est ici égale à $(2\,3\,4\,1)$.
Dans ce paragraphe, on montre que cela est vrai de façon générale et
on explique comment obtenir la permutation $\weyl'$ à partir de $\weyl$.

\begin{deftn}
Soit $\weyl$ une permutation de l'ensemble $\{1, \ldots, d\}$.

Un \emph{record} de $\weyl$ est un entier $\weyl(i)$ tel que $\weyl(j) <
\weyl(i)$ pour tout $j < i$. Si $\weyl(i)$ est un record de $\weyl$, on
dit qu'il apparaît en \emph{position} $i$.

La \emph{permutation des perdants} de $\weyl$ est la permutation 
$\weyl'$ de $\{1, \ldots, d-1\}$ définie par récurrence en décrétant
que $\weyl'(i)$ est le plus petit élément de l'ensemble différence
\begin{equation}
\label{eq:perdants}
\big\{ \, \weyl(1), \weyl(2), \ldots, \weyl(i+1) \, \big\} 
\, \backslash \, \big\{ \, \weyl'(1), \ldots, \weyl'(i-1) \, \big\}.
\end{equation}
\end{deftn}

\noindent
La notion de record est classique : elle a déjà été
introduite il y a de nombreuses années dans \cite{renyi} et a depuis
fait l'objet de multiples études, notamment en ce qui concerne leur
distribution asymptotique. Par contre, l'auteur n'a pas réussi à trouver
une trace antérieure de la permutation des perdants.

\medskip

On montre immédiatement par récurrence que l'ensemble 
\eqref{eq:perdants} est toujours de cardinal $2$ (\emph{i.e.} que 
l'ensemble que l'on ôte est toujouts inclus dans le premier), et plus 
précisément qu'il contient exactement l'élément $\weyl(i+1)$, et le 
dernier record de $\weyl$ apparaissant avant $i$. Ainsi si $\weyl(i+1)$ 
n'est pas un record de $\weyl$, on a $\weyl'(i) = \weyl(i+1)$, tandis 
que dans le cas contraire, $\weyl'(i)$ est le record précédent de 
$\weyl$. Dans le cas de la permutation $\weyl = (3\,2\,4\,5\,1)$, on 
voit que les records de $\weyl$ sont les entiers $3$, $4$ et $5$ et que 
la permutation est perdants de $\weyl$ est $\weyl' = (2\,3\,4\,1)$.

\medskip

On peut donner une reformulation moins mathématique des définitions
précédentes qui donne tout son sens à la terminologie. Il faut pour cela
imaginer que les entiers de $1$ à $d$ sont des candidats qui prennent
part à une compétition, dans l'ordre indiqué par la permutation $\weyl$
: si l'on reprend notre exemple, cela signifie que $3$ joue d'abord, $2$
juste après, \emph{etc}. En outre un candidat est d'autant plus fort au
jeu que l'entier qui lui est attaché est grand. Les records
correspondent alors aux records au sens usuel : l'entier $3$ joue en
premier, et donc décroche le record (il n'a pas grand mérite, mais peu
importe) ; ensuite vient le tour de l'entier $2$ dont la performance est
moins bonne, il n'a donc pas le record ; ensuite, joue $4$ qui fait un
nouveau record ; puis $5$ qui bat encore le record ; et enfin $1$ qui
n'améliore certainement pas le record. Les records successifs sont donc
bien $3$, $4$ et $5$. 

La permutation des perdants, quand à elle, s'interprète comme suit. Il
faut imaginer qu'au fur et à mesure que la compétition se déroule, on
met à jour une liste des perdants. Au premier tour, l'entier $3$ joue et
il n'y a pour l'instant aucun perdant ; on écrit donc rien sur la
liste. Ensuite, c'est au tour de l'entier $2$ de jouer ; celui-ci fait
un moins bon résultat et se retrouve ainsi être le premier perdant.
C'est maintenant $4$ qui s'élance, et il subtilise le record à $3$ ;
l'entier $3$ devient comme ceci un perdant et on l'inscrit sur la liste
en dessous de $2$. Et ainsi de suite, on obtient la liste des perdants
--- ou la permutation des perdants pour reprendre la terminologie 
mathématique --- composée dans l'ordre des nombres $2$, $3$, $4$ et
$1$.

Dans la suite, et notamment lors des démonstrations, on continuera 
d'employer la terminologie imagée issue de la métaphore de la 
compétition.

\begin{lemme}
\label{lem:youngperdants}
Soient $\weyl \in \mathfrak S_d$ et $\weyl'$ sa permutation des perdants.
Soit $\ord_{\weyl'} : I' \to \{1, \ldots, d-1\}$ (avec $I' = \{ (i,j) \in
\N^2 \, | \, 1 \leq i \leq j \leq d-1 \}$) la fonction associée à
$\weyl'$. Alors $\ord_{\weyl'}(i,j) = \ord_\weyl(i+1,j+1) - 1$.
\end{lemme}

\begin{proof}
On raisonne sur la représentation de la fonction $\ord_\weyl$ sous forme 
de tableau triangulaire (voir figure \ref{fig:ord}). 
Il s'agit alors de montrer qu'un entier $i \geq 2$ apparaît exactement
$\weyl(i) - \weyl'(i-1)$ fois sur la première ligne. Si $\weyl(i)$ n'est
pas un record, alors manifestement $i$ n'apparaît pas sur la première
ligne et on a déjà dit qu'alors $\weyl(i) = \weyl'(i-1)$. Dans ce cas, 
on a donc bien ce que l'on voulait. Si, au contraire, $\weyl(i)$ est un
record, soit $\weyl(j)$ le record de $\weyl$ qui apparaît juste avant
$i$, \emph{i.e.} $\weyl(j) = \max \{ \weyl(1), \ldots, \weyl(i-1) \}$.
Alors $\weyl'(i-1) = \weyl(j)$ et l'entier $i$ apparaît $\weyl(i) - 
\weyl(j)$ sur la première ligne du tableau. On a donc, à nouveau, bien
vérifié ce que l'on avait annoncé.
\end{proof}

\subsubsection{La fin de la démonstration}

On en revient à la démonstration de la proposition \ref{prop:AQgl}.
Il reste à démontrer que $Q_\weyl \subset Q$ pour toute permutation 
$\weyl \in \mathfrak S_d$.
On fixe à partir de maintenant $\weyl \in \mathfrak S_d$ et un élément $q
= (q_{i,j}) \in Q_\weyl$. On note également $\mu_{i,j}$ les nombres réels
définis par la formule \eqref{eq:rele3}. Par hypothèse $\mu_{1,j} \geq
\mu_{2,j} \geq \cdots \geq \mu_{d,j}$ et, en examinant la définition de
$Q'_\weyl$, on prouve qu'il existe des réels $q_1 \leq q_2 \leq \cdots
\leq q_d$ tels que $q_{i,j} = q_{\ord_\weyl(i,j)}$. On veut montrer
que $q \in Q$, ce qui signifie que les inégalités $\mu_{i,j} \leq
\mu_{i+1,j}$ et $\mu_{i,j} \geq \mu_{i+1,j+1}$ sont satisfaites pour
dès que cela a un sens. 
La première étape consiste à exprimer les $\mu_{i,j}$ en fonction des
$q_i$. Pour cela, on considère $\weyl_1, \ldots, \weyl_d$ les 
permutations
définies par récurrence en convenant que $\weyl_1 = \weyl$ et que
$\weyl_{i+1}$ est la permutation des perdants de $\weyl_i$ ; ainsi
$\weyl_i$ est une permutation de l'ensemble $\{1, \ldots, d+1-i\}$. 
On définit encore les permutations $\weyldd_i \in \mathfrak S_{d+1-i}$
par $\weyldd_i(j) = \weyl_i^{-1}(d+2-i-j)$ où $\weyl_i^{-1}$ désigne bien 
sûr la permutation inverse de $\weyl_i$.

\begin{lemme}
\label{lem:exprmuij}
Avec les notations précédentes, on a pour tout couple $(i,j) \in I$ :
$$\mu_{i,j} = b q_{\weyldd_i(j-i+1)+i-1} - q_j.$$
\end{lemme}

\begin{proof}
D'après le lemme \ref{lem:youngperdants} et les expressions donnant les
valeurs de $\mu_{i,j}$, il suffit de traiter le cas où $i = 1$. La 
formule à démontrer s'écrit alors simplement $\mu_{1,j} = b 
q_{\weyldd_1(j)} - q_j$. Par définition, on a :
\begin{eqnarray*}
\mu_{1,j} & = & -q_{j,d} + b \cdot \sum_{i=1}^j q_{i,j} - b \cdot
\sum_{i=0}^{j-1} q_{i,j-1} \\
& = & -q_{\ord_\weyl(j,d)} + b \cdot \sum_{i=1}^j q_{\ord_\weyl(i,j)}
- b \cdot \sum_{i=0}^{j-1} q_{\ord_\weyl(i,j-1)}.
\end{eqnarray*}
Or, si l'on se souvient de l'interprétation de
la fonction $\ord_\weyl$ en termes de remplissage de tableau, il est
clair que $\ord_\weyl(j,d) = j$ et que l'ensemble $\{\ord_\weyl(1,j),
\ldots, \ord_\weyl(j,j)\}$ consiste en les entiers $s$ tels que
$\weyl(s) \geq d+1-j$. Ainsi, trouve-t-on :
$$ \{\ord_\weyl(1,j), \ldots, \ord_\weyl(j,j)\} = \{\ord_\weyl(1,j-1),
\ldots, \ord_\weyl(j-1,j-1)\} \cup \{ \weyldd_1(j) \}$$
et la formule du lemme en découle.
\end{proof}

On est en mesure à présent de montrer que $\mu_{1,j} \leq \mu_{2,j}$ 
pour tout indice $j \in \{2, \ldots, d\}$. Par définition de 
$\weyldd_2$, on a $\weyl_2 \circ \weyldd_2 (j-1) = d+1-j$ ; autrement 
dit, l'entier $d+1-j$ apparaît en position $\weyldd_2(j-1)$ dans la 
permutation des perdants de $\weyl$. Il a donc forcément joué avant le 
tour $\weyldd_2(j-1)+1$, ce qui se traduit mathématiquement par 
l'inégalité $\weyldd_1(j) \leq \weyldd_2 (j-1) + 1$. Ainsi, grâce au 
lemme \ref{lem:exprmuij}, on obtient :
$$\mu_{1,j} = b q_{\weyldd_1(j)} - q_j \leq b q_{\weyldd_2(j-1) + 1}
- q_j = \mu_{2,j}$$
ce qui est bien ce que l'on désirait. La démonstration de l'inégalité
$\mu_{1,j} \geq \mu_{2,j+1}$ suit une idée analogue. On remarque d'abord
que, d'après le lemme \ref{lem:exprmuij}, la différence entre les deux
nombres à comparer s'exprime comme suit :
$$\mu_{1,j} - \mu_{2,j+1} = b (q_{\weyldd_1(j)} - q_{\weyldd_2(j)+1})
+ q_{j+1} - q_j.$$
Comme $q_{j+1} \geq q_j$, l'égalité que l'on souhaite démontrer est
trivialement satisfaite si $q_{\weyldd_1(j)} \geq q_{\weyldd_2(j)+1}$ et
donc dès que $\weyldd_1(j) > \weyldd_2(j)$. On suppose donc à partir de
maintenant que $\weyldd_1(j) \leq \weyldd_2(j)$. On pose $i_1 = \weyldd_1
(j)$ et $i_2 = \weyldd_2(j)$ ; on a alors $\weyl_1(i_1) = d+1-j$ et 
$\weyl_2(i_2) = d-j$. Selon la permutation $\weyl$, l'entier $d-j$ a
donc forcément joué avant le temps $i_2 + 1$. Par ailleurs, il n'a
pas pu jouer avant le temps $i_1$ car sinon, il aurait été battu par
$d+1-j$ au temps $i_1$ et donc aurait au pire rejoint la liste des
perdants en position $i_1 - 1$. De même, il n'a pas pu jouer en 
$i_1$ car c'était alors le tour de $d+1-j$, ni entre les temps 
$i_1+1$ et $i_2$ car il aurait été alors perdant tout de suite
(étant donné que $d+1-i$ aurait déjà fait une meilleure performance 
avant). On en déduit que $d-j$ a joué au temps $i_2 + 1$, c'est-à-dire
que $\weyl(i_2 + 1) = d-j$, soit encore $\weyldd_1(j+1) = i_2 + 1 = 
\weyldd_2(j) + 1$. 
En appliquant le lemme \ref{lem:exprmuij}, on obtient
$$\mu_{1,j+1} = b q_{\weyldd_1(j+1)} - q_{j+1} =
b q_{\weyldd_2(j) + 1} - q_{j+1} = \mu_{2,j+1}$$
et la conclusion résulte alors de l'hypothèse $\mu_{1,j} \geq 
\mu_{1,j+1}$.

À présent, des deux égalités que l'on vient de prouver, il suit
$\mu_{2,j} \geq \mu_{1,j} \geq \mu_{2,j+1}$ pour tout $j \in \{2,
\ldots, d-1\}$. On peut donc reitérer l'argumentation précédente en
décalant les indices (ou, si l'on préfère, appliquer ce que l'on vient
de démontrer à la permutation $\weyl'$) afin d'obtenir les inégalités
$\mu_{2,j} \leq \mu_{3,j}$ pour $3 \leq j \leq d$, et $\mu_{2,j} \geq
\mu_{3,j+1}$ pour $2 \leq j \leq d-1$. En continuant ainsi, on démontre
bien au final ce que l'on voulait.

\subsection{Démonstration des théorèmes \ref{theo:dimcaruso},
\ref{theo:dimdivelem2} et \ref{theo:dimdivelem}}

\subsubsection{Le cas des variétés $\calX_\mu$}

On commence par démontrer le théorème \ref{theo:dimdivelem2}. 
Soit $\mu = (\mu_1, \ldots, \mu_d)$ un $d$-uplet d'entiers vérifiant 
$\mu_1 \geq \cdots \geq \mu_d$. En passant au
déterminant, on démontre tout de suite que si $b-1$ ne divise pas $\mu_1
+ \cdots + \mu_d$, la variété $\calX_\mu$ est vide. On suppose désormais
que $\mu_1 + \cdots + \mu_d$ est un multiple de $b-1$. On a 
$$\dim_k \calX_\mu = \max \dim_k \calX_\varphi$$
où le maximum est pris sur tous les $\varphi \in \Phi_\Z$ tels que
$\mu_{1,i}(\varphi) = \mu_i$ pour tout $i$. Les congruences énoncées
sont donc une conséquence immédiate du corollaire \ref{cor:dim} et du
théorème \ref{theo:dim}. Pour le reste, on rappelle que l'on se 
restreint au cas où $h \geq 0$. On a donc
$\dim_k \calX_\mu = b'_{Q,R,g,\ell,0} (\mu)$, et, en vertu de la
proposition \ref{prop:bprime}, on est amené à justifier les deux
assertions suivantes :
\begin{itemize}
\item l'image de $R \cap Q_\R$ (où on rappelle que $Q_\R$ est défini 
comme le $\R$-espace vectoriel engendré par $Q$) par la fonction $g$ 
contient
l'ensemble des $d$-uplets $(y_1, \ldots, y_d)$ tels que $b-1$ divise
$y_1 + \cdots + y_d$ ;
\item le nombre $b_{Q,g,\ell,0}(\mu) = a_{Q,g,\ell}(\mu)$ est égal au
minimum qui apparaît dans l'énoncé du théorème \ref{theo:dimdivelem2}.
\end{itemize}
Il est facile de construire un élement $q = (q_{i,j})$ telle que toutes 
les inégalités
des jeux I, II et III définissant $Q$ soient strictes ; on pourra pour
cela s'inspirer de l'exemple de la figure \ref{fig:exemplevarphi}. Un
tel point est dans l'intérieur de $Q$, ce qui assure que $Q$ est
d'intérieur non vide et donc que $Q_\R = E$. Si l'on note $\vec v_1,
\ldots, \vec v_d$ la base canonique de $\R^d$, un calcul facile montre
qu'étant donné un entier $s$ compris entre $1$ et $d$, l'image par $g$
de l'élément $(q_{i,j}) \in R$ défini par $q_{i,j} =
\delta_{(i,j),(s,d)}$ ($\delta$ étant le symbole de Kronecker) est $\vec
w_s = b \vec v_d - \vec v_s$. En outre, un vecteur $y$ de coordonnées
$(y_1, \ldots, y_d)$ se décompose sur la familles de $\vec w_s$ comme
suit :
$$y = \frac b{b-1} \cdot (y_1 + \cdots + y_d) \cdot \vec w_d - (y_1 
\vec w_1 + \cdots + y_d \vec w_d).$$
Ainsi on trouve que $y$ appartient à $g(R) = g(R \cap Q_\R)$ dès que 
$b-1$ divise la somme de ses coordonnées $y_1 + \cdots + y_d$, ce qui
démontre la première assertion.

On en vient à la seconde. Comme l'on ne considère que des éléments $\mu$ 
dans $C$, on a $a_{Q,g,\ell}(\mu) = b_{Q,g,\ell,0, C} (\mu)$ et de même 
$a_{Q_\max,g,\ell}(\mu) = b_{Q_\max,g,\ell,0, C} (\mu)$. D'après la 
proposition \ref{prop:dualconv2} et la description des points extrémaux 
de $A_{Q_\max,g,\ell}$ donnée par la proposition \ref{prop:extrAQmaxgl}, 
il suffit de montrer que $B_{Q,g,\ell,0,C} = B_{Q_\max,g,\ell,0,C}$, 
\emph{i.e.} $A_{Q,g,\ell} + C^\star = A_{Q_\max,g,\ell} + C^\star$, ce 
qui suit de la proposition \ref{prop:AQgl}.

\begin{rem}
Bien entendu, $a_{Q,g,\ell}(\mu)$ s'écrit aussi comme le minimum des
produits scalaires $\left< \alpha | \mu \right>_d$ où $\alpha$ décrit
l'ensemble des points extrémaux de $A_{Q,g,\ell}$. Cela signifie que le
théorème \ref{theo:dimdivelem2} reste vrai si l'on se contente de
prendre le minimum sur le sous-ensemble (strict) $\mathcal S_d$ de
$\mathfrak S_d$, qui est défini au début du \S \ref{subsec:AQgl}. En
d'autres termes, les permutations qui n'appartiennent pas à $\mathcal
S_d$ ne contribuent jamais (\emph{i.e.} pour aucun $\mu \in C$) au
minimum.
\end{rem}

Il reste encore à démontrer que lorsque $\mu$ est $b$-régulier,
le minimum des produits scalaires $\left< \vect \rho_\weyl | \mu
\right>_d$ (pour $\weyl$ décrivant $\mathfrak S_d$) est atteint lorsque
$\weyl = \weyl_0 : i \mapsto d+1-i$, et qu'il vaut alors $\frac 1{b+1}
\cdot \left< 2 \vect \rho | \mu \right>_d$ (on rappelle que $\vec \rho$
est défini par $\vec \rho = (\frac{d-1} 2, \frac{d-3} 2, \ldots,
\frac{1-d} 2) \in \R^d$). On note pour cela $\Reg$ le sous-ensemble de
$\R^d$ formé des éléments $\mu = (\mu_1, \ldots, \mu_d)$ qui sont
$b$-réguliers ; il s'agit d'un cône convexe dont le cône dual est
noté $\Reg^\star$. Avec la proposition \ref{prop:dualconv2}, on voit 
aisément qu'il suffit de démontrer le lemme suivant.

\begin{lemme}
\label{lem:AQmingl}
Soit $\weyl_0$ la permutation $i \mapsto d+1-i$. On a l'égalité :
$$A_{Q_{\weyl_0},g,\ell} = A_{Q_\min,g,\ell} = \frac{2 \vec \rho}{b+1} + 
\Reg^\star.$$
\end{lemme}

\begin{proof}
On a déjà dit que les convexes $A_{Q_{\weyl_0},g,\ell}$ et 
$A_{Q_\min,g,\ell}$ étaient égaux. Pour $i \in \{1, \ldots, d-1\}$, 
on note $\vec v_i \in \R^d$ le vecteur
$\vec v_i = (0, \ldots, 0, 1, -1, 0 \ldots, 0)$ où le $1$ est en $i$-ème
position. Ces vecteurs forment une base de l'hyperplan \og somme des
coordonnées égale $0$ \fg, noté $H$. D'après la définition des points
$b$-réguliers, le cône dual $\Reg^\star$ est le cône convexe engendré
par les vecteurs $\vec w_i = b\,\vec v_{d-i} - \vec v_i$ pour $1 \leq i
\leq d-1$. Par ailleurs, comme dans la démonstration du lemme
\ref{lem:AQptau}, on obtient qu'un élément $y = (y_1, \ldots, y_d) \in 
\R^d$ appartient à $A_{Q_{\weyl_0},g,\ell}$ si, et seulement si
$$\left\{\begin{array}{l}
y_1 + y_2 + \cdots + y_d = 0 \\
(y_1 + \cdots + y_s) - b (y_{d+1-s} + \cdots + y_d) \geq s(d-s),
\quad \forall s \in \{1, \ldots, d-1\}. 
\end{array} \right. $$
Comme le vecteur $\frac{2 \vec \rho}{b+1}$ n'est autre que
l'intersection de ces $d$ hyperplans (vérification immédiate), le
convexe $-\frac{2 \vec \rho}{b+1} + A_{Q_{\weyl_0},g,\ell}$ est le cône
défini dans $H$ par les inégalités $(y_1 + \cdots + y_s) - b (y_{d+1-s}
+ \cdots + y_d) \geq 0$ pour $s \in \{1, \ldots, d-1\}$. Pour conclure,
il suffit de montrer que ce cône est engendré par les vecteurs $\vec
w_i$. Or, il est engendré par les intersections $(d-2)$ à $(d-2)$ des
$(d-1)$ hyperplans (de $H$) frontière des demi-espaces précédents. Un
calcul direct montre enfin que ces intersections sont exactement les
vecteurs $\vec w_i$, ce qui conclut.
\end{proof}

\subsubsection{Le cas des variétés $\calX_{\leq \mu}$}

On en vient maintenant aux variétés $\calX_{\leq \mu}$, c'est-à-dire
à la démonstration du théorème \ref{theo:dimdivelem}. Pour ce faire,
on suit à nouveau la même méthode : on donne une description de
l'ensemble convexe $B_{Q,g,\ell,C^\star}$, à partir de laquelle on
déduit une formule pour la fonction $b_{Q,g,\ell,C^\star}$ d'où
résultera le théorème.

\paragraph{Le calcul de $A_{Q,g,\ell} \cap C$}

Le lemme \ref{lem:AQmingl} donne une description explicite de l'ensemble
$A_{Q_\min,g,\ell}$. Dans ce paragraphe, on se propose de démontrer 
qu'il s'en déduit une autre description simple de l'intersection 
$A_{Q,g,\ell} \cap C$, au moins si $b$ est suffisamment grand. Plus 
précisément, on se propose de démontrer que, si $b \geq b_0$, on a :
\begin{equation}
\label{eq:AQglC}
A_{Q,g,\ell} \cap C = A_{Q_\min,g,\ell} \cap C =
\Big( \frac{2 \vec \rho}{b+1} + \Reg^\star \Big) \cap C
\end{equation}
D'après la proposition \ref{prop:AQgl} (en fait, seule l'inclusion
facile sert), il suffit d'établir l'égalité $(A_{Q_\max,g,\ell} +
C^\star) \cap C = A_{Q_\min,g,\ell} \cap C$. On commence par un
lemme concernant les points extrémaux de $A_{Q_\max,g,\ell}$.

\begin{lemme}
\label{lem:extrAQmaxglC}
On suppose $b \geq \max(b_0, d+1)$. Alors le seul point extrémal 
de $A_{Q_\max,g,\ell}$ qui appartient à $C$ est $\frac {2 \vec
\rho} {b+1}$.
\end{lemme}

\begin{proof}
Comme on a supposé $b \geq b_0$, la proposition \ref{prop:extrAQmaxgl} 
s'applique et les points extrémaux de $A_{Q_\max,g,\ell}$ sont de la 
forme $\vec \rho_\weyl$. On est ainsi ramené à montrer que le vecteur 
$\vec \rho_\weyl$ n'appartient pas à $C$ dès que $\weyl$ n'est pas la 
permutation $i \mapsto d+1-i$. Or, par définition de $C$ et d'après la 
formule donnant les coordonnées de $\vec \rho_\weyl$, le fait que $\vec 
\rho_\weyl$ soit élément de $C$ implique que pour tout $i \in \{1, 
\ldots, d-1\}$, l'égalité suivante est satisfaite :
$$\sum_{n=1}^\infty \frac{
\weyl^n(i+1) - \weyl^n(i) + 1}{b^n} \geq 0.$$
Du fait que $b \geq d+1$, ceci implique que le premier
terme de la somme (\emph{i.e.} pour $n = 1)$ est lui-même positif ou
nul. Ainsi $\weyl(i+1) \geq \weyl(i) - 1$ pour tout indice $i \in \{1, 
\ldots, d-1\}$ et il est alors clair que $\weyl$ ne peut être que la 
permutation $i \mapsto d+1-i$.
\end{proof}

\begin{lemme}
\label{lem:AQvois}
Soit $y = (y_1, \ldots, y_d)$ un élément de $A_{Q_\min,g,\ell}$ tel que
$y_i \leq y_{i+1} + 1$ pour tout $i \in \{1, \ldots, d-1\}$. Alors $y \in
A_{Q_\max,g,\ell}$.
\end{lemme}

\begin{proof}
D'après la descrition de $A_{Q_\max,g,\ell}$ donnée par le système
\eqref{eq:AQmaxgl}, il s'agit de montrer que $f_T(y) \geq 0$ pour toute 
partie $T$ de $\{1, \ldots, d\}$. Soit $T$ une telle partie. Si 
on note $t_1 < \cdots < t_s$ ses éléments (avec $s = \card\,T$), 
l'inégalité $f_T(y) \geq 0$ devient
\begin{equation}
\label{eq:inAQmax}
y_1 + \cdots + y_s - b \cdot \sum_{i=1}^s y_{d+1-t_i} \geq s(d-s) + 
b \cdot \sum_{i=1}^s (t_i - i).
\end{equation}
Par ailleurs, comme $y$ est pris dans $A_{Q_\min,g,\ell}$, on a
$$y_1 + \cdots + y_s - b \cdot \sum_{i=1}^s y_{d+1-t_i} \geq 
s(d-s) + b \cdot \sum_{i=1}^s (y_{d+1-i} - y_{d+1-t_i}).$$
Or, on a $t_i \geq i$ pour tout $i$ d'où, d'après l'hypothèse, on
tire $y_{d+1-i} - y_{d+1-t_i} \leq t_i - i$. L'inégalité \eqref{eq:inAQmax} 
s'ensuit.
\end{proof}

On peut à présent démontrer l'égalité \eqref{eq:AQglC} lorsque $b \geq 
\max(b_0, d+1)$. On rappelle qu'il suffit de montrer que le polytope 
$P_1 = A_{Q_\min,g,\ell} \cap C$ est inclus dans le polytope
$P_2 = (A_{Q_\max, g,\ell} + C^\star) \cap C$ et, pour cela, on peut
se contenter de prouver que tous les sommets de $P_1$ (en incluant ceux 
à l'infini) sont dans $P_2$. On vérifie tout de suite que $\Reg \subset
C^\star$ ; ainsi $C \subset \Reg^\star$ et la description donnée par le
lemme \ref{lem:AQmingl} implique $A_{Q_\max,g,\ell} + C =
A_{Q_\max,g,\ell}$ et finalement $P_1 + C = P_1$. Comme en plus $P_1
\subset C$, on en déduit que les sommets à l'infini de $P_1$ sont ceux
de $C$. Mais du fait que $C \subset C^\star$ (vérification facile), il
suit que $P_2 + C = P_2$. Ainsi les sommets à l'infini de $C$ sont bien
dans $P_2$.

Il reste donc à traiter le cas des sommets de $P_1$ à distance fini, 
c'est-à-dire des points extrémaux de $P_1$. Soient $F_1, \ldots, 
F_{d-1}$ les facettes de $A_{Q_\min,g,\ell}$ ; ce sont des cônes 
simpliciaux de dimension $d-2$ issus de $\frac{2 \vec \rho} {b+1}$. Le 
lemme \ref{lem:AQvois} assure que les convexes $A_{Q_\max,g,\ell} + 
C^\star$ et $A_{Q_\min,g,\ell}$ sont égaux sur un voisinage de $\frac{2 
\vec \rho} {b+1}$. On en déduit que $A_{Q_\max,g,\ell} + C^\star$ a 
exactement $d-1$ facettes issues de $\frac{2 \vec \rho} {b+1}$ et que 
l'on peut numéroter celles-ci $F'_1, \ldots, F'_{d-1}$ de sorte que 
$F'_i \subset F_i$ pour tout $i$. Pour chaque $i$, les sommets de $F'_i$ 
à distance finie sont des points extrémaux de $A_{Q_\max,g,\ell} + 
C^\star$ et donc \emph{a fortiori} des points extrémaux de 
$A_{Q_\max,g,\ell}$. En particulier, d'après la proposition 
\ref{prop:extrAQmaxgl} et le lemme \ref{lem:extrAQmaxglC}, ceux-ci ne 
sont pas dans $C$ à l'exception de $\frac {2 \vec \rho} {b+1}$. Les 
sommets à l'infini, quant à eux, correspondent aux directions de 
$\Reg^\star$, c'est-à-dire aux vecteurs $\vec w_i$ introduits dans la 
démonstration du lemme \ref{lem:AQmingl}. Comme aucun des $\vec w_i$, ni 
de leurs opposés, n'est dans $C$, on en déduit finalement que $F_i \cap 
C = F'_i \cap C$.

Il est maintenant facile de conclure : tout point extrémal de $P_1 =
A_{Q_\min,g,\ell} \cap C$ appartient à l'une des faces $F_i$ et bien sûr
également à $C$. Il appartient donc à la face $F'_i$ correspondante, et
par suite à $A_{Q_\max,g,\ell} + C^\star$ puis à $P_2$.

\paragraph{Obtention de la majoration}

On sait que la dimension de la variété $\calX_{\leq \mu}$ est majorée
par la quantité $b_{Q,g,\ell,C^\star}(\mu)$ qui d'après la proposition
\ref{prop:dualconv2} vaut 
$\inf_{\alpha \in B_{Q,g,\ell,C^\star}} \left< \alpha | \mu \right>_E$
où, on rappelle que $B_{Q,g,\ell,C^\star} = A_{Q,g,\ell} \cap C$. Pour
démontrer la première majoration :
$$\dim_k \calX_{\leq \mu} \leq \frac{\left< 2 \vect \rho | \mu \right>_d}
{b+1}$$
il suffit donc de prouver que le vecteur $\frac{2 \vect \rho}{b+1}$ est
dans $B_{Q,g,\ell,C^\star}$. Il est déjà évident qu'il appartient à $C$.
De plus, le lemme \ref{lem:AQmingl} montre qu'il appartient à
$A_{Q_\min,g,\ell}$, d'où il suit, grâce au lemme \ref{lem:AQvois},
qu'il est aussi dans $A_{Q_\max,g,\ell}$ et donc \emph{a fortiori} 
dans $A_{Q,g,\ell}$.
On en vient finalement à la deuxième majoration du théorème : on suppose 
$b \geq \max(b_0, d+1)$ et on veut montrer
\begin{equation}
\label{eq:majrenf}
\dim_k \calX_{\leq \mu} \leq \sup_{\substack{\mu' \leq \mu \\ \mu'
\in \Reg}} \frac{\left< 2 \vect \rho | \mu' \right>_d} {b+1}.
\end{equation}
Si l'on définit la forme linéaire $L : \R^d \to \R$, $\mu \mapsto \frac
1 {b+1}{\cdot}\left< 2 \vect \rho | \mu \right>_d$, on remarque que le
majorant dans \eqref{eq:majrenf} n'est autre que
$b_{\Reg,\id,L,C^\star}(\mu)$. Pour terminer la preuve, il suffit donc
de montrer que $b_{\Reg,\id,L,C^\star} = b_{Q,g,\ell,C^\star}$. Or, en
déroulant les définitions, on trouve
$A_{\Reg,\id,L} = \frac{2 \vec \rho}{b+1} + \Reg^\star$
et donc $A_{\Reg,\id,L} = A_{Q_\min,g,\ell}$ d'après le lemme
\ref{lem:AQmingl}. L'égalité \eqref{eq:AQglC} implique alors 
$A_{\Reg,\id,L} \cap C = A_{Q,g,\ell} \cap C$ à partir de quoi la
proposition \ref{prop:dualconv2} permet de conclure.

\paragraph{Obtention de la minoration}

Il ne reste plus qu'à démontrer la minoration.
Étant donné que si $\mu' \leq \mu$, on a tautologiquement $\calX_{\leq
\mu'} \subset \calX_{\leq \mu}$ et donc $\dim_k \calX_{\leq \mu'} \leq
\dim_k \calX_{\leq \mu}$, il suffit de montrer que si $\mu = (\mu_1,
\ldots, \mu_d)$ est fortement intégralement $b$-régulier, alors :
$$\dim_k \calX_{\leq \mu} \geq \frac{\left< 2 \vect \rho | \mu \right>_d}
{b+1} - (d-1)^2 - \frac {(d-2)^2} 4.$$
Pour cela, il suffit de trouver un élément $q \in Q \cap R$ tel que
$g(q) \in \mu - C^\star$ et $\ell(q) \geq \frac{\left< 2 \vect \rho | \mu
\right>_d} {b+1} - (d-1)^2 - \frac {(d-2)^2} 4$. On pose, pour tout $i$,
$q'_i = \frac {\mu_i + b \mu_{d+1-i}}{b^2-1}$. 
La condition de divisibilité qui apparaît dans la définition d'un
$d$-uplet fortement intégralement $b$-régulier assure que si $q'_1 + \cdots
+ q'_d$ est un entier alors que les conditions d'inégalité se traduisent
par $q'_1 \leq q'_2 \leq \cdots \leq q'_{d-1} \leq q'_d - d$.
Pour $i < d$, on note $q_i$ la partie entière supérieure de $q'_i$
(\emph{i.e.} le plus petit entier plus grand ou égal à $q'_i$) et
on définit $q_d$ de sorte que $q_1 + \cdots + q_d = q'_1 + \cdots +
q'_d$. Tous les $q_i$ sont alors des entiers, et on vérifie directement
que la suite qu'ils forment est croissante. Ainsi le vecteur $q$ de
coordonnées $q_{i,j} = q_{d-j+i}$ (pour $(i,j) \in I$) appartient à $Q$.
En outre, un calcul immédiat montre que pour tout $s$, on a :
\begin{eqnarray*}
\left< \vec \mu_1 + \cdots + \vec \mu_s | q \right>_E 
& = & b(q_d + \cdots + q_{d+1-s}) - (q_1 + \cdots + q_s) \\
& = & - b (q_1 + \cdots + q_{d-s}) - (q_1 + \cdots + q_s) \\
& \leq & - b (q'_1 + \cdots + q'_{d-s}) - (q'_1 + \cdots + q'_s) \\
& = & b(q'_d + \cdots + q'_{d+1-s}) - (q'_1 + \cdots + q'_s)
\, = \,\mu_1 + \cdots + \mu_s
\end{eqnarray*}
et l'égalité est atteinte lorsque $s = d$. Autrement dit, on a bien 
$g(q) \in \mu - C^\star$. Il reste à minorer $\ell(q)$ :
$$\ell(q) = \left< \vect \delta | q \right>_d = \sum_{i=1}^d (2i-d-1)
q_i = \frac{\left< 2 \vect \rho | \mu \right>_d}{b+1} + 
\sum_{i=1}^d (2i-d-1) (q_i - q'_i).$$
Or $q_d - q'_d$ est compris entre $1-d$ et $0$, tandis que les $q_i
- q'_i$ pour $i < d$ sont compris entre $0$ et $1$. On en déduit
que :
$$\ell(q) \geq \frac{\left< 2 \vect \rho | \mu \right>_d}{b+1} -
(d-1)^2 - \sum_{i=2}^d \max(2i-d-1,0) \geq \frac{\left< 2 \vect \rho | 
\mu \right>_d}{b+1} - (d-1)^2 - \frac{(d-2)^2} 4.$$

\subsubsection{Le cas des variétés $\calX_{\leq e}$}
\label{subsec:demcaruso}

\paragraph{Preuve de la majoration}

Soit $e$ un nombre entier. La variété $\calX_{\leq e}$ s'écrit 
manifestement comme l'union disjointe des variétés $\calX_{\leq \mu}$ 
pour $\mu = (\mu_1, \cdots, \mu_d)$ vérifiant $e \geq \mu_1 \geq \cdots 
\geq \mu_d \geq 0$. Il suffit donc de démontrer que, si $\mu$ est tel 
qu'on vient de le décrire, la dimension de $\calX_{\leq \mu}$ est 
majorée par $[\frac{d^2} 4] \cdot \frac e{b+1}$. Or, par le théorème 
\ref{theo:dimdivelem}, on sait que cette dimension est majorée par 
$\frac 1{b+1} \cdot \left< 2 \vect \rho | \mu \right>_d$ et on a 
$$\frac{\left< 2 \vect \rho | \mu \right>_d}{b+1} = \sum_{i=1}^d 
\frac{d+1-i}{b+1} \cdot \mu_i \leq e \cdot \sum_{i=1}^{[d/2]} 
\frac{d+1-i}{b+1} = \cro{\frac{d^2} 4} \cdot \frac e {b+1}.$$ La 
conclusion en résulte.

\paragraph{Preuve de la minoration}

Pour obtenir la minoration, on doit estimer la quantité
$b'_{Q,R,f,\ell,C}(e,0)$\footnote{Remarquez que la proposition
\ref{prop:bprime} et la discussion qui suit sa démontration nous dit
qu'elle diffère de $b_{Q,f,\ell,C}(e,0)$ d'une quantité bornée.
Cependant, pour arriver à la minoration énoncée dans le théorème
\ref{theo:dimcaruso}, on a besoin d'être plus précis que cela.}.
Or, il est évident par définition que si $q \in Q \cap R$ vérifie $f(q)
\in (e,0) + C$, cette quantité est minorée par $\ell(q)$. Il suffit donc,
pour terminer la démonstration du théorème \ref{theo:dimcaruso}, de
construire un élément $q = (q_{i,j}) \in Q \cap R$ tel que :
$$f(q) \in (e,0) + C \quad \text{et} \quad
\ell(q) = \cro{\frac{d^2} 4} \cdot \cro{\frac{e-b+2}{b+1}}.$$
Soient $n = [\frac{e-b+2} {b+1}]$ et $m$ le reste de la division 
euclidienne de $(-n)$ par $b-1$. On considère l'élément $q$ défini comme 
suit :
$$\begin{array}{rcll}
q_{i,j} & = & \frac{m + b n}{b-1} & \text{si } 0 \leq j - i < \frac d 2\\
q_{i,j} & = & \frac{m + n}{b-1} & \text{si } j - i \geq \frac d 2.
\end{array}$$
Il est clair que $\frac{m + b n}{b+1} \geq \frac{m+n}{b+1}$. À partir
de là et du fait que la valeur de $q_{i,j}$ ne dépend que de la
différence $j-i$, il résulte que $q$ appartient à $Q_\min$, et donc
\emph{a fortiori} à $Q$. Comme, par ailleurs, il suit de la définition
de $m$ que tous les $q_{i,j}$ sont entiers, on a bien $q \in Q \cap R$.
D'autre part, un calcul direct donne $\left< \vec \mu_1 | q \right>_E =
m + n(b+1) \leq b-2 + e - b + 2 = e$ et $\left< \vec \mu_d | q \right>_E
= m \geq 0$, d'où il suit $f(q) \in (e,0) + C$. Finalement, à nouveau un
calcul facile conduit à la valeur souhaitée pour $\ell(q) = \left< \vect
\delta | q \right>_E$.

\subsection{Points extrémaux de $A_{Q,g,\ell} \cap C$ : quelques
exemples}
%\label{subsec:AQglC}

J'aimerais revenir un instant sur la démonstration de la majoration du 
théorème \ref{theo:dimdivelem}. Elle procédait ainsi. Dans un premier 
temps, on a utilisé le théorème \ref{theo:dim} afin de majorer la 
dimension de $\calX_{\leq \mu}$ par la quantité $b_{Q,g,\ell,C^\star}(\mu)$,
et ensuite on a prouvé l'égalité $b_{Q,g,\ell,C^\star}(\mu) = 
b_{\Reg,\id,L,C^\star}(\mu)$ en montrant que chacun de ces deux nombres
s'égalisait avec :
$$M(\mu) = \inf_{\alpha \in K} \left< \alpha | \mu \right>_d
\qquad \text{où} \quad
K = \Big( \frac{2 \vec \rho}{b+1} + \Reg^\star\Big) \cap C.$$
Ainsi, l'expression $M(\mu)$ est une nouvelle façon d'exprimer le
majorant obtenu. En outre, $K$ est un polytope qui n'a qu'un nombre
fini de points extrémaux, et si on appelle $\Lambda$ leur ensemble,
on a :
$$M(\mu) = \inf_{\lambda \in \Lambda} \left< \lambda | \mu \right>_d$$
pour tout $\mu \in \Reg + C^\star = C^\star$, c'est-à-dire dès que
$\mu_1 \geq \cdots \geq \mu_d$ si, comme d'habitude, on appelle $\mu_i$
($1 \leq i \leq d$) les coordonnées de $\mu$.

Déterminer les points extrémaux de $K$ apparaît donc comme une question
naturelle et importante. Malheureusement, bien que $K$ soit défini de
façon plutôt simple (c'est l'intersection de $2d-2$ demi-espaces affines
dans un espace de dimension $d-1$), la combinatoire de ses points
extrémaux paraît compliquée. Par exemple, le nombre de ces points semble
exploser très rapidement lorsque $d$ augmente. Ces deux dernières
impressions se sont forgées à la suite de calculs numériques effectués
grâce au logiciel {\tt polymake} \cite{polymake}. La deuxième ligne du
tableau \ref{tab:nbextr} donne, par exemple, le nombre de points
extrémaux de $K$ pour diverses valeurs de $d$ avec $b = 10\,000$ (la
valeur de $b$ importe peu pour la complexité de l'ensemble des points
extrémaux de $K$ comme l'explique le théorème \ref{theo:depb} ci-après). 
La croissance est apparemment exponentielle ; on notera qu'elle ne peut, 
en tout cas, pas être pire car, un point extrémal étant situé à 
l'intersection de $d$ hyperplans parmi les $2d$ définissant $K$, leur 
nombre est trivialement majoré par $\binom {2d-2}{d-1} \leq 4^{d-1}$.

\begin{table}[h]
\begin{center}
\begin{tabular}{|>{\raggedright}m{3.3cm}||c|c|c|c|c|c|c|c|c|}
\hline
Dimension $d$ & 2 & 3 & 4 & 5 & 6 & 7 & 8 & 9 & 10 \\
\hline
Nombre de points extrémaux de $K$
& 1 & 3 & 6 & 15 & 33 & 70 & 136 & 347 & 667 \\
\hline
Nombre de points extrémaux de $K{+}C^\star$
& 1 & 3 & 5 & 9 & 17 & 31 & 47 & 103 & 163 \\
\hline
\end{tabular}
\end{center}
\caption{Nombre de points extrémaux de $K$ et $K+C^\star$
pour $b = 10\,000$}
\label{tab:nbextr}
\end{table}

\medskip

Cependant, les points extrémaux que l'on a trouvé précédemment ne sont
pas vraiment tous pertinents. En effet, comme l'on ne s'intéresse aux 
variétés $\calX_{\leq \mu}$ que lorsque les coordonnées de $\mu$ sont 
triés par ordre décroissant\footnote{En fait, si $\mu$ est quelconque, 
on montre qu'il existe $\mu' = (\mu'_1 \geq \cdots \mu'_d)$, facilement 
explicitable à partir de $\mu$, tel que $\calX_{\leq \mu} = \calX_{\leq 
\mu'}$.}, c'est-à-dire lorsque $\mu \in C$, on peut, au lieu de 
considérer la fonction $b_{Q,f,\ell,C^\star}$, travailler plutôt avec 
la fonction $b_{Q,f,\ell,C^\star,C}$ définie par :
$$\begin{array}{rcll}
b_{Q,g,\ell,C^\star,C}(y) & = & b_{Q,g,\ell,C^\star} (y)
& \text{si } y \in C \\
& = & -\infty & \text{sinon}.
\end{array}$$
Par la proposition \ref{prop:dualconv2}, on a
$$M'(\mu) = \inf_{\alpha \in K'} \left< \alpha | \mu \right>_d
\qquad \text{où} \quad K' = K + C^\star$$
Ainsi, plutôt que décrire les points extrémaux de $K$, on a plutôt envie
de comprendre ceux de $K+C^\star$ qui forment un sous-ensemble (en général
strict) des points extrémaux de $K$. Si l'on reprend les exemples
précédents ($b = 10\:000$, $d$ petit), on constate sur le tableau 
\ref{tab:nbextr}, que l'on élimine en effet ainsi un bon paquet
de points extrémaux, au moins pour les petites valeurs de $d$.

Il a été dit précédemment que la dépendance en $b$ est moins délicate à 
comprendre. En effet, on a le théorème suivant.

\begin{theo}
\label{theo:depb}
On pose, comme précédemment, 
$$K(b) = \Big( \frac{2 \vec \rho}{b+1} + \Reg^\star\Big) \cap C
\quad \text{et} \quad K'(b) = K(b) + C^\star$$
et on note $\Lambda(b)$ et $\Lambda'(b)$ l'ensemble des points extrémaux
de $K(b)$ et $K'(b)$ respectivement. Alors, pour $b$ suffisamment grand,
le cardinal de $\Lambda(b)$ (resp. $\Lambda'(b)$) est constant et les
coordonnées de points de cet ensemble s'expriment comme des fractions
rationnelles en $b$.
\end{theo}

\begin{proof}
Le théorème résulte du fait que les méthodes de calcul de points
extrémaux de polytopes s'appliquent dans n'importe quel corps 
ordonné. Ici donc, on peut voir $K(b)$ et $K'(b)$ comme des
polytopes définis sur le corps réel $\R(b)$ muni de l'ordre qui fait
de $b$ un élément infiniment grand. Ces polytopes ont alors bien sûr
un nombre de sommets qui ne dépend pas de $b$ et les coordonnées de ces 
sommets sont des éléments de $\R(b)$, c'est-à-dire des fractions 
rationnelles en $b$.
Il reste à justifier que ces expressions redonnent bien les points
extrémaux de $K(b)$ et $K'(b)$ lorsque l'on spécialise $b$ en une
valeur suffisamment grande. Mais c'est évident car le calcul des 
sommets faits dans $\R(b)$ est valable dès que $b$ satisfait un
certain nombre fini d'inégalités, et donc en particulier dès que 
$b$ est suffisamment grand.
\end{proof}

\noindent
Le tableau \ref{tab:coordextr} montre les fractions
rationnelles que l'on obtient pour les petites valeurs de $d$.

\begin{table}
\def\espace{\vphantom{\frac {\displaystyle \sum} b}}
\def\espaced{\vphantom{\frac {\displaystyle \sum} {\displaystyle \sum}}}
$$\begin{array}{|c|>{\centering}m{2cm}|>{\displaystyle}c|}
\hline
\text{Dimension } d & Domaine de validité & 
\text{Points extrémaux de } K+C^\star \\
\hline
\hline
2 & $b \geq 2$ &  \frac{2 \vec \rho}{b+1} = \frac{(1,-1)}{b+1} \espaced \\
\hline
& & \frac{2 \vec \rho}{b+1} = \frac {(2,0,-2)} {b+1} \espace \\ 
3 & $b \geq 2$ & \frac {(-2,-2,4)} {b+2} \espace \\
& & \frac {(4,-2,-2)} {b+2} \espaced \\
\hline
& & \frac{2 \vec \rho}{b+1} = \frac {(3,1,-1,-3)} {b+1} \espace \\
& & \frac {(3,3,3,-9)} {b+3} \espace \\
4 & $b \geq 3$ & \frac {(9,-3,-3,-3)} {b+3} \espace \\
& & \frac {(1,-1,-1,1)} {b-1} + \frac {(4,0,0,-4)} {b+1} \espace \\
& & \frac {(-1,1,1,-1)} {b-1} + \frac {(4,0,0,-4)} {b+1} \espaced \\
\hline
& & \frac{2 \vec \rho}{b+1} = \frac {(4,2,0,-2,-4)} {b+1} \espace \\
& & \frac {(4,4,4,4,-16)} {b+4} \espace \\
& & \frac {(16,-4,-4,-4,-4)} {b+4} \espace \\ 
& & \frac {(3,-2,-2,-2,3)} {b-1} + \frac {(7,0,0,0,-7)} {b+1} \espace \\
5 & $b \geq 6$ & \frac {(-3,2,2,2,-3)} {b-1} + \frac {(7,0,0,0,-7)} {b+1} \espace\\
& & \frac {(-2,2,0,0,0)} b + \frac {(6,0,0,0,-6)} {b+1} \espace \\
& & \frac {(0,0,0,-2,2)} b + \frac {(6,0,0,0,-6)} {b+1} \espace \\
& & \frac {(4, 0,0,0,-4)} {b+1} + \frac {(0,2,2,-4,0)} {b+2} \espace \\
& & \frac {(4, 0,0,0,-4)} {b+1} + \frac {(0,4,-2,-2,0)} {b+2} \espaced \\
\hline
\end{array}$$
\caption{Coordonnées des points extrémaux de $K+C^\star$}
\label{tab:coordextr}
\end{table}

\section{Perspectives et conjectures}
\label{sec:conjectures}

\subsection{Peut-on espérer une formule exacte pour la dimension ?}
\label{subsec:dimexacte}

Si $h \geq 0$, le théorème \ref{theo:dim} donne une
formule exacte pour la dimension des variétés $\calX_\varphi$. On peut
donc raisonnablement penser que, dans ce cas, il est possible d'en
déduire une formule exacte
pour la dimension de $\calX_{\leq e}$, $\calX_\mu$ et $\calX_{\leq
\mu}$. Et de fait, on dispose d'une telle formule car on peut toujours 
écrire (comme nous l'avons déjà fait plusieurs fois) :
\begin{equation}
\label{eq:dimexacte}
\dim_k \calX_{\leq e} = b'_{Q,R,f,\ell,\R^+ \times \R^-}(0,e) = 
\sup_{\substack{q \in Q \cap R \\ f(q) \in (e,0) - (\R^+ \times \R^-)}} 
\ell(q)
\end{equation}
ainsi que des expressions analogues pour les autres variétés. Les
notations dans la formule \eqref{eq:dimexacte} sont celles qui ont
été utilisées dans les sections précédentes ; on renvoie le lecteur 
aux débuts des \S\S \ref{subsec:conephi} et \ref{subsec:vecmui} pour
un récapitulatif rapide des définitions. Lorsque $h \geq 0$, calculer la 
dimension de $\calX_{\leq e}$ revient ainsi à calculer le nombre 
$b'_{Q,R,f,\ell,\R^+ \times \R^-}(0,e)$. Comme ce dernier s'exprime
comme le maximum d'une forme linéaire sur un ensemble fini on peut, en
un certain sens, considérer que le problème est résolu ; en tout cas, il
est aisé à partir de là d'écrire un algorithme qui répond à la question
pour des entiers $d$, $b$ et un $d$-uplet $\mu$ donnés. Toutefois, cela
n'est pas entièrement satisfaisant car l'on aimerait comprendre par
exemple le comportement précis de la dimension de $\calX_{\leq e}$ 
lorsque les paramètres $d$, $b$ et $e$ varient. Pour ce
type de questions, l'approche algorithmique naïve, que l'on vient de 
présenter, s'avère insuffisante.  À l'opposé de cette approche
algorithmique, il y a un théorème général de logique qui prédit la
dépendance de $b'_{Q,R,f,\ell,\R^+ \times \R^-}(0,e)$ en fonction de
$e$. Voici ce qu'il implique dans notre cas.

\begin{theo}
\label{theo:presburger}
On suppose que $h \neq 0$. Alors, il existe un entier $N$ et une
fonction $f : \Z/N\Z \to \Q$ tel que, pour $e$ suffisamment grand, on
ait :
$$\dim_k \calX_{\leq e} = \cro{\frac{d^2} 4} \cdot \frac e{b+1} + 
f(e \mod N).$$
\end{theo}

\begin{proof}
D'après la formule \eqref{eq:dimexacte}, la différence $\delta(e) =
\dim_k \calX_{\leq e} - [\frac{d^2} 4] \cdot \frac e{b+1}$ est définie
par une formule de l'arithmétique de Presburger ayant la variable libre
$e$. Par le théorème d'élimination des quantificateurs dans
l'arithmétique de Presburger, cette formule est équivalente à une
formule sans quantificateurs. Par ailleurs, le théorème
\ref{theo:dimcaruso} montre que la fonction $\delta : e \mapsto
\delta(e)$ est bornée sur $\N$. On déduit facilement à partir de là que,
pour $e$ suffisamment grand, elle ne dépend que de la réduction de $e$
modulo un certain entier $N$. C'est exactement ce qu'il fallait
démontrer.
\end{proof}

On peut reformuler le théorème précédent en disant que la série
génératrice
$$\sum_{e = 0}^\infty (\dim_k \calX_{\leq e}) \cdot X^e.$$
est, en fait, une fraction rationnelle. Ce théorème est probablemement 
intéressant sur le plan théorique mais, d'une point de vue pratique, le 
théorème \ref{theo:presburger} est absolument inutile car il ne dit rien 
ni sur l'entier $N$, ni sur la fonction $f$, ni sur le moment à partir 
duquel la formule pour la dimension est correcte. On souligne en outre, 
au cas où l'énoncé n'était pas clair sur ce point, que ces données 
dépendent \emph{a priori} de $d$ et de $b$. Un élément positif malgré 
tout est qu'il existe des algorithmes pour les calculer. Par contre, 
malheureusement, au delà de la dimension $3$ (pour laquelle on peut 
encore faire les calculs à la main), il n'est pas envisageable 
d'utiliser de tels outils, ceux-ci étant (à l'heure actuelle) trop peu 
efficaces.

Le théorème \ref{theo:presburger} admet, bien sûr, des analogues pour 
les variétés $\calX_\mu$ et $\calX_{\leq \mu}$ qui sont peu ou prou 
équivalents à la rationnalité des séries
$$\sum_{\mu \in \N^d} (\dim_k \calX_\mu) 
\cdot X_1^{\mu_1} X_2^{\mu_2} \cdots X_d^{\mu_d}
\quad \text{et} \quad 
\sum_{\mu \in \N^d} (\dim_k \calX_{\leq \mu}) 
\cdot X_1^{\mu_1} X_2^{\mu_2} \cdots X_d^{\mu_d}$$
où les entiers $\mu_i$ désignent les coordonnées de $\mu$.

\subsubsection{Calcul en petites dimensions}

En guise d'illustration du résultat du théorème \ref{theo:presburger}
(ou plutôt de l'un de ses analogues qui viennent d'être évoqués), on
se propose de calculer les dimensions exactes des variétés $\calX_\mu$ 
lorsque $d = 2$ et également lorsque $d = 3$ dans certains cas. 
Pour cela, plutôt que d'utiliser les $q_{i,j}$ pour paramétrer les
$d$-uplet $\varphi = (\varphi_1, \ldots, \varphi_d)$ comme cela a été 
fait jusqu'à présent, on va plutôt travailler ici avec les $\mu_{i,j}$,
ce qui sera plus commode. Les inégalités qui définissent l'ensemble $Q$ 
s'écrivent
$$\begin{array}{ccl}
\mu_{i-1,j-1} \leq \mu_{i,j} \leq \mu_{i-1,j} &
\text{ et } &
\displaystyle b \mu_{i,i} + b \sum_{s=i+1}^j (\mu_{i,s} 
- \mu_{i+1,s}) + \sum_{s=j+1}^d (\mu_{i,s} - \mu_{i+1,s}) \\
& & \hspace{0.4cm} \leq
\displaystyle b \mu_{i-1,i-1} + b \sum_{s=i}^{j-1} (\mu_{i-1,s} 
- \mu_{i,s}) + \sum_{s=j}^d (\mu_{i-1,s} - \mu_{i,s})
\end{array}$$
pour tout couple d'entiers $(i,j)$ avec $2 \leq i \leq j \leq d$,
alors que les conditions d'intégrité, qui définissent le réseau $R$,
sont données par la proposition \ref{prop:integrite} :
$$\begin{array}{rl}
\forall (i,j) \in I, & \mu_{i,j} \in \Z \\
\forall i \in \{1, \ldots, d\}, &
\mu_{i,i} + \mu_{i,i+1} + \cdots + \mu_{i,d} \equiv 0 \pmod{b-1}.
\end{array}$$
On rappelle que si $\varphi$ est l'élément de $\Phi$ correspondant à
une donnée $(\mu_{i,j})$ satisfaisant les conditions précédentes, alors
pour tout réseau $L \subset M$, les exposants des diviseurs élémentaires 
du $k[[u]]$-module engendré par $\sigma(L)$ par rapport à $L$
sont les $\mu_i = \mu_{1,i}$ (voir propositions
\ref{prop:varphi} et \ref{prop:bijphipsi}) et que :
$$\dim(\varphi) = \sum_{j=1}^d (d+1-j) \cdot \mu_{1,j} - \sum_{(i,j) \in
I} \mu_{i,j} = \sum_{j=1}^d (d+1-j) \cdot \mu_j -
\sum_{(i,j) \in I} \mu_{i,j}$$
(voir lemme \ref{lem:dim}). Dans la suite, on supposera toujours
que $h \geq 0$ de sorte que la quantité précédente s'égalise avec
la dimension de la variété $\calX_\varphi$. On note $\lceil x \rceil$,
la partie entière supérieure du nombre réel $x$, c'est-à-dire le plus
petit entier supérieur ou égal à $x$. On pose aussi $\df(x) = \lceil 
x \rceil - x$ ; c'est à l'évidence un nombre compris entre $0$ et $1$ 
qui ne dépend que de la congruence de $x$ modulo $\Z$.

\paragraph{En dimension $\mathbf 2$}

D'après ce que l'on vient de rappeler, étant donnés des nombres entiers
$\mu_1 \geq \mu_2$ tels que $b-1$ divise $\mu_1 + \mu_2$, calculer la
dimension de la variété $\calX_{(\mu_1, \mu_2)}$ revient à maximiser
le nombre $\mu_1 - \mu_{1,2}$ sous les contraintes
$$\left\{ \begin{array}{l}
\mu_1 \leq \mu_{1,2} \leq \mu_2 \\
(b+1) \mu_{1,2} \geq b \mu_1 + \mu_2 \\
b-1 \text{ divise } \mu_{1,2}
\end{array} \right.$$
En écrivant $\mu_{1,2} = (b-1) x$, on voit tout de suite que le maximum
cherché est atteint pour $x = \lceil \frac{b\mu_1 + \mu_2}{b^2-1}
\rceil$. Ainsi, obtient-on :
\begin{equation}
\label{eq:dimmu12}
\dim_k \calX_{(\mu_1, \mu_2)} = \mu_1 - (b-1) \cdot 
\left\lceil \frac{b \mu_1 + \mu_2}{b^2-1} \right\rceil
= \frac{\mu_1 - \mu_2}{b+1} - (b-1) \cdot \df\left( \frac{b \mu_1 + 
\mu_2} {b^2-1} \right)
\end{equation}
On constate immédiatement sur la dernière écriture que la dimension de
$\calX_{(\mu_1, \mu_2)}$ s'exprime comme la somme de $\frac{\mu_1 -
\mu_2}{b+1}$ (qui correspond au terme attendu) et d'un terme correctif
qui ne dépend que des congruences de $\mu_1$ et $\mu_2$ modulo $b^2-1$.
De surcroît, ce terme correctif varie dans l'intervalle $]1-b, 0]$ ;
la dimension de $\calX_{(\mu_1, \mu_2)}$ se caractérise donc encore
comme le plus grand entier $\leq \frac{\mu_1 - \mu_2}{b+1}$ qui est
congru à $\mu_1$ modulo $b-1$ (en accord avec la congruence du
théorème \ref{theo:dimdivelem2}).

La dimension de la variété $\calX_{\leq e}$, quant à elle, s'obtient en
prenant le maximum de $\dim_k \calX_{(\mu_1, \mu_2)}$ sur tous les
couples d'entiers $(\mu_1, \mu_2)$ vérifiant $0 \leq \mu_2 \leq \mu_1
\leq e$ et $\mu_1 + \mu_2 \equiv 0 \pmod {b-1}$. Le calcul devient alors
pénible et conduit à distinguer de nombreux cas ; nous ne le faisons 
pas. Il est quand même possible à peu de frais d'obtenir le résultat
suivant.

\begin{prop}
\label{prop:Vedeux}
On suppose $d = 2$ et $h \geq 0$.
Alors pour tout entier $e \geq 0$, on a :
$$\dim_k \calX_{\leq e + (b^2-1)} = \dim_k \calX_{\leq e} + 
(b-1).$$
\end{prop}

\begin{rem}
La proposition signifie exactement que, dans le théorème 
\ref{theo:presburger}, on peut choisir $N = b^2-1$ et que l'égalité 
énoncée vaut alors pour tout $e$.
\end{rem}

\begin{proof}
On considère un couple $(\mu_1, \mu_2)$ pour lequel les variétés
$\calX_{\leq e}$ et $\calX_{(\mu_1, \mu_2)}$ ont même dimension. Alors
par la formule \eqref{eq:dimmu12}, on a $\dim_k \calX_{(\mu_1 + b^2-1, 
\mu_2)} = \dim_k \calX_{(\mu_1, \mu_2)} + (b-1)$. Il en résulte 
l'inégalité 
$\dim_k \calX_{\leq e + (b^2-1)} \geq \dim_k \calX_{\leq e} +
(b-1)$.
On pose à présent $e' = e + (b^2-1)$ et on choisit $(\mu'_1, \mu'_2)$
tel que $\dim_k \calX_{\leq e'} = \calX_{(\mu'_1, \mu'_2)}$. Comme
$e' \geq b^2-1$, on a :
$$b - 1 = \dim_k \calX_{\leq b^2-1} \leq \dim_k \calX_{\leq e'} \leq
\frac{\mu'_1 - \mu'_2}{b+1}$$
d'où $\mu'_1 - (b^2-1) \geq \mu'_2$. La dimension de la variété
$\calX_{(\mu'_1 - (b^2-1), \mu'_2)}$ peut donc encore se calculer par la
formule \eqref{eq:dimmu12} et elle vaut $\dim_k \calX_{(\mu'_1, \mu'_2)}
- (b-1)$. À partir de là, on déduit $\dim_k \calX_{\leq e} \geq \dim_k
\calX_{\leq e'} - (b-1)$, et la proposition est démontrée.
\end{proof}

\paragraph{En dimension $\mathbf 3$}

On considère $(\mu_1, \mu_2, \mu_3)$ un triplet d'entiers tels que 
$\mu_1
\geq \mu_2 \geq \mu_3$ et $\mu_1 + \mu_2 + \mu_3 \equiv 0 \pmod {b-1}$.
De façon similaire à ce qui se passait en dimension $2$, calculer la
dimension de $\calX_{(\mu_1, \mu_2, \mu_3)}$ revient à maximiser la
quantité $\dim(\mu) = 2 \mu_1 + \mu_2 - \mu_{2,2} - \mu_{2,3} -
\mu_{3,3}$ sous les contraintes
$$\left\{ \begin{array}{l}
\mu_3 \leq \mu_{2,3} \leq \mu_2 \leq \mu_{2,2} \leq \mu_1
\quad ; \quad \mu_{2,3} \leq \mu_{3,3} \leq \mu_{2,2} \\
(b+1) \mu_{3,3} \geq b \mu_{2,2} + \mu_{2,3} \\
(b+1) \mu_{2,2} + 2\mu_{2,3} - \mu_{3,3} \geq b \mu_1 + \mu_2 + \mu_3 \\
2b \mu_{2,2} + (b+1)\mu_{2,3} - b \mu_{3,3} \geq b \mu_1 + b \mu_2 + \mu_3 \\
\mu_{2,2} + \mu_{2,3} \equiv \mu_{3,3} \equiv 0 \pmod {b-1}
\end{array} \right.$$
À $\mu_{2,2}$ et $\mu_{2,3}$ fixés, le meilleur $\mu_{3,3}$ (\emph{i.e.} 
le plus petit) est toujours $(b-1) \cdot \lceil 
\frac{b\mu_{2,2}+\mu_{2,3}}{b^2-1} \rceil$. On peut ainsi reformuler le 
problème en éliminant la variable $\mu_{3,3}$ ; celui-ci est équivalent
à maximiser la somme
$$2 \mu_1 + \mu_2 - \mu_{2,2} - \mu_{2,3} - (b-1) \cdot
\left\lceil\frac{b\mu_{2,2}+ \mu_{2,3}}{b^2-1} \right\rceil$$
sous les nouvelles contraintes
$$\left\{ \begin{array}{l}
\mu_3 \leq \mu_{2,3} \leq \mu_2 \leq \mu_{2,2} \leq \mu_1 \\
(b+1) \mu_{3,3} \geq b \mu_{2,2} + \mu_{2,3} \\
(b+1) \mu_{2,2} + 2\mu_{2,3} - (b-1) \cdot \left\lceil
\frac{b\mu_{2,2}+\mu_{2,3}}{b^2-1} \right\rceil
\geq b \mu_1 + \mu_2 + \mu_3 \\
2b \mu_{2,2} + (b+1)\mu_{2,3} - b (b-1) \cdot \left\lceil
\frac{b\mu_{2,2}+\mu_{2,3}}{b^2-1} \right\rceil
\geq b \mu_1 + b \mu_2 + \mu_3 \\
\mu_{2,2} + \mu_{2,3} \equiv 0 \pmod {b-1}
\end{array} \right.$$
L'écriture se simplifie encore si l'on effectue le changement de
variables $x = \frac{b \mu_{2,2} + \mu_{2,3}} {b-1}$ et
$y = \frac{\mu_{2,2} + \mu_{2,3}} {b-1} - \lceil \frac x
{b+1} \rceil$. En effet, on a alors 
$$\dim(\mu) = 2 \mu_1 + \mu_2 - 2 (b-1) \cdot \left\lceil \frac x
{b+1} \right\rceil - (b-1) y$$
tandis que les contraintes deviennent :
$$\left\{ \begin{array}{l}
x, y \in \Z \\
\displaystyle \mu_3 \leq 
-x + b \cdot \left\lceil \frac x {b+1} \right\rceil + by
\leq \mu_2 \leq 
x - \left\lceil \frac x {b+1} \right\rceil - y
\leq \mu_1 \smallskip \\
\displaystyle x + y \geq \frac{b \mu_1 + \mu_2 + \mu_3} {b-1} \quad ;
\quad x + b y \geq \frac{b \mu_1 + b \mu_2 + \mu_3} {b-1}
\end{array} \right.$$
On laisse momentanément de côté les inégalités compliquées de la seconde 
ligne pour se concentrer sur celles de la troisième.
\begin{figure}
\begin{center}
\includegraphics[scale=0.8]{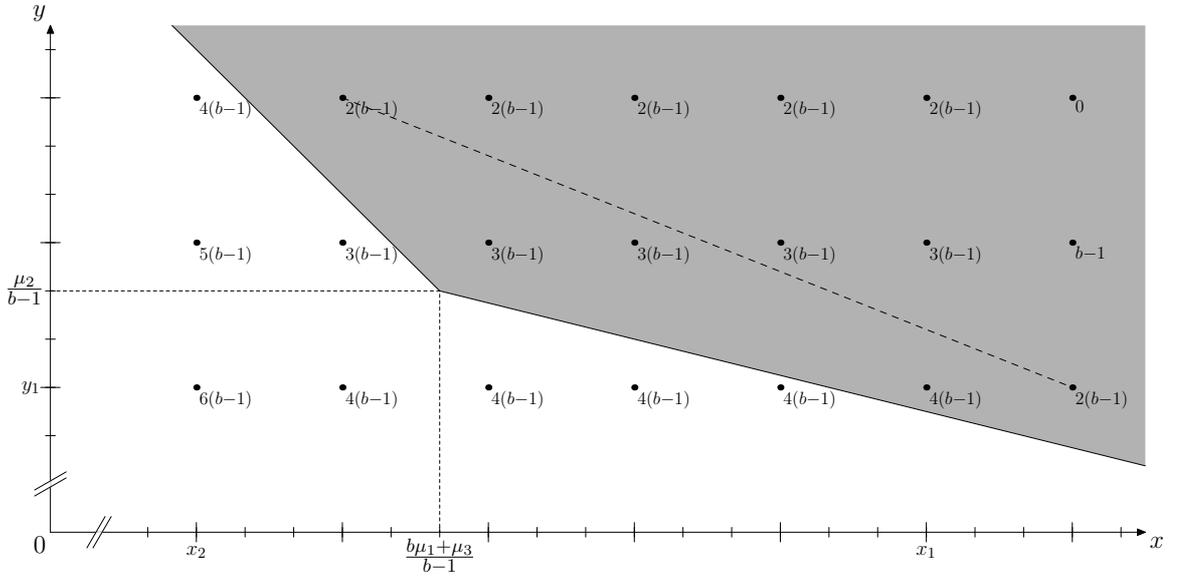}
\end{center}

\caption{Illustration du problème d'optimisation pour $b = 4$}
\label{fig:objectif}
\end{figure}
Sur la figure \ref{fig:objectif} est représentée la région définie par
celles-ci (en gris) et sont notées les
valeurs de la fonction à maximiser (à une constante additive près). En
étudiant cette figure --- et notamment en comparant les pentes des
droites définissant le domaine à celle de la droite oblique en
pointillés qui relie deux points de même valeur --- on démontre que le
maximum est nécessairement atteint au point de coordonnées $(x_1, y_1)$
avec
$$x_1 = (b+1) \cdot \left\lceil \frac{b\mu_1 + \mu_3}{b^2-1}\right\rceil
\quad \text{et} \quad
y_1 = \left\lceil \frac{\mu_2}{b-1} - \frac{b+1} b \cdot \df
\left(\frac{b \mu_1 + \mu_3}{b^2-1}\right) \right\rceil$$
ou au point de coordonnées $(x_2, y_2)$ avec
$$x_2 = x_1 - (b+1)
\quad \text{et} \quad
y_2 = b+1 + \left\lceil \frac{\mu_2}{b-1} - (b+1) \cdot \df
\left(\frac{b \mu_1 + \mu_3}{b^2-1}\right) \right\rceil.$$
Plus précisément, le maximum est atteint en $(x_1, y_1)$ si $\df (
\frac{b\mu_1 + \mu_3}{b^2-1} ) \leq \frac b {b+1}$ et en $(x_2,y_2)$ dans
le cas contraire. Dans la suite, on notera $(x_0, y_0)$ ce point.
Au sujet de la valeur du maximum, un calcul montre qu'il vaut :
\begin{equation}
\label{eq:dimmu123}
2\mu_1 + \mu_2 - 2 (b-1) \cdot \left\lceil \frac{b\mu_1 +
\mu_3}{b^2-1} \right\rceil - (b-1) \cdot \left\lceil \frac{\mu_2}{b-1} 
- (b+1) m \right\rceil
\end{equation}
avec 
$$m = \min \left\{ \frac 1 b \cdot \df \left( \frac{b \mu_1 +
\mu_3}{b^2-1} \right), \,\, \df \left( \frac{b \mu_1 + \mu_3}{b^2-1}
\right) - \frac{b-1}{b+1} \right\}.$$
On rappelle quand même que certaines contraintes avaient été mises de
côté. Il faut donc encore au moins se demander à quelles conditions 
les $x_0$ et $y_0$ précédents les satisfont. On remarque pour cela que 
$b+1$ divise $x_0$ de sorte que $\lceil \frac {x_0} {b+1} \rceil = 
\frac{x_0}{b+1}$ et que l'on a les encadrements suivants :
$$\frac{b \mu_1 + \mu_3}{b-1} -1 \leq x_0 \leq 
\frac{b \mu_1 + \mu_3}{b-1} + b 
\quad \text{et} \quad
\frac{\mu_2}{b-1} - 1 \leq y_0 \leq \frac{\mu_2}{b-1} + 1.$$
Ainsi il vient :
$$\frac{-b \mu_1 + b(b+1) \mu_2 - \mu_3}{b^2-1} - \frac{b^2+2b}{b+1} 
\leq -x_0 + b \left \lceil \frac {x_0} {b+1} \right\rceil + b y_0 \leq
\frac{-b \mu_1 + b(b+1) \mu_2 - \mu_3}{b^2-1} + \frac{b^2+b+1}{b+1}$$
$$\frac{b^2 \mu_1 - (b+1) \mu_2 + b \mu_3}{b^2-1} - \frac{2b+1}{b+1} 
\leq x_0 - \left \lceil \frac {x_0} {b+1} \right\rceil - y_0 \leq
\frac{b^2 \mu_1 - (b+1) \mu_2 + b \mu_3}{b^2-1} + \frac{b^2+b+1}{b+1}$$
à partir de quoi il suit que le couple $(x_0, y_0)$ est solution du
problème dès que le triplet $(\mu_1, \mu_2, \mu_3)$ vérifie $\mu_1
- \mu_2 \leq b(\mu_2 - \mu_3) - (b^2+b+1)(b-1)$ et $\mu_2 - \mu_3 \leq
b(\mu_1 - \mu_2) - (b^2+b+1)(b-1)$, ce qui revient encore à dire que le
triplet $(\mu_1-b^2-b-1, \mu_2, \mu_3+b^2+b+1)$ est $b$-régulier. On a 
ainsi démontré la proposition suivante.

\begin{prop}
On suppose $h \geq 0$. Soit $(\mu_1, \mu_2, \mu_3)$ un 
triplet tel que $(\mu_1-b^2-b-1, \mu_2, \mu_3+b^2+b+1)$ soit
intégralement $b$-régulier (voir définition \ref{def:breg} de
l'introduction). Alors la dimension de la variété $\calX_{(\mu_1,
\mu_2, \mu_3)}$ est donnée par la formule \eqref{eq:dimmu123}.
\end{prop}

\noindent
À partir de la formule \eqref{eq:dimmu123}, on voit que, dans le cas où
$(\mu_1-b^2-b-1, \mu_2, \mu_3+b^2+b+1)$ est intégralement $b$-régulier,
la dimension de $\calX_{(\mu_1, \mu_2, \mu_3)}$ est la somme du terme
attendu $\frac{2 (\mu_1 - \mu_3)}{b+1}$ et d'une quantité bornée qui ne
dépend que des congruences de $b \mu_1 + \mu_3$ modulo $b^2-1$ et
$\mu_2$ modulo $b-1$. 

\subsection{Généralisations envisageables}

\subsubsection{À un opérateur $\sigma : M \to M$ arbitraire}
\label{subsec:frobarbit}

Dans tout cet article, on a supposé que $\sigma$ agissait coordonnée par 
coordonnée sur $M$. Ceci est en fait assez restrictif, et une situation 
plus générale que l'on aimerait étudier (notamment car elle correspond à 
certains problèmes importants de déformation) est celle où on se donne 
une application $\sigma$-semi-linéaire quelconque $\sigma_M : M \to M$. 
Dans ce cas, les variétés $\calX_{\leq e}(\sigma_M)$, $\calX_\mu 
(\sigma_M)$ et $\calX_{\leq \mu}(\sigma_M)$ sont définies de
façon analogue. Par exemple, l'ensemble des $k$-points de $\calX_{\leq
e}(\sigma_M)$ est l'ensemble des réseaux $L$ de $M$ satisfaisant
$$u^e L \subset \sigma^\star_M(k[[u]] \otimes_{\sigma, k[[u]]} L) \subset
L$$
où $\sigma_M^\star : k[[u]] \otimes_{\sigma, k[[u]]} M \to M$ est
l'application linéarisée de $\sigma_M$. Si $A$ désigne la matrice de
$\sigma_M$ dans une $k((u))$-base de $M$ (par exemple la base canonique),
on s'autorisera à écrire $\calX_{\leq e}(A)$ à la place de $\calX_{\leq
e}(\sigma_M)$, et de même pour les deux autres variantes.
L'auteur pense que les théorèmes \ref{theo:dimcaruso},
\ref{theo:dimdivelem2} et \ref{theo:dimdivelem} s'étendent sans grande
modification à ce cas plus général.

\begin{conj}
\label{conj:anyfrob}
Il existe des constantes $b_0, c_1, \ldots, c_7$ et un vecteur
$\mu_0 \in \R^d$ tel que si $b \geq b_0$, alors
\begin{itemize}
\item pour tout entier $e$, on ait :
$$\dim_k \calX_{\leq e}(\sigma_M) \leq c_1 + \cro{\frac{d^2} 4} \cdot 
\frac e {b+1}$$
\item pour tout entier $e$ suffisamment grand, on ait :
$$\dim_k \calX_{\leq e}(\sigma_M) \geq -c_2 +  \cro{\frac{d^2} 4} \cdot 
\frac e {b+1}$$
\item pour tout $\mu = (\mu_1, \ldots, \mu_d) \in \R^d$ tel que $\mu_1
\geq \cdots \geq \mu_d$ et $\mu_1 + \cdots + \mu_d \equiv \val 
(\det \sigma_M) \pmod {b-1}$\footnote{Comme $\sigma_M$ est une 
application semi-linéaire, le déterminant de sa matrice peut varier 
lorsqu'on le calcule dans deux bases différentes ; toutefois la 
congruence modulo $b-1$ de sa valuation reste, elle, fixe. Il fait
donc bien sens d'écrire que $\val (\det \sigma_M)$ est congru à un 
certain entier modulo $b-1$. On notera également que dans le cas où
$\mu_1 + \cdots + \mu_d$ n'est pas congru à $\val (\det \sigma_M)$
modulo $b-1$, la variété $\calX_\mu$ est vide.}, on ait :
$$\dim_k \calX_\mu(\sigma_M) \leq c_3 + (b-1) \cdot \sum_{i=1}^d
\sum_{n=1}^\infty \mu_i \cdot \frac{d+1-i-\weyl^n(i)}{b^n}$$
\item pour tout $\mu$ comme précedemment tel qu'en outre $\mu_i \geq
\mu_{i+1} + c_4$ pour tout $i$, on ait si $d \geq 3$ :
$$\dim_k \calX_\mu(\sigma_M) \geq - c_5 + (b-1) \cdot \sum_{i=1}^d
\sum_{n=1}^\infty \mu_i \cdot \frac{d+1-i-\weyl^n(i)}{b^n}$$
\item pour tout $\mu = (\mu_1, \ldots, \mu_d) \in \R^d$ tel que $\mu_1
\geq \cdots \geq \mu_d$, on ait :
$$-c_6 + \sup_{ \substack{\mu' \leq \mu \\ \mu' - \mu_0\,b\text{\rm
-rég.}}} \frac{\left< 2 \vect \rho | \mu' \right>_d} {b+1} \leq \dim_k
\calX_{\leq \mu} (\sigma_M) \leq c_7 + \sup_{ \substack{\mu' \leq \mu \\
\mu'\,b\text{\rm -rég.}}} \frac{\left< 2 \vect \rho | \mu' \right>_d}
{b+1}.$$
\end{itemize}
\end{conj}

\noindent
La condition $d \geq 3$ peut paraître étrange, mais certains calculs 
explicites en dimension $2$ (voir par exemple \cite{hellmann}) montrent 
que, dans ce cas, pour certains $\sigma_M$ (précisément, ceux qui 
conduisent à des objets simples) des congruences supplémentaires sur les 
$\mu_i$ doivent être imposées afin que la variété $\calX_\mu(\sigma_M)$ 
résultante ne soit pas vide. Néanmoins, l'auteur pense --- et certains 
calculs numériques tendent à le confirmer --- qu'il s'agit là d'un 
phénomène lié à la petite dimension qui disparaît à partir de $d = 3$.

\subsubsection{À d'autres propriétés géométriques}

Pour l'instant, seule la dimension des variétés $\calX_{\leq e}$, 
$\calX_\mu$ et $\calX_{\leq \mu}$ a été regardée. Toutefois, d'autres 
propriétés géométriques revêtent également un intérêt certain. Il en est 
ainsi notamment du nombre de composantes connexes de ces variétés. À 
part pour le cas $d = 2$ qui peut être traité à la main par des méthodes 
\emph{ad hoc} (voir \cite{hellmann} et \cite{hellmann2}) et qui conduit 
déjà à des énoncés non triviaux, pratiquement rien n'est connu. De façon 
générale, étudier la géométrie fine des variétés précédentes paraît être 
une question très difficile. On peut néanmoins se demander dans quelle 
mesure les méthodes développées dans cet article sont susceptibles 
d'apporter une aide dans l'accomplissement de cette tâche. Si tout ce 
qui concerne l'optimisation linéaire semble lié exclusivement au calcul 
de la dimension, il est raisonnable de croire que la stratification par 
les variétés $\calX_\varphi$ définie au \S \ref{sec:stratification} ait 
encore un rôle à jouer pour d'autres questions, comme par exemple le 
calcul de la fonction zêta si le corps de base $k$ est fini ou de la 
caractéristique d'Euler-Poincaré.

En s'inspirant de la théorie de l'intégration motivique, on peut être
encore plus précis. Soit $K_0(\Var_k)$ le groupe abélien présenté de
la façon suivante :
\begin{itemize}
\item les générateurs sont les symboles $[X]$ où $X$ est un schéma
de type fini sur $k$ ;
\item les relations sont
$$\begin{array}{ll}
[X] = [X_\red] & \text{où } X_\red \text{ est le réduit de } X \\{}
[X] = [Y] & \text{si } X \text{ et } Y \text{ sont isomorphes} \\{}
[X] = [U] + [F] & \text{si } U \text{ est un ouvert de } X
\text{ et } F \text { est le fermé complémentaire.}
\end{array}$$
\end{itemize}
La formule $[X] \cdot [Y] = [X \times_k Y]$ définit un produit sur
$K_0(\Var_k)$ qui est fait un anneau commutatif. L'élément neutre pour
l'addition (resp. la multiplication) est le symbole de la variété vide
(resp. du point). De même que l'on a considéré dans le \S
\ref{subsec:dimexacte} les séries génératrices des dimensions de
$\calX_{\leq e}$, $\calX_\mu$ et $\calX_{\leq \mu}$, on peut définir ici
la série génératrice suivante :
$$S(X_1, \ldots, X_d) = \sum_{\mu \in \N^d} 
[\calX_\mu(\sigma_M)] \cdot X_1^{\mu_1} X_2^{\mu_2} \cdots X_d^{\mu_d}$$
où $\sigma_M$ est un certain opérateur $\sigma$-semi-linéaire agissant sur
$M$ et les $\mu_i$ sont les coordonnées de $\mu$. Bien sûr, on peut 
également considérer les séries génératrices associées aux variétés
$\calX_{\leq \mu}(\sigma_M)$ et $\calX_{\leq e}(\sigma_M)$ mais 
celles-ci se déduisent de la précédente à l'aide de manipulations 
algébriques élémentaires, et c'est pourquoi on se contente de celle de 
$\calX_\mu(\sigma_M)$.

Si l'on note $\L = [\A^1_k]$ le symbole de la droite affine, des 
résultats ou conjectures classiques en intégration motivique stipulent 
que les séries du type de $S$ sont en fait des fractions rationnelles 
lorsque leurs coefficients sont vus dans le localisé $K_0(\Var_k) 
[\L^{-1}]$ (ou parfois encore, un certain complété de cet anneau). 
L'auteur pense qu'il est raisonnable d'énoncer une conjecture similaire 
dans la situation de cet article.

\begin{conj}
\label{conj:motivique}
Il existe deux polynômes $P, Q \in K_0(\Var_k)[X_1, \ldots, X_d]$ tels
que $Q(0,\ldots, 0)$ soit inversible dans $K_0(\Var_k)[\L^{-1}]$ et
l'égalité
$$S(X_1, \ldots, X_d) = \frac{P(X_1, \ldots, X_d)}{Q(X_1, \ldots,
X_d)}$$
ait lieu dans l'anneau $K_0(\Var_k)[\L^{-1}][[X_1, \ldots, X_d]]$.
\end{conj}

\noindent
L'intérêt d'un tel énoncé est qu'il peut être spécialisé à un certain
nombres d'invariants géométriques ou arithmétiques plus classiques. 
Plus précisément dès que l'on dispose d'un morphisme $f$ de 
$K_0(\Var_k)$ dans un anneau $A$ qui envoie $\L$ sur un élément
inversible, sa véracité implique la rationalité de la série :
$$\sum_{\mu \in \N^d} f([\calX_\mu(\sigma_M)]) \cdot X_1^{\mu_1}
X_2^{\mu_2} \cdots X_d^{\mu_d}.$$
Or, il existe un certain nombre de tels morphismes $f$ intéressants. Si
$k$ est un corps fini, il y a par exemple celui qui a un symbole $[X]$
associe le cardinal de $X(k)$, ou plus généralement la fonction zêta de
$X$. Pour un corps $k$ quelconque, on dispose également d'exemples
construits par voie cohomologique comme les nombres de Betti ou le
polynôme de Poincaré virtuel. À partir de là, on peut retrouver la
dimension de $X$ ou encore son nombre de composantes irréductibles de
dimension maximale. La conjecture \ref{conj:motivique} admet donc pour
conséquence la rationalité de la série génératrice des dimensions (qui
a été démontrée directement dans cet article), mais implique 
également la rationalité d'autres séries génératrices numériques.

\medskip

Finalement, pour étudier d'autres propriétés géométriques qui ne 
proviennent pas de $K_0(\Var_k)$, il pourrait être intéressant de 
comprendre comment les variétés $\calX_\varphi$ s'agencent entre elles à 
l'intérieur de $\calX_{\leq e}$, $\calX_\mu$ ou $\calX_{\leq \mu}$. 
Notamment, une question qui paraît importante est de déterminer 
l'adhérence de $\calX_\varphi$ à l'intérieur de ces variétés. Par 
exemple, s'écrit-elle comme une union de certains $\calX_{\varphi'}$ où 
$\varphi'$ vérifie une condition qui s'exprime facilement en fonction de 
$\varphi$ ?

\subsubsection{À un groupe réductif connexe arbitraire}
\label{subsec:reductif}

À l'instar des variétés de Deligne-Lusztig, il est possible d'étendre la
définition des variétés $\calX_\mu$ et $\calX_{\leq \mu}$ à un groupe
réductif connexe déployé quelconque (le cas qui a été considéré dans
cet article 
étant celui de $\text{GL}_d$). Plus précisément, on considère un groupe
réductif connexe $G$ défini sur le corps $k$ (qui, pour simplifier, est
encore supposé algébriquement clos) et $T \subset G$ un tore maximal. 
Soit $X_\star(T)$ le groupe des caractères de $T$. On fixe une chambre
de Weyl dans $X_\star(T) \otimes \R$ dont l'adhérence est notée $C$. Si
$\lambda \in X_\star(T)$, on appelle $u^\lambda$ l'image de $u \in
\G_m(K)$ dans $T(K) \subset G(K)$, où $K = k((u))$. Si on pose $\O_K =
k[[u]]$, la décomposition de Cartan dit que $G(K)$ s'écrit comme l'union
disjointe des doubles classes $G(\O_K) u^\mu G(\O_K)$ où $\mu$ parcourt
l'ensemble des copoids dominants. On définit par ailleurs un opérateur
$\sigma$ agissant sur $G(K)$ de la façon suivante : on fixe un morphisme
de groupes algébriques $\sigma_0 : G \to G$ qui induit une bijection sur
les $k$-points (on rappelle que $k$ est supposé algébriquement clos) et
on pose $\sigma = G(u \mapsto u^b) \circ \sigma_0(K)$ où $\sigma_0(K)$
désigne l'application induite par $\sigma_0$ sur les $K$-points et où
$G(u \mapsto u^b)$ est l'application déduite par fonctorialité du
morphisme d'anneaux $K \to K$, $\sum_{i \gg -\infty} a_i u^i \mapsto
\sum_{i \gg -\infty} a_i u^{bi}$.
Si $\mu$ est un copoids dominant et si $A \in G(K)$, on peut alors
définir des variétés $\calX^G_\mu (A)$ dont les $k$-points sont :
$$\calX^G_\mu (A)(k) = \big\{ \, g \in G(K)/G(\O_K) \, | \, g^{-1}
A \sigma(g) \in G(\O_K) u^\mu G(\O_K) \, \big\}.$$
On définit également $\calX^G_{\leq \mu}(A)$ comme la réunion des
$\calX^G_{\mu'}(A)$ où $\mu'$ décrit l'ensemble des copoids dominants
tels que $\mu - \mu'$
s'écrive comme une combinaison linéaire à coefficients positifs des
racines simples correspondant au choix de $C$. Si
$G$ est le groupe linéaire $\text{GL}_d$, on retrouve les variétés
$\calX_\mu (A)$ et $\calX_{\leq \mu}(A)$. De façon générale, les
variétés $\calX^G_\mu(A)$ et $\calX^G_{\leq \mu}(A)$ sont toujours de 
dimension finie, et on peut s'interroger sur la valeur de cette 
dimension. 

Dans cette optique, une première question est de savoir si le théorème
\ref{theo:dimdivelem2} a des chances de se généraliser à cette nouvelle
situation, et le cas échéant sous quelle forme. Un premier coup d'\oe il
à l'expression
\begin{equation}
\label{eq:exprGLd}
(b-1) \cdot \min_{\weyl \in \mathfrak S_d} \left< \vect \rho_\weyl |
\mu \right>_d \qquad \text{où} \quad
\vec \rho_\weyl = 
\Bigg(\sum_{n=1}^\infty \frac{d+1-i-\weyl^n(i)}{b^n}\Bigg)_{1 \leq i 
\leq d} \in \R^d
\end{equation}
qui apparaît dans son énoncé (et qui constitue une première
approximation de la dimension de $\calX_\mu$) laisse bon espoir. En
effet, on voit d'emblée apparaître un minimum pris sur le groupe des
permutations de $\{1, \ldots, d\}$, c'est-à-dire exactement sur le
groupe de Weyl de $\text{GL}_d$. En outre si l'on fait agir ce groupe de
manière naturelle sur $\R^d$ --- c'est-à-dire par $\weyl \cdot (y_1, 
\ldots, y_d) = (y_{\weyl^{-1}(1)}, \ldots, y_{\weyl^{-1}(d)})$ --- le 
vecteur $\vec \rho_\weyl$ s'exprime en fonction de $\vec \rho = 
(\frac{d+1}2 - i)_{1 \leq i \leq d}$ comme suit :
$$\vec \rho_\weyl = (b-1) \cdot \sum_{n=1}^\infty \frac{\vec \rho + \weyl^{-n} 
\vec \rho}{b^n} = \vec \rho + (b-1) \cdot (b w - 1)^{-1}(\vec \rho)$$
au moins lorsque $b$ est assez grand pour que l'endomorphisme $bw-1$ de
$\R^d$ soit inversible. Si l'on se rappelle finalement que $\vec\rho$ est
égal à la demi-somme des racines positives du système de racines $A_d$,
on voit que la formule \eqref{eq:exprGLd} s'exprime uniquement en
termes du système de racines du groupe $\text{GL}_d$.
Ces considérations conduisent à la conjecture suivante.

\begin{conj}
\label{conj:reductif}
Soit $G$ un groupe réductif connexe sur $k$. Soit $T$ un tore maximal de
$G$. On note $W$ le groupe de Weyl associé et on fixe une fois pour 
toutes le choix d'une chambre de Weyl. Soient $\vec \rho$ la demi-somme 
des racines positives de $G$ et $A \in G(K)$.
Alors, il existe des constantes $b_0$ et $c_0$ telles que pour tout
$b \geq b_0$, on ait :
$$\dim_k \calX^G_\mu(A) \leq c_0 + \inf_{\weyl \in W} \left< \vect
\rho_\weyl | \mu \right> \quad \text{où} \quad
\vec \rho_\weyl = \vec \rho + (b-1) \cdot (b \weyl - 1)^{-1}(\vec \rho)$$
pour tout copoids dominant $\mu$. 
\end{conj}

\begin{rem}
Lorsque $\weyl$ est le mot le plus long $\weyl_0$ de $W$, le vecteur
$\rho_{\weyl_0}$ se calcule facilement. En effet, $\weyl_0$ échange les 
racines positives avec les racines négatives. En particulier, on a
$\weyl_0(\vec \rho) = - \vec \rho$, d'où il résulte que $(b w_0 - 1) 
(\vec \rho) = -(b+1) \vec \rho$, et par suite que $\vec \rho_{w_0} = 
\frac {2 \vec \rho} {b+1}$ ; on retrouve donc encore une fois ce vecteur 
particulier.
\end{rem}

\noindent
Il paraît aussi raisonnable de croire qu'une minoration de la dimension
de $\calX^G_\mu(A)$ par une expression du même type soit valable, au
moins lorsque $\mu$ vérifie une certaine condition d'intégrité et reste
suffisamment lors de la frontière de $C$. Malgré tout, gardant à
l'esprit le comportement singulier des variétés $\calX_\mu(\sigma_M)$
lorsque $\sigma_M$ représente un $\sigma$-module simple en dimension $2$,
nous préférons rester prudent et évasif à ce sujet.

La conjecture donne également une indication sur la façon d'étendre la 
définition de copoids $b$-réguliers : un copoids dominant $\mu$ est dit 
\emph{$b$-régulier} lorsque le minimum des produits scalaires $\left< 
\vect \rho_\weyl | \mu \right>$ ($\weyl \in W$) est atteint lorsque 
$\weyl$ est le mot le plus long $\weyl_0$, ce qui s'écrit en déroulant 
les définitions :
\begin{equation}
\label{eq:breggrp}
\left< \vect \rho + (b+1)\cdot (b\weyl -1)^{-1}(\vect \rho) | \mu 
\right> \geq 0
\end{equation}
pour tout $\weyl \in W$. Cette définition s'étend à tous les $\mu
\in X_\star(T) \otimes \R$.
On peut démontrer que, si $b$ est assez grand, un copoids $\mu$ est
$b$-régulier si, et seulement s'il vérifie les inégalités
\eqref{eq:breggrp} pour tout $\weyl \in W$ de longueur $\ell(\weyl_0) -
1$. Par ailleurs, pour $b$ suffisamment grand, la suite des copoids
$b$-réguliers est croissante en $b$ (\emph{i.e.} si $\mu$ est
$b$-régulier, alors il est $b'$-régulier pour tout $b' \geq b$) et tout
copoids $b$-régulier est dominant, dans le sens où il appartient à $C$.
Réciproquement, si $\mu$ est un élément de l'\emph{intérieur} de $C$
(c'est-à-dire un élément de la chambre de Weyl choisie), il est 
$b$-régulier pour $b$ suffisamment grand (le \og suffisamment \fg\
dépendant bien sûr de $\mu$).

\medskip

On a également une conjecture pour les variétés $\calX^G_{\leq
\mu}(A)$ :

\begin{conj}
Soit $G$ un groupe réductif connexe sur $k$. Soit $T$ un tore maximal de
$G$. On note $W$ le groupe de Weyl associé et on fixe une fois pour 
toutes le choix d'une chambre de Weyl. Soient $\vec \rho$ la demi-somme 
des racines positives de $G$ et $A \in G(K)$. Alors, il existe des 
constantes $c_1$ et $c_2$ et un élément $\mu_0 \in X_\star(T) \otimes 
\R$ tels que :
$$-c_2 + \sup_{\substack{\mu' \leq \mu \\ \mu' - \mu_0\,b\text{\rm
-rég.}}} \frac{\left< 2 \vect \rho | \mu' \right>_d} {b+1} \leq \dim_k
\calX^G_{\leq \mu}(A) \leq c_1 + \sup_{\substack{\mu' \leq \mu \\
\mu'\,b\text{\rm -rég.}}} \frac{\left< 2 \vect \rho | \mu' \right>_d}
{b+1}$$
où $\mu'$ désigne, ici, un copoids réel.
\end{conj}

Un cas particulièrement intéressant, qui apparaît déjà dans l'article de 
Kisin \cite{kisin}, est celui où l'on suppose le corps $k$ parfait (par 
exemple $k = \F_p$), où l'on se donne une extension finie $\ell$ de $k$ 
et où l'on considère le groupe $G$ défini comme la restriction des 
scalaires à la Weil de $\ell$ à $k$ de $\text{GL}_d$. Les variétés 
obtenues ont alors encore une interprétation arithmétique puisqu'elles 
apparaissent comme certaines espaces de modules de schémas en groupes 
définis sur des corps locaux. Lorsque $d = 2$, le calcul de leur 
dimension a déjà été accompli par Imai dans \cite{imai3} et, dans ce 
cas, les résultats qu'il obtient sont en accord avec les conjectures 
précédentes.

\end{document}